\RequirePackage{fix-cm}
\documentclass[smallextended]{svjour3}       
\smartqed  
\usepackage{lipsum}
\usepackage{amsfonts}
\usepackage{graphicx}
\usepackage{epstopdf}
\usepackage{algorithmic}

\usepackage{amsmath}
\usepackage{amssymb}
\usepackage{hyphenat}
\usepackage{hyperref}
\usepackage{times}
\usepackage{ifpdf}
\usepackage{array}
\usepackage{multirow}
\usepackage{setspace}
\usepackage{latexsym}
\usepackage{enumerate}
\usepackage{mathtools}
\usepackage{float}
\usepackage{subfig}
\usepackage{caption}
\usepackage{accents}
\usepackage{algorithm}
\usepackage{mathrsfs}
\usepackage{tabularx}
\usepackage{scalerel}
\usepackage{tikz}
\usepackage[numbers,sort&compress]{natbib}
\newcommand{\ignore}[1]{}
\usepackage[colorinlistoftodos,prependcaption,textsize=small]{todonotes}

\usepackage{amsopn}

\definecolor{Fcolor}{rgb}{0,0.5,0.25}

\newcommand{\sm}[2]{\scaleto{#1\mathstrut}{#2pt}}

\begin{document}

\title{Penalized Semidefinite Programming for Quadratically-Constrained Quadratic Optimization
\thanks{This work is in part supported by the NSF Award 1809454. Javad Lavaei is supported by an AFOSR YIP Award and ONR N000141712933.  Alper Atamt\"{u}rk is supported, in part, by grant FA9550-10-1-0168 from the Office of the Assistant Secretary of Defense for Research \& Engineering, NSF award 1807260, DOE ARPA-E award 260801540061, and ONR award 12951270.
}
}


\author{Ramtin Madani \and 
	Mohsen Kheirandishfard \and 
	Javad Lavaei \and 
        Alper Atamt\"{u}rk 
}


\institute{R. Madani \and M. Kheirandishfard\at
517 Nedderman Hall, 
University of Texas, 
Arlington, TX 76019, USA\\
\email{ramtin.madani@uta.edu,mohsen.kheirandishfard@mavs.uta.edu}
\and
J. Lavaei and A. Atamt\"{u}rk \at
4175 Etcheverry Hall, 
University of California,
Berkeley, CA 94720, USA\\
Tel.: +1-510-642-4559\\
\email{lavaei@berkeley.edu,atamturk@berkeley.edu}
}

\date{Received: date / Accepted: date}

\maketitle

\begin{abstract}
In this paper, we give a new penalized semidefinite programming approach for non-convex quadratically-constrained quadratic programs (QCQPs). We incorporate penalty terms into the objective of convex relaxations in order to retrieve feasible and near-optimal solutions for non-convex QCQPs. We introduce a generalized linear independence constraint qualification (GLICQ) criterion and prove that any GLICQ regular point that is sufficiently close to the feasible set can be used to construct an appropriate penalty term and recover a feasible solution. 
{\color{black}Inspired by these results, we develop a heuristic}
sequential procedure that preserves feasibility and aims to improve the objective value at each iteration. Numerical experiments on large-scale system identification problems as well as benchmark instances from the library of quadratic programming (QPLIB) demonstrate the ability of the proposed penalized semidefinite programs in finding near-optimal solutions for non-convex QCQP.
\keywords{Semidefinite programming \and non-convex optimization  \and non-linear programming \and convex relaxation}
\PACS{87.55.de}
\subclass{65K05 \and 90-08 \and 90C26 \and 90C22}
\end{abstract}

\section{Introduction}

{\color{black}
This paper studies a subclass of polynomial optimization, 
referred to as quadratically-constrained quadratic programming (QCQP), which
minimizes a quadratic function within a feasible set that is also characterized by quadratic functions. 
QCQP arises in various scientific and engineering applications, such as} electric power systems 
\cite{madani2016promises,madani2015convex,madani2015convexification,madani2014low}, 
imaging science 
\cite{bandeira2014tightness,candes2015phase,fogel2013phase,singer2011angular}, 
signal processing \cite{aittomaki2009beampattern,aubry2013ambiguity,chen2009mimo,li2012approximation,luo2010semidefinite,mariere2003blind}, automatic control
\cite{fazelnia2014convex,toker1998complexity,madani2014low,ahmadi2017dsos}, 
quantum mechanics 
\cite{hilling2010geometric,deza1994applications,laurent2015conic,burgdorf2015closure},
and cybersecurity
\cite{cid2005small,cid2004computational,courtois2002cryptanalysis,murphy2002essential}.
The development of efficient optimization techniques and numerical algorithms for QCQP has been an active area of research for decades. Due to the barriers imposed by NP-hardness, the focus of some research efforts has shifted from designing general-purpose algorithms to specialized methods that are robust and scalable for specific application domains. Notable examples for which methods with guaranteed performance have been offered in the literature include the problems of multisensor beamforming in communication theory \cite{gershman2010convex}, phase retrieval in signal processing \cite{candes2013phaselift}, and matrix completion in machine learning \cite{mu2016scalable,candes2009exact}.

{\color{black}
This paper advances a popular framework for the global analysis of QCQP through semidefinite programming (SDP) relaxations \cite{sherali1990hierarchy,nesterov1994interior,lasserre2001global,lasserre2006convergent,josz2015moment,chen2012globally,papp2013semidefinite,mohammad2017rounding}. A relaxation is said to be \textit{exact} if it has the same optimal objective value as the original problem. SDP has been critically important for constructing strong convex relaxations of non-convex optimization problems and its exactness has been verified for numerous real-world problems \cite{lasserre2001explicit,kim2003exact,sojoudi2013exactness1,sojoudi2013exactness2,burer2018exact}. 
Additionally, for many problems where an exact relaxation is not available, SDP relaxations have offered effective approximation algorithms \cite{Nesterov98,Ye991,Ye992,Zhang00,Zhang06,Luo07,He08,He10}. 
Geomans and Williamson \cite{goemans1995improved} show that an SDP relaxation objective is within 14\% of the optimal value for the MAXCUT problem on graphs with non-negative weights.
Additionally, SDP relaxations are used within branch-and-bound algorithms \cite{chen2017spatial,burer2008finite} for finding globally optimal solutions to non-convex optimization problems. 

In particular, forming hierarchies of SDP relaxations \cite{LS:SDP,sherali1990hierarchy,lasserre2001global,lasserre2006convergent,josz2015moment,chen2012globally,papp2013semidefinite,mohammad2017rounding} has been proven to yield the convex hull of non-convex QCQP problems. Despite solid theoretical guarantees, one of the primary challenges for the application of SDP hierarchies beyond small-scale instances is the rapid growth of dimensionality. In response, one direction of research has exploited sparsity and structural patterns to boost efficiency 
\cite{aa:rank-one,aa:2by2,muramatsu2003new,kim2003second,kim2003exact,bao2011semidefinite,natarajan2013penalized}. 
Another direction, pursued in \cite{alizadeh2003second,AN:cmir,majumdar2014control,permenter2014partial,madani2017scalable,ahmadi2017dsos,bienstock2018lp}, is to use lower-complexity relaxations as alternatives to computationally demanding SDP relaxations. In this paper, we offer an alternative approach, which focuses on penalizing the objective function of the SDP relaxations as opposed to strengthening the quality of the relaxations, which can be computationally prohibitive. We show that under certain conditions, incorporating a penalty term in the objective can remedy inexact relaxations and lead to feasible points for non-convex QCQPs.}

\subsection{Contributions}
This paper is concerned with non-convex quadratically-constrained quadratic programs for which {SDP relaxations} are inexact. In order to recover feasible points for QCQP, we incorporate a linear penalty term into the objective of {SDP relaxations} and show that feasible and near-globally optimal points can be obtained for the original QCQP by solving the resulting penalized {SDPs}. The penalty term is based on an arbitrary initial point. Our first result states that if the initial point is feasible and satisfies the linear independence constraint qualification (LICQ) condition, then penalized {SDP} produces a unique solution that is feasible for the original QCQP and its objective value is not worse than that of the initial point. 
Our second result states that if the initial point is infeasible, but instead is sufficiently close to the feasible set and satisfies a generalized LICQ condition, then the unique optimal solution to penalized {SDP} is feasible for QCQP. Lastly, motivated by these results on constructing feasible points, we propose a {heuristic} sequential procedure for non-convex QCQP and demonstrate its performance on benchmark instances from the QPLIB library \cite{FuriniEtAl2017TR} as well as on large-scale system identification problems.

The success of sequential frameworks and penalized {SDP} in solving bilinear matrix inequalities (BMIs) is demonstrated in \cite{ibaraki2001rank,BMI1,BMI2}. In \cite{ashraphijuo2016characterization}, it is shown that penalized {SDP} is able to find the roots of overdetermined systems of polynomial equations. Moreover, the incorporation of penalty terms into the objective of SDP relaxations are proven to be effective for solving non-convex optimization problems in power systems \cite{madani2015convex,madani2016promises,OPF_CDC,UC_CDC}. These papers show that penalizing certain physical quantities in power network optimization problems such as reactive power loss or thermal loss facilitates the recovery of feasible points from convex relaxations. 
In \cite{ibaraki2001rank}, a sequential framework is introduced for solving BMIs without theoretical guarantees. Papers \cite{BMI1,BMI2} investigate this approach further and offer theoretical results through the notion of generalized Mangasarian-Fromovitz regularity condition. However, these conditions are not valid in the presence of equality constraints and for general QCQPs. Motivated by the success of penalized {SDP}, 
this paper offers a theoretical framework for general QCQP and, by extension, polynomial optimization problems.


\subsection{Notations}
Throughout the paper, scalars, vectors, and matrices are respectively shown by italic letters, lower-case italic bold letters, and upper-case italic bold letters. The symbols $\mathbb{R}$, $\mathbb{R}^{n}$, and $\mathbb{R}^{n\times m}$ denote the sets of real scalars, real vectors of size $n$, and real matrices of size $n\times m$, respectively. The set of $n\times n$ real symmetric matrices is shown by $\mathbb{S}_n$. For a given vector $\boldsymbol{a}$ and a matrix $\boldsymbol{A}$, the symbols $a_i$ and $A_{ij}$ respectively indicate the $i^{th}$ element of $\boldsymbol{a}$ and the $(i,j)^{\mathrm{th}}$ element of $\boldsymbol{A}$.
The symbols $\langle\cdot\,,\cdot\rangle$ and $\|\cdot\|_{\mathrm{F}}$ denote the Frobenius inner product and norm of matrices, respectively. The notation $|\cdot|$ represents either the absolute value operator or cardinality of a set, depending on the context. The notation $\|\cdot\|_2$ denotes the $\ell_2$ norm of vectors, matrices, and matrix pencils. The $n\times n$ identity matrix is denoted by $\boldsymbol{I}_n$. The origin of $\mathbb{R}^n$ is denoted by $\boldsymbol{0}_n$. The superscript $(\cdot)^\top$ and the symbol $\mathrm{tr}\{\cdot\}$ represent the transpose and trace operators, respectively. 
Given a matrix $\boldsymbol{A}\in\mathbb{R}^{m\times n}$,
the notation $\sigma_{\min}(\boldsymbol{A})$ represents the minimum singular value of $\boldsymbol{A}$.
The notation $\boldsymbol{A}\succeq 0$ means that $\boldsymbol{A}$ is symmetric positive-semidefinite. For a pair of $n\times n$ symmetric matrices $(\boldsymbol{A},\boldsymbol{B})$ and proper cone $\mathcal{C}\subseteq\mathbb{S}_n$, the notation $\boldsymbol{A}\succeq_{\mathcal{C}}\boldsymbol{B}$ means that $\boldsymbol{A}-\boldsymbol{B}\in\mathcal{C}$, whereas $\boldsymbol{A}\succ_{\mathcal{C}}\boldsymbol{B}$ means that $\boldsymbol{A}-\boldsymbol{B}$ belongs to the interior of $\mathcal{C}$.
Given an integer $r>1$, define $\mathcal{C}_r$ as the cone of $n\times n$ symmetric matrices whose $r\times r$ principal submatrices are all positive semidefinite. Similarly, define $\mathcal{C}^{\ast}_r$ as the dual cone of $\mathcal{C}_r$, i.e., the cone of $n\times n$ symmetric matrices { whose every $r\times r$ principal submatrix is positive semidefinite (i.e., factor-width bounded by $r$)}.
Given a matrix $\boldsymbol{A}\in\mathbb{R}^{m\times n}$ and two sets of positive integers $\mathcal{S}_1$ and $\mathcal{S}_2$, define $\boldsymbol{A}\{\mathcal{S}_1, \mathcal{S}_2\}$ as the submatrix of $\boldsymbol{A}$ obtained by removing all rows of $\boldsymbol{A}$ whose indices do not belong to $\mathcal{S}_1$, and all columns of $\boldsymbol{A}$  whose indices do not belong to $\mathcal{S}_2$. Moreover, define $\boldsymbol{A}\{\mathcal{S}_1\}$ as the submatrix of $\boldsymbol{A}$ obtained by removing all rows of $\boldsymbol{A}$ that do not belong to $\mathcal{S}_1$.
Given a vector $\boldsymbol{a}\in\mathbb{R}^{n}$ and a set $\mathcal{F}\subseteq\mathbb{R}^{n}$, define $d_{\mathcal{F}}(\boldsymbol{a})$ as the minimum distance between $\boldsymbol{a}$ and members of $\mathcal{F}$. Given a pair of integers $(n,r)$, the binomial coefficient ``$n$ choose $r$'' is denoted by $n\choose r$. The notations $\nabla_{\boldsymbol{x}}f(\boldsymbol{a})$ and $\nabla^2_{\boldsymbol{x}}f(\boldsymbol{a})$, respectively, represent the gradient and Hessian of the function $f$, with respect to the vector $\boldsymbol{x}$, at a point $\boldsymbol{a}$. 

\subsection{Outline}
The remainder of the paper is organized as follows. In Section~\ref{sec:prelim}, we review the basic
lifted and reformulation linearization technique (RLT) as well as the standard {SDP} relaxations. Section~\ref{sec:penalty} presents the main results of the paper: the penalized {SDP}, its theoretical analysis on producing a feasible solution along with a generalized linear independence constraint qualification, and finally the sequential penalization procedure. 
In Section~\ref{sec:exp} we present numerical experiments to test the effectiveness of the sequential penalization approach for non-convex QCQPs from the library of quadratic programming instances (QPLIB)  as well as  large-scale system identification problems. Finally, we conclude in section~\ref{sec:conc} with a few final remarks.

\section{Preliminaries}
\label{sec:prelim}

In this section, we review the lifting and reformulation-linearization technique (RLT) as well as the standard convex relaxations of QCQP that are necessary for the development of the main results on penalized {SDP} in Section~\ref{sec:penalty}.
Consider a general quadratically-constrained quadratic program (QCQP):
\begin{subequations}\begin{align}
	\underset{\boldsymbol{x}\in\mathbb{R}^n}{\text{minimize}} \ \ 
	&  q_0(\boldsymbol{x})   \label{prob_obj}\\
	\text{s.t.} \ \ \ \ 
	&  q_k(\boldsymbol{x})\leq 0, \ \  k\in\mathcal{I}\label{prob_ineq}\\	 
	& q_k(\boldsymbol{x})=0, \ \  k\in\mathcal{E},\label{prob_eq} 
	\end{align}\end{subequations}	
where $\mathcal{I}$ and $\mathcal{E}$ index the sets of inequality and equality constraints, respectively. For every $k\in\{0\}\cup\mathcal{I}\cup\mathcal{E}$, $q_k:\mathbb{R}^n\to\mathbb{R}$ is a quadratic function of the form $q_k(\boldsymbol{x})\triangleq\boldsymbol{x}^{\top}\boldsymbol{A}_k\boldsymbol{x}+2\boldsymbol{b}^{\top}_k\boldsymbol{x}+c_k$, where $\boldsymbol{A}_k\in\mathbb{S}_n$, $\boldsymbol{b}_k\in\mathbb{R}^n$, and $c_k\in\mathbb{R}$. Denote $\mathcal{F}$ as the feasible set of the QCQP \eqref{prob_obj}--\eqref{prob_eq}.
To derive the optimality conditions for a given point, it is useful to define the Jacobian matrix of the constraint functions.
\begin{definition}[Jacobian Matrix]\label{Jacobian}
	For every $\hat{\boldsymbol{x}}\in\mathbb{R}^n$, the Jacobian matrix $\mathcal{J}(\hat{\boldsymbol{x}})$ 
	for the constraint functions $\{q_k\}_{k\in\mathcal{I}\cup\mathcal{E}}$ is
		\begin{align}
		\mathcal{J}(\hat{\boldsymbol{x}})&\triangleq
		[\nabla_{\boldsymbol{x}} q_1(\hat{\boldsymbol{x}}),\ldots,\nabla_{\boldsymbol{x}} q_{|\mathcal{I}\cup\mathcal{E}|}(\hat{\boldsymbol{x}})]^\top.
		\end{align}
	For every $\mathcal{Q}\subseteq\mathcal{I}\cup\mathcal{E}$, define $\mathcal{J}_{\mathcal{Q}}(\hat{\boldsymbol{x}})$ as the submatrix of $\mathcal{J}(\hat{\boldsymbol{x}})$ resulting from the rows that belong to $\mathcal{Q}$.
\end{definition}

Given a feasible point for the QCQP \eqref{prob_obj}--\eqref{prob_eq}, the well-known linear independence constraint qualification (LICQ) condition can be used as a regularity criterion.

\begin{definition}[LICQ Condition]\label{LICQ}
	A feasible point $\hat{\boldsymbol{x}}\in\mathcal{F}$ is LICQ regular if the rows of
	$\mathcal{J}_{\hat{\mathcal{B}}}(\hat{\boldsymbol{x}})$ are linearly independent, where $\hat{\mathcal{B}}\triangleq\{k\in\mathcal{I}\cup\mathcal{E}\;|\;q_k(\hat{\boldsymbol{x}})=0\}$ denotes the set of binding constraints at $\hat{\boldsymbol{x}}$.
\end{definition}

Finding a feasible point for the QCQP \eqref{prob_obj}--\eqref{prob_eq}, however, is NP-hard as the Boolean Satisfiability Problem (SAT) is a special case. Therefore, in Section~\ref{sec:penalty}, we introduce the notion of generalized LICQ as a regularity condition for both feasible and infeasible points.

\subsection{Convex relaxation}
{ A common approach for tackling the non-convex QCQP \eqref{prob_obj}--\eqref{prob_eq} introduces an auxiliary variable $\boldsymbol{X}\in\mathbb{S}_n$ accounting for $\boldsymbol{x}\boldsymbol{x}^{\top}$. Then, the objective function \eqref{prob_obj} and constraints \eqref{prob_ineq}--\eqref{prob_eq} can be written as linear functions of $\boldsymbol{x}$ and $\boldsymbol{X}$. For every $k\in\{0\}\cup\mathcal{I}\cup\mathcal{E}$, define $\bar{q}_k:\mathbb{R}^n\times\mathbb{S}_n\to\mathbb{R}$ as
\begin{align}
\bar{q}_k(\boldsymbol{x},\boldsymbol{X})\triangleq\langle\boldsymbol{A}_k,\boldsymbol{X}\rangle+2\boldsymbol{b}^{\top}_k\boldsymbol{x}+c_k.
\end{align}
Consider the following relaxation of QCQP \eqref{prob_obj}--\eqref{prob_eq}:
\begin{subequations}\begin{align}
	\underset{\begin{subarray}{l} \boldsymbol{x}\in\mathbb{R}^n, 
		\boldsymbol{X}\in\mathbb{S}_n\end{subarray}}{\text{minimize}} \ \ &
	\bar{q}_0(\boldsymbol{x},\boldsymbol{X})  \label{prob_lifted_obj} \\
	\ \ \ \ \ \ \text{s.t.} \ \ \ \ 
	&  \bar{q}_k(\boldsymbol{x},\boldsymbol{X}) \leq 0, \qquad\qquad k\in\mathcal{I}\label{prob_lifted_ineq} \\ 
	&  \bar{q}_k(\boldsymbol{x},\boldsymbol{X}) = 0,  \qquad\qquad k\in\mathcal{E}\label{prob_lifted_eq} \\
	&  \boldsymbol{X}-\boldsymbol{x}\boldsymbol{x}^\top\succeq_{\mathcal{C}_r}0  \label{prob_lifted_conic}
	\end{align}\end{subequations}
where 
the additional conic constraint \eqref{prob_lifted_conic} is a convex relaxation of the equation $\boldsymbol{X}=\boldsymbol{x}\boldsymbol{x}^\top$.  
We refer to the convex problem \eqref{prob_lifted_obj}--\eqref{prob_lifted_conic} as the $r\times r$ {SDP} relaxation of the QCQP \eqref{prob_obj}--\eqref{prob_eq}.} 
The choice $r=n$ { yields the well-known semidefinite programming (SDP) relaxation.
Additionally, in the homogeneous case (i.e., if 
$\boldsymbol{b}_0=\boldsymbol{b}_1=\dots=\boldsymbol{b}_{|\mathcal{I}\cup\mathcal{E}|}=\boldsymbol{0}$),
the case $r=2$ leads to the second-order conic programming (SOCP) relaxation.

In the presence of affine constraints, the RLT method of Sherali and Adams \cite{sherali2013reformulation} can be used to produce additional inequalities with respect to $\boldsymbol{x}$ and $\boldsymbol{X}$ to strengthen convex relaxations. This is covered in Appendix \ref{apn:RLT}.
} 

If the relaxed problem \eqref{prob_lifted_obj}--\eqref{prob_lifted_conic} has an optimal solution $(\accentset{\ast}{\boldsymbol{x}},\accentset{\ast}{\boldsymbol{X}})$ that satisfies $\accentset{\ast}{\boldsymbol{X}}=\accentset{\ast}{\boldsymbol{x}}\accentset{\ast}{\boldsymbol{x}}^{\top}$, then the relaxation is said to be \textit{exact} and $\accentset{\ast}{\boldsymbol{x}}$ is a globally optimal solution for the QCQP \eqref{prob_obj}--\eqref{prob_eq}. The next section offers a penalization method for addressing the case where relaxations are not exact.

\section{Penalized {SDP}}
\label{sec:penalty}

If the relaxed problem \eqref{prob_lifted_obj}--\eqref{prob_lifted_conic} is not exact, the resulting solution is not necessarily feasible for the original QCQP \eqref{prob_obj}--\eqref{prob_eq}. In this case, we use an initial point $\hat{\boldsymbol{x}}\in\mathbb{R}^n$ (either feasible or infeasible) to revise the objective function, resulting in a \textit{penalized {SDP}} of the form: {
\begin{subequations}\begin{align}
	\underset{\begin{subarray}{l} \boldsymbol{x}\in\mathbb{R}^n, 
		\boldsymbol{X}\in\mathbb{S}_n\end{subarray}}{\text{minimize}} \ \ 
	& \bar{q}_0(\boldsymbol{x},\boldsymbol{X}) + \eta\times(\mathrm{tr}\{\boldsymbol{X}\} - 2\hat{\boldsymbol{x}}^{\!\top}\boldsymbol{x} + \hat{\boldsymbol{x}}^{\!\top}\hat{\boldsymbol{x}} )   \label{prob_lifted_pen_obj} \\
	\ \ \ \ \ \ \text{s.t.} \ \ \ 
	&  \bar{q}_k(\boldsymbol{x},\boldsymbol{X}) \leq 0, \qquad \qquad k\in\mathcal{I}\label{prob_lifted_pen_ineq} \\
	& 
	\bar{q}_k(\boldsymbol{x},\boldsymbol{X}) = 0,  \qquad \qquad k\in\mathcal{E}\label{prob_lifted_pen_eq} \\	 
	& \boldsymbol{X}-\boldsymbol{x}\boldsymbol{x}^\top\succeq_{\mathcal{C}_r}0  \label{prob_lifted_pen_conic}
	\end{align}\end{subequations}}
where $\eta>0$ is a fixed penalty parameter. 
Note that the penalty term $\mathrm{tr}\{\boldsymbol{X}\} - 2\hat{\boldsymbol{x}}^{\!\top}\boldsymbol{x} + \hat{\boldsymbol{x}}^{\!\top}\hat{\boldsymbol{x}}$ equals zero for $\boldsymbol{X} =  \hat{\boldsymbol{x}}\hat{\boldsymbol{x}}^{\!\top}$.
The penalization is said to be {\it tight} if problem \eqref{prob_lifted_pen_obj}--\eqref{prob_lifted_pen_conic} has a unique optimal solution $(\accentset{\ast}{\boldsymbol{x}},\accentset{\ast}{\boldsymbol{X}})$ that satisfies $\accentset{\ast}{\boldsymbol{X}}=\accentset{\ast}{\boldsymbol{x}}\accentset{\ast}{\boldsymbol{x}}^{\top}$. In the next section, we give sufficient conditions under which penalized SDP is tight. 

\subsection{Theoretical analysis}

The following theorem guarantees that if $\hat{\boldsymbol{x}}$ is feasible and satisfies the LICQ regularity condition (in Section~\ref{sec:prelim}), then the solution of \eqref{prob_lifted_pen_obj}--\eqref{prob_lifted_pen_conic} is guaranteed to be feasible for the QCQP \eqref{prob_obj}--\eqref{prob_eq} for an appropriate choice of $\eta$. 

\begin{theorem}\label{thm1}
	Let $\hat{\boldsymbol{x}}$ be a feasible point for the QCQP \eqref{prob_obj}--\eqref{prob_ineq} that satisfies the LICQ condition. For sufficiently large $\eta>0$, the {SDP} \eqref{prob_lifted_pen_obj}--\eqref{prob_lifted_pen_conic} has a unique optimal solution $(\accentset{\ast}{\boldsymbol{x}},\accentset{\ast}{\boldsymbol{X}})$ such that $\accentset{\ast}{\boldsymbol{X}}=\accentset{\ast}{\boldsymbol{x}}\accentset{\ast}{\boldsymbol{x}}^{\top}$. Moreover, $\accentset{\ast}{\boldsymbol{x}}$ is feasible for \eqref{prob_obj}--\eqref{prob_eq} and satisfies $q_0(\accentset{\ast}{\boldsymbol{x}})\leq q_0(\hat{\boldsymbol{x}})$.
\end{theorem}
\begin{proof}
The proof is given in Section \ref{sec:proofs}.
\end{proof}

If $\hat{\boldsymbol{x}}$ is not feasible, but satisfies a generalized LICQ regularity condition, introduced below, and is close enough to the feasible set $\mathcal{F}$, then the penalization is still tight for large enough $\eta> 0$. This result is described formally in Theorem~\ref{thm2}.
First, we define a distance measure from an arbitrary point in $\mathbb{R}^n$ to the feasible set of the problem.

\begin{definition}[{Distance Function}]\label{def_feas_dis}
	The {distance function} $d_{\mathcal{F}}:\mathbb{R}^n\to\mathbb{R}$ is defined as
	\begin{align}
	d_{\mathcal{F}}(\hat{\boldsymbol{x}})\triangleq\min\{\|\boldsymbol{x}-\hat{\boldsymbol{x}}\|_2\,|\,\boldsymbol{x}\in\mathcal{F}\} .\label{dist_formula}
	\end{align}
\end{definition}

\begin{definition}[Generalized LICQ Condition]\label{GLICQ}
	For every $\hat{\boldsymbol{x}}\in\mathbb{R}^n$, the set of {\it quasi-binding} constraints is defined as
	\begin{align}
	\!\!\!\hat{\mathcal{B}}&\triangleq\mathcal{E}\cup\bigg\{k\!\in\!\mathcal{I}\;\bigg|\;\!\!q_k(\hat{\boldsymbol{x}})\!+\!\|\nabla q_k(\hat{\boldsymbol{x}})\|_2d_{\mathcal{F}}(\hat{\boldsymbol{x}})\! + \!
	\frac{\|\nabla^2 q_k(\hat{\boldsymbol{x}})\|_2}{2}\;d_{\mathcal{F}}(\hat{\boldsymbol{x}})^2\geq 0\bigg\} \cdot\!\!
	\end{align}
	The point $\hat{\boldsymbol{x}}$ is said to satisfy the GLICQ condition if the rows of
	$\mathcal{J}_{\hat{\mathcal{B}}}(\hat{\boldsymbol{x}})$ are linearly independent.
	Moreover, the sensitivity function $s:\mathbb{R}^n\to\mathbb{R}$ is defined as
	\begin{align}
	\!\!\!s(\hat{\boldsymbol{x}}) \!\triangleq\!
	\left\{
	\begin{array}{ll}
	\sigma_{\min}(\mathcal{J}_{\hat{\mathcal{B}}}(\hat{\boldsymbol{x}}))  &\quad \text{if $\hat{\boldsymbol{x}}$ satisfies GLICQ} \\
	0 &\quad \text{otherwise,} 
	\end{array}
	\right.\!\!\!
	\end{align}
	where $\sigma_{\min}(\mathcal{J}_{\hat{\mathcal{B}}}(\hat{\boldsymbol{x}}))$ denotes the smallest singular value of $\mathcal{J}_{\hat{\mathcal{B}}}(\hat{\boldsymbol{x}})$.
\end{definition}
Observe that if $\hat{\boldsymbol{x}}$ is feasible, then $d_{\mathcal{F}}(\hat{\boldsymbol{x}})=0$, and GLICQ condition reduces to the LICQ condition. Moreover, GLICQ is satisfied if and only if $s(\hat{\boldsymbol{x}})>0$.

The next definition introduces the notion of matrix pencil corresponding to the QCQP \eqref{prob_obj}--\eqref{prob_eq}, which will be used as a sensitivity measure.
\begin{definition}[Pencil Norm]\label{def_pen}
	For the QCQP \eqref{prob_obj}-\eqref{prob_eq}, define the corresponding matrix pencil $\boldsymbol{P}:\mathbb{R}^{|\mathcal{I}|}\times\mathbb{R}^{|\mathcal{E}|}\to\mathbb{S}_n$ as follows:
	\begin{align}
	\boldsymbol{P}(\boldsymbol{\gamma},\boldsymbol{\mu})\triangleq
	\sum_{k\in\mathcal{I}}{\gamma_k\boldsymbol{A}_k}+
	\sum_{k\in\mathcal{E}}{\mu_k\boldsymbol{A}_k}.
	\end{align}
	Moreover, define the pencil norm $\|\boldsymbol{P}\|_2$ as
	\begin{align}
	\|\boldsymbol{P}\|_2\triangleq\max\big\{\|\boldsymbol{P}(\boldsymbol{\gamma},\boldsymbol{\mu})\|_2\,\big|\,
	\|\boldsymbol{\gamma}\|^2_2+\|\boldsymbol{\mu}\|^2_2=1\big\},
	\end{align}
	which is upperbounded by $\sqrt{\sum_{k\in\mathcal{I}\cup\mathcal{E}}\|\boldsymbol{A}_k\|^2_2}$~.
\end{definition}

\begin{theorem}\label{thm2}
	Let $\hat{\boldsymbol{x}}\in\mathbb{R}^n$ satisfy the GLICQ condition for the QCQP \eqref{prob_obj}-\eqref{prob_ineq}, and assume that 
	\begin{align}
	d_{\mathcal{F}}(\hat{\boldsymbol{x}})
	< \left[1+ {n-1\choose r-1}\right]^{-1} \frac{s(\hat{\boldsymbol{x}})}{2\|\boldsymbol{P}\|_2}\;, \label{thm2ineq}
	\end{align}
	where ${n-1\choose r-1}$ denotes the binomial coefficient ``$n-1$ choose $r-1$'' and the distance function  $d_{\mathcal{F}}(\cdot)$, sensitivity function $s(\cdot)$ and pencil norm $\|\boldsymbol{P}\|_2$ are given by Definitions \ref{def_feas_dis}, \ref{GLICQ} and \ref{def_pen}, respectively.
	If $\eta$ is sufficiently large,  then the convex problem \eqref{prob_lifted_pen_obj}--\eqref{prob_lifted_pen_conic} has a unique optimal solution $(\accentset{\ast}{\boldsymbol{x}},\accentset{\ast}{\boldsymbol{X}})$ such that $\accentset{\ast}{\boldsymbol{X}}=\accentset{\ast}{\boldsymbol{x}}\accentset{\ast}{\boldsymbol{x}}^{\top}$ and $\accentset{\ast}{\boldsymbol{x}}$ is feasible for \eqref{prob_obj}--\eqref{prob_eq}.
\end{theorem}

\begin{proof}
The proof is given in Section \ref{sec:proofs}.
\end{proof}

{
The motivation behind Theorem \ref{thm2} is to show that even an infeasible initial point can produce feasible points.
It should be noted that, in general, it is computationally hard to calculate the exact distance from $\mathcal{F}$ and to verify GLICQ as a consequence. However, local search methods can be used in practice to find a local solution for \eqref{dist_formula}, resulting in upper bounds on feasibility distance. In Section \ref{sec:exp}, we use this simple technique to verify condition \eqref{thm2ineq} for several benchmark cases. {\color{black} Despite the theoretical insights offered by Theorems \ref{thm1} and \ref{thm2}, they do not provide practical bounds for $\eta$. Additionally, this section is primarily focused on offering a non-constructive proof for the existence of $\eta$ and we leave the derivation of analytical bounds for future work. In Section \ref{sec:exp}, we demonstrate that for real-world problems, appropriate choices of $\eta$ can be found via a simple bisection technique.
}
}

\subsection{Proof of theorems}\label{sec:proofs}
The rest of this section is devoted to proving Theorems~\ref{thm1}~and~\ref{thm2}.
To this end,
it is convenient to consider the following optimization problem:
\begin{subequations}\begin{align}
	\underset{\boldsymbol{x}\in\mathbb{R}^n}{\text{minimize}} \ \ 
	&  q_0(\boldsymbol{x})  + \eta\|\boldsymbol{x}-\hat{\boldsymbol{x}}\|^2_2 \label{prob_pen_obj}\\
	\ \ \ \  \text{s.t.} \ \ \ 
	&  q_k(\boldsymbol{x})\leq 0, \quad \quad  k\in\mathcal{I}\label{prob_pen_ineq}\\
	& 
	q_k(\boldsymbol{x})=0, \quad \quad  k\in\mathcal{E}. \label{prob_pen_eq}
	\end{align}\end{subequations}
{
Observe that the problem \eqref{prob_lifted_pen_obj} -- \eqref{prob_lifted_pen_conic} is a convex relaxation of \eqref{prob_pen_obj} -- \eqref{prob_pen_eq} and this is the motivation behind its introduction.}

Consider an $\alpha>0$ for which the inequality
\begin{align}
|q_0(\boldsymbol{x})|\leq\alpha\|\boldsymbol{x}-\hat{\boldsymbol{x}}\|_2^2+\alpha,\label{alpha_aprox}
\end{align}
is satisfied for every $\boldsymbol{x}\in\mathbb{R}^n$. If $\eta>\alpha$, then the objective function \eqref{prob_pen_obj} is lower bounded by $-\alpha$. {\color{black}Hence, if $\mathcal{F}$ is non-empty, then the optimal solution of \eqref{prob_pen_obj} -- \eqref{prob_pen_eq} is attainable, i.e., there exists $\accentset{\ast}{\boldsymbol{x}}\in\mathcal{F}$ which satisfies 
\begin{align}
q_0(\accentset{\ast}{\boldsymbol{x}}) \leq q_0(\boldsymbol{x}) \nonumber
\end{align}
for every $\boldsymbol{x}\in\mathcal{F}$.}
To prove the existence of $\alpha$, assume that
\begin{subequations}
	\begin{align}
	\alpha\geq\;&\sigma_{\max}
	\left(\begin{bmatrix}
	\boldsymbol{A}_0 & \;\;\boldsymbol{b}_0 \\
	\boldsymbol{b}_0^{\top} & 
	\;\;-\hat{\boldsymbol{x}}^{\!\top}\!\boldsymbol{A}_0\hat{\boldsymbol{x}}
	-2\boldsymbol{b}_0^{\!\top}\hat{\boldsymbol{x}}
	\end{bmatrix}\right)
	\sigma^{-1}_{\min}
	\left(\begin{bmatrix}
	\,\boldsymbol{I}_{\! n} & \;-\hat{\boldsymbol{x}} \\
	-\hat{\boldsymbol{x}}^{\!\top} & \;\frac{1}{2}\!+\!\hat{\boldsymbol{x}}^{\!\top}\!\hat{\boldsymbol{x}}
	\end{bmatrix}\right)
	\\
	\alpha\geq\;&2|\hat{\boldsymbol{x}}^{\!\top}\!\boldsymbol{A}_0\hat{\boldsymbol{x}}
	+2\boldsymbol{b}_0^{\!\top}\hat{\boldsymbol{x}}
	+c_0|
	\end{align}
\end{subequations}
then we have
\begin{subequations}
\begin{align}
|q_0(\boldsymbol{x})|&=
\Big\vert\!
\begin{bmatrix}
\boldsymbol{x}^{\!\top} & 1
\end{bmatrix}\!\!
\begin{bmatrix}
\boldsymbol{A}_0 & \;\;\boldsymbol{b}_0 \\
\boldsymbol{b}_0^{\top} & 
\;\;-\hat{\boldsymbol{x}}^{\!\top}\!\boldsymbol{A}_0\hat{\boldsymbol{x}}
-2\boldsymbol{b}_0^{\!\top}\hat{\boldsymbol{x}}
\end{bmatrix}\!\!
\begin{bmatrix}
\boldsymbol{x} \\ 1
\end{bmatrix}
+\hat{\boldsymbol{x}}^{\!\top}\!\boldsymbol{A}_0\hat{\boldsymbol{x}}
+2\boldsymbol{b}_0^{\!\top}\hat{\boldsymbol{x}}
+c_0
\Big\vert
\\
&\leq
\alpha\, 
\begin{bmatrix}
\boldsymbol{x}^{\!\top} & 1
\end{bmatrix}\!\!
\begin{bmatrix}
\,\boldsymbol{I}_{\! n} & \;-\hat{\boldsymbol{x}} \\
-\hat{\boldsymbol{x}}^{\!\top} & \;\frac{1}{2}\!+\!\hat{\boldsymbol{x}}^{\!\top}\!\hat{\boldsymbol{x}}
\end{bmatrix}\!\!
\begin{bmatrix}
\boldsymbol{x} \\ 1
\end{bmatrix}
+\frac{\alpha}{2}\\
&=\alpha\|\boldsymbol{x}-\hat{\boldsymbol{x}}\|_2^2+\alpha
\end{align}
\end{subequations}
which concludes \eqref{alpha_aprox}.

{
The next lemma shows that by increasing the penalty term $\eta$, the optimal solution $\accentset{\ast}{\boldsymbol{x}}$ can get as close to the initial point $\hat{\boldsymbol{x}}$ as $d_{\mathcal{F}}(\hat{\boldsymbol{x}})$. This lemma will later be used to show that $\accentset{\ast}{\boldsymbol{x}}$ can inherit the LICQ property from $\hat{\boldsymbol{x}}$.
}

\begin{lemma}\label{lemma1}
	Given an arbitrary $\hat{\boldsymbol{x}}\in\mathbb{R}^n$ and $\varepsilon>0$, for sufficiently large $\eta >0$, every optimal solution $\accentset{\ast}{\boldsymbol{x}}$ of the problem \eqref{prob_pen_obj}-\eqref{prob_pen_eq} satisfies
	\begin{align}
	0\leq\|\accentset{\ast}{\boldsymbol{x}}-\hat{\boldsymbol{x}}\|_2-d_{\mathcal{F}}(\hat{\boldsymbol{x}})\leq\varepsilon. \label{lemma1_ineq}
	\end{align}
	
\end{lemma}
\begin{proof}
	Consider an optimal solution $\accentset{\ast}{\boldsymbol{x}}$. Due to Definition \ref{def_feas_dis}, the distance between $\hat{\boldsymbol{x}}$ and every member of $\mathcal{F}$ is not less than $d_{\mathcal{F}}(\hat{\boldsymbol{x}})$, which concludes the left side of \eqref{lemma1_ineq}.
	Let $\boldsymbol{x}_d$ be an arbitrary member of the set $\{\boldsymbol{x}\in\mathcal{F}\;|\;\|\boldsymbol{x}-\hat{\boldsymbol{x}}\|_2=d_{\mathcal{F}}(\hat{\boldsymbol{x}})\}$. Due to the optimality of $\accentset{\ast}{\boldsymbol{x}}$, we have
	\begin{align}
	q_0(\accentset{\ast}{\boldsymbol{x}})+\eta\|\accentset{\ast}{\boldsymbol{x}}-\hat{\boldsymbol{x}}\|^2_2\leq
	q_0(\boldsymbol{x}_d)+\eta\|\boldsymbol{x}_d-\hat{\boldsymbol{x}}\|^2_2.\label{lll1}
	\end{align}
	According to the inequalities \eqref{lll1} and \eqref{alpha_aprox}, one can write
	\begin{subequations}
		\begin{align}
		&\hspace{-2mm}(\eta-\alpha)\|\accentset{\ast}{\boldsymbol{x}}-\hat{\boldsymbol{x}}\|^2_2-\alpha\leq
		(\eta+\alpha)\|\boldsymbol{x}_d-\hat{\boldsymbol{x}}\|^2_2+\alpha\quad\\
		&\hspace{-1mm}\Rightarrow\;\|\accentset{\ast}{\boldsymbol{x}}-\hat{\boldsymbol{x}}\|^2_2\leq
		\|\boldsymbol{x}_d-\hat{\boldsymbol{x}}\|^2_2+\frac{2\alpha}{\eta-\alpha}(1+\|\boldsymbol{x}_d-\hat{\boldsymbol{x}}\|^2_2)\hspace{-1mm}\\
		&\hspace{-1mm}\Rightarrow\;\|\accentset{\ast}{\boldsymbol{x}}-\hat{\boldsymbol{x}}\|^2_2\leq
		d_{\mathcal{F}}(\hat{\boldsymbol{x}})^2+\frac{2\alpha}{\eta-\alpha}(1+d_{\mathcal{F}}(\hat{\boldsymbol{x}})^2),\hspace{-1mm}
		\end{align}
	\end{subequations}
	which concludes the right side of \eqref{lemma1_ineq}, provided that $\eta\geq\alpha+2\alpha(1+d_{\mathcal{F}}(\hat{\boldsymbol{x}})^2)[\varepsilon^2+2\varepsilon d_{\mathcal{F}}(\hat{\boldsymbol{x}})]^{-1}$.
\end{proof}

\begin{lemma}\label{lemma2}
	Assume that $\hat{\boldsymbol{x}}\in\mathbb{R}^n$ satisfies the GLICQ condition for the problem \eqref{prob_pen_obj}--\eqref{prob_pen_eq}. Given an arbitrary $\varepsilon>0$, for sufficiently large $\eta>0$, 
	every optimal solution $\accentset{\ast}{\boldsymbol{x}}$ of the problem satisfies
	\begin{align}
	s(\hat{\boldsymbol{x}})-s(\accentset{\ast}{\boldsymbol{x}})\leq 2 d_{\mathcal{F}}(\hat{\boldsymbol{x}})\|\boldsymbol{P}\|_2+\varepsilon.\label{lemma2eq}
	\end{align}
	
\end{lemma}
\begin{proof}
	Let $\hat{\mathcal{B}}$ and $\accentset{\ast}{\mathcal{B}}$ denote the sets of quasi-binding constraints for $\hat{\boldsymbol{x}}$ and binding constraints for $\accentset{\ast}{\boldsymbol{x}}$, respectively (based on Definition \ref{GLICQ}). Due to Lemma \ref{lemma1}, for every $k\in\mathcal{I}\setminus\hat{\mathcal{B}}$ and every arbitrary $\varepsilon_1>0$, we have
	\begin{align}
	\!\!\!q_k(\accentset{\ast}{\boldsymbol{x}})\hspace{-4mm}&\hspace{4mm}\!-\!q_k(\hat{\boldsymbol{x}})\!= 
	2(\boldsymbol{A}_k\hat{\boldsymbol{x}}\!+\!\boldsymbol{b}_k)^{\!\top}(\accentset{\ast}{\boldsymbol{x}} - \hat{\boldsymbol{x}}) \!+\!(\accentset{\ast}{\boldsymbol{x}} - \hat{\boldsymbol{x}})^{\!\top}\!\boldsymbol{A}_k(\accentset{\ast}{\boldsymbol{x}} - \hat{\boldsymbol{x}})\nonumber\\
	&\hspace{-2mm}\leq 
	\|\nabla q_k(\hat{\boldsymbol{x}})\|_2\|\accentset{\ast}{\boldsymbol{x}} - \hat{\boldsymbol{x}}\|_2 + 
	\|\boldsymbol{A}_k\|_2\|\accentset{\ast}{\boldsymbol{x}} - \hat{\boldsymbol{x}}\|^2_2\nonumber\\
	&\hspace{-2mm}\leq 
	\|\nabla q_k(\hat{\boldsymbol{x}})\|_2d_{\mathcal{F}}(\hat{\boldsymbol{x}}) + 
	\|\boldsymbol{A}_k\|_2d_{\mathcal{F}}(\hat{\boldsymbol{x}})^2+\varepsilon_1\! <\!-q_k(\hat{\boldsymbol{x}}),
	\end{align}
	if $\eta$ is sufficiently large, which yields $\accentset{\ast}{\mathcal{B}}\subseteq\hat{\mathcal{B}}$. Let $\boldsymbol{\nu}\in\mathbb{R}^{|\hat{\mathcal{B}}|}$ be the left singular vector of $\mathcal{J}_{\hat{\mathcal{B}}}(\accentset{\ast}{\boldsymbol{x}})$, corresponding to the smallest singular value. Hence
	\begin{subequations}	
		\begin{align}
		\!\!\!\!s(\accentset{\ast}{\boldsymbol{x}}) 
		&={\sigma_{\min}}\{\mathcal{J}_{\accentset{\ast}{\mathcal{B}}}(\accentset{\ast}{\boldsymbol{x}})\}
		\geq{\sigma_{\min}}\{\mathcal{J}_{\hat{\mathcal{B}}}(\accentset{\ast}{\boldsymbol{x}})\}
		\!=\!\|\mathcal{J}_{\hat{\mathcal{B}}}(\accentset{\ast}{\boldsymbol{x}})^{\top}\boldsymbol{\nu}\|_2\\
		&\geq \|\mathcal{J}_{\hat{\mathcal{B}}}(\hat{\boldsymbol{x}})^{\top}\boldsymbol{\nu}\|_2
		-\|[\mathcal{J}_{\hat{\mathcal{B}}}(\hat{\boldsymbol{x}})-\mathcal{J}_{\hat{\mathcal{B}}}
		(\accentset{\ast}{\boldsymbol{x}})]^{\top}\boldsymbol{\nu}\|_2\\
		&\geq {\sigma_{\min}}\{\mathcal{J}_{\hat{\mathcal{B}}}(\hat{\boldsymbol{x}})\}\|\boldsymbol{\nu}\|_2
		-2\|\boldsymbol{P}\|_2\|\hat{\boldsymbol{x}}-\accentset{\ast}{\boldsymbol{x}}\|_2\|\boldsymbol{\nu}\|_2\\
		&\geq s(\hat{\boldsymbol{x}}) 
		-2\|\boldsymbol{P}\|_2\|\hat{\boldsymbol{x}}-\accentset{\ast}{\boldsymbol{x}}\|_2\\
		&\geq s(\hat{\boldsymbol{x}}) 
		-2 d_{\mathcal{F}}(\hat{\boldsymbol{x}})\|\boldsymbol{P}\|_2-\varepsilon,
		\end{align}
	\end{subequations}
	if $\eta$ is large, which concludes the inequality \eqref{lemma2eq}.
\end{proof}

{
In light of Lemma \ref{lemma2}, if $\hat{\boldsymbol{x}}$ is GLICQ regular and relatively close to $\mathcal{F}$, then $\accentset{\ast}{\boldsymbol{x}}$ is LICQ regular as well. This will be used next to prove the existence of Lagrange multipliers.}

\begin{lemma}\label{lemma3}
	Let $\accentset{\ast}{\boldsymbol{x}}$ be an optimal solution of the problem \eqref{prob_pen_obj}--\eqref{prob_pen_eq}, and assume that $\accentset{\ast}{\boldsymbol{x}}$ is LICQ regular. 
	There exists a pair of dual vectors $(\accentset{\ast}{\boldsymbol{\gamma}},\accentset{\ast}{\boldsymbol{\mu}})\in\mathbb{R}^{|\mathcal{I}|}_{+}\times\mathbb{R}^{|\mathcal{E}|}$ that satisfies the following Karush-Kuhn-Tucker (KKT) conditions:
	\begin{subequations}
		\begin{align}
		\!\!\!\!\!\!2(\eta\boldsymbol{I}\!\!+\!\boldsymbol{A}_0)(\accentset{\ast}{\boldsymbol{x}}\!-\!\hat{\boldsymbol{x}})
		\!+\!2(\boldsymbol{A}_0\hat{\boldsymbol{x}}+\boldsymbol{b}_0)\!+\!\mathcal{J}(\accentset{\ast}{\boldsymbol{x}})^{\!\top}
		[\accentset{\ast}{\boldsymbol{\gamma}}^{\!\top},\accentset{\ast}{\boldsymbol{\mu}}^{\!\top}]^{\!\top}\!\!\!=0,\!\!\!\! \label{KKT1}\\
		\accentset{\ast}{\gamma}_k q_k(\accentset{\ast}{\boldsymbol{\boldsymbol{x}}})=0,\qquad \forall k\in\mathcal{I}.\!\!\!\! \label{KKT2}
		\end{align}
	\end{subequations}
\end{lemma}
\begin{proof}	
	Due to the LICQ condition, there exists a pair of dual vectors $(\accentset{\ast}{\boldsymbol{\gamma}},\accentset{\ast}{\boldsymbol{\mu}})\in\mathbb{R}^{|\mathcal{I}|}_{+}\times\mathbb{R}^{|\mathcal{E}|}$, which satisfies the KKT stationarity and complementary slackness conditions. Due to stationarity, we have
	\begin{align}
	\!\!\!\!&\!\!\!0=\nabla_{\!\boldsymbol{x}}\,\mathcal{L}(\accentset{\ast}{\boldsymbol{x}},\accentset{\ast}{\boldsymbol{\gamma}},\accentset{\ast}{\boldsymbol{\mu}})/2\nonumber\\
	&=\eta(\accentset{\ast}{\boldsymbol{x}}-\hat{\boldsymbol{x}})\!+\!
	(\boldsymbol{A}_0\accentset{\ast}{\boldsymbol{x}}\!+\!\boldsymbol{b}_0)\!+\!
	\boldsymbol{P}(\accentset{\ast}{\boldsymbol{\gamma}},\accentset{\ast}{\boldsymbol{\mu}})\accentset{\ast}{\boldsymbol{x}}+\!
	\sum_{k\in\mathcal{I}}\accentset{\ast}{\gamma}_{k} \boldsymbol{b}_k\!+\!\!
	\sum_{k\in\mathcal{E}}\accentset{\ast}{\mu}_{k} \boldsymbol{b}_k\!\nonumber\\
	&=(\eta\boldsymbol{I}\!+\!\boldsymbol{A}_0)(\accentset{\ast}{\boldsymbol{x}}\!-\!\hat{\boldsymbol{x}})
	\!+\!(\boldsymbol{A}_0\hat{\boldsymbol{x}}\!+\!\boldsymbol{b}_0)\!+\!
	\mathcal{J}(\accentset{\ast}{\boldsymbol{x}})^{\!\top}[\accentset{\ast}{\boldsymbol{\gamma}}^{\!\top},\accentset{\ast}{\boldsymbol{\mu}}^{\!\top}]^{\!\top}/2.\!\!\!
	\end{align}
	Moreover, \eqref{KKT2} is concluded from the complementary slackness.
\end{proof}

{
The next lemma bounds the Lagrange multipliers whose existence is proven previously. This bound is helpful to prove that $\accentset{\ast}{\boldsymbol{X}}=\accentset{\ast}{\boldsymbol{x}}\accentset{\ast}{\boldsymbol{x}}^{\top}$. 
}

\begin{lemma}\label{lemma4}
	Consider an arbitrary $\varepsilon>0$ and suppose $\hat{\boldsymbol{x}}\in\mathbb{R}^n$ satisfies the inequality
	\begin{align}\label{lemma4ieq2}
	s(\hat{\boldsymbol{x}})>2 d_{\mathcal{F}}(\hat{\boldsymbol{x}})\|\boldsymbol{P}\|_2.
	\end{align}
	If $\eta$ is sufficiently large, for every optimal solution $\accentset{\ast}{\boldsymbol{x}}$ of the problem \eqref{prob_pen_obj}--\eqref{prob_pen_eq}, there exists a pair of dual vectors $(\accentset{\ast}{\boldsymbol{\gamma}},\accentset{\ast}{\boldsymbol{\mu}})\in\mathbb{R}^{|\mathcal{I}|}_{+}\times\mathbb{R}^{|\mathcal{E}|}$ that satisfies the inequality
	\begin{align}\label{lemma4ieq}
	\frac{1}{\eta}\sqrt{\|\accentset{\ast}{\boldsymbol{\gamma}}\|^2_2+\|\accentset{\ast}{\boldsymbol{\mu}}\|^2_2} \leq 
	\frac{2 d_{\mathcal{F}}(\hat{\boldsymbol{x}})}
	{ s(\hat{\boldsymbol{x}}) - 2 d_{\mathcal{F}}(\hat{\boldsymbol{x}})\|\boldsymbol{P}\|_2 }+\varepsilon
	\end{align}
	as well as the equations \eqref{KKT1} and \eqref{KKT2}.
\end{lemma}
\begin{proof}
	Due to Lemma \ref{lemma3}, there exists $(\accentset{\ast}{\boldsymbol{\gamma}},\accentset{\ast}{\boldsymbol{\mu}})\in\mathbb{R}^{|\mathcal{I}|}_{+}\times\mathbb{R}^{|\mathcal{E}|}$ that satisfies the equations \eqref{KKT1} and \eqref{KKT2}. Let
	$\boldsymbol{\tau}\triangleq
	[\accentset{\ast}{\boldsymbol{\gamma}}^{\top},\accentset{\ast}{\boldsymbol{\mu}}^{\top}]^{\top}$ and let $\accentset{\ast}{\mathcal{B}}$ be the set of binding constraints for $\accentset{\ast}{\boldsymbol{x}}$. Due to equations \eqref{KKT1} and \eqref{KKT2}, one can write
	\begin{align}
	\!\!\!\!2(\eta\boldsymbol{I}+\boldsymbol{A}_0)(\accentset{\ast}{\boldsymbol{x}}-\hat{\boldsymbol{x}})
	+2(\boldsymbol{A}_0\hat{\boldsymbol{x}}+\boldsymbol{b}_0)+\mathcal{J}_{\accentset{\ast}{\mathcal{B}}}(\accentset{\ast}{\boldsymbol{x}})^{\!\top}
	\boldsymbol{\tau}\{\accentset{\ast}{\mathcal{B}}\}&=0.\!\!\!
	\end{align}
	Let
	$\phi \triangleq s(\hat{\boldsymbol{x}}) - 2 d_{\mathcal{F}}(\hat{\boldsymbol{x}})\|\boldsymbol{P}\|_2$
	and define
	\begin{align}
	\varepsilon_1 \triangleq \phi\times \frac{\varepsilon-2\eta^{-1}{\phi}^{-1}(\|\boldsymbol{A}_0\hat{\boldsymbol{x}}\!+\!\boldsymbol{b}_0\|_2+d_{\mathcal{F}}(\hat{\boldsymbol{x}})\|\boldsymbol{A}_0\|_2)}
	{\varepsilon+2+2\eta^{-1}\|\boldsymbol{A}_0\|_2+2{\phi}^{-1}d_{\mathcal{F}}(\hat{\boldsymbol{x}})}\cdot \label{ep1}
	\end{align}
	If $\eta$ is sufficiently large, $\varepsilon_1$ is positive and 
	based on Lemmas \ref{lemma1} and \ref{lemma2}, we have
	\begin{align}
	&\frac{\|\boldsymbol{\tau}\|_2}{\eta}=
	\frac{\|\boldsymbol{\tau}\{\accentset{\ast}{\mathcal{B}}\}\|_2}{\eta}
	\leq
	\frac{2\|(\eta\boldsymbol{I}+\boldsymbol{A}_0)(\accentset{\ast}{\boldsymbol{x}}-\hat{\boldsymbol{x}})
		+(\boldsymbol{A}_0\hat{\boldsymbol{x}}+\boldsymbol{b}_0)\|_2}
	{\eta\sigma_{\min}\{\mathcal{J}_{\accentset{\ast}{\mathcal{B}}}(\accentset{\ast}{\boldsymbol{x}})\}}\nonumber\\
	&\leq
	\frac{2\eta\|\accentset{\ast}{\boldsymbol{x}}-\hat{\boldsymbol{x}}\|_2
		+2\|\boldsymbol{A}_0\|_2\|\accentset{\ast}{\boldsymbol{x}}-\hat{\boldsymbol{x}}\|_2
		+2\|\boldsymbol{A}_0\hat{\boldsymbol{x}}+\boldsymbol{b}_0\|_2}
	{\eta s(\accentset{\ast}{\boldsymbol{x}})}\nonumber\\
	&\leq
	\frac{2(d_{\mathcal{F}}(\hat{\boldsymbol{x}})\!+\!\varepsilon_1)
		\!+\!2\eta^{-1}[\|\boldsymbol{A}_0\|_2(d_{\mathcal{F}}(\hat{\boldsymbol{x}})\!+\!\varepsilon_1)
		\!+\!\|\boldsymbol{A}_0\hat{\boldsymbol{x}}\!+\!\boldsymbol{b}_0\|_2]}
	{s(\hat{\boldsymbol{x}}) - 2 d_{\mathcal{F}}(\hat{\boldsymbol{x}})\|\boldsymbol{P}\|_2-\varepsilon_1 }\nonumber\\
	&=
	\frac{2 d_{\mathcal{F}}(\hat{\boldsymbol{x}})}
	{ s(\hat{\boldsymbol{x}}) - 2 d_{\mathcal{F}}(\hat{\boldsymbol{x}})\|\boldsymbol{P}\|_2 }+\varepsilon,
	\end{align}
	where the last equality is a result of the equation \eqref{ep1}.
\end{proof}

{
The next two lemmas provide sufficient conditions for $\accentset{\ast}{\boldsymbol{X}}=\accentset{\ast}{\boldsymbol{x}}\accentset{\ast}{\boldsymbol{x}}^{\top}$
with respect to Lagrange multipliers that will be used later to prove Theorems \ref{thm1} and \ref{thm2}. 
}

\begin{lemma}\label{lemma5}
	Consider an optimal solution $\accentset{\ast}{\boldsymbol{x}}$ of the problem \eqref{prob_pen_obj}--\eqref{prob_pen_eq}, and a pair of dual vectors $(\accentset{\ast}{\boldsymbol{\gamma}},\accentset{\ast}{\boldsymbol{\mu}})\in\mathbb{R}^{|\mathcal{I}|}_{+}\times\mathbb{R}^{|\mathcal{E}|}$ that satisfies the conditions \eqref{KKT1} and \eqref{KKT2}. If the matrix inequality
	\begin{align}\label{dual_conic}
	\eta\boldsymbol{I}+\boldsymbol{A}_0+
	\boldsymbol{P}(\accentset{\ast}{\boldsymbol{\gamma}},\accentset{\ast}{\boldsymbol{\mu}})\succ_{\mathcal{D}_r} 0,
	\end{align}
	holds true, then the pair $(\accentset{\ast}{\boldsymbol{x}},\accentset{\ast}{\boldsymbol{x}}\accentset{\ast}{\boldsymbol{x}}^{\top})$ is the unique primal solution to the penalized {SDP} \eqref{prob_lifted_pen_obj}--\eqref{prob_lifted_pen_conic}.
\end{lemma}
\begin{proof}
	Let $\boldsymbol{\Lambda}\in\mathbb{S}^+_n$ denotes the dual variable associated with the conic constraint \eqref{prob_lifted_pen_conic}. Then, the KKT conditions for the problem \eqref{prob_lifted_pen_obj}-\eqref{prob_lifted_pen_conic} can be written as follows:
	\begin{subequations}
		\begin{align}
		\!\!\!\nabla_{\!\boldsymbol{x}\phantom{\boldsymbol{X}}\!\!\!\!}\,\bar{\mathcal{L}}(\boldsymbol{x},\boldsymbol{X},
		\boldsymbol{\gamma},\boldsymbol{\mu},\boldsymbol{\Lambda})&=
		2\left(\boldsymbol{\Lambda}\boldsymbol{\boldsymbol{x}}-\eta\hat{\boldsymbol{x}}+\boldsymbol{b}_0+\!
		\sum_{k\in\mathcal{I}}\accentset{\ast}{\gamma}_{k} \boldsymbol{b}_k\!+\!\!
		\sum_{k\in\mathcal{E}}\accentset{\ast}{\mu}_{k} \boldsymbol{b}_k\!\right)=0,\label{KKT_1}\\
		\!\!\!\nabla_{\!\boldsymbol{X}\phantom{\boldsymbol{x}}\!\!\!\!}\,\bar{\mathcal{L}}(\boldsymbol{x},\boldsymbol{X},
		\boldsymbol{\gamma},\boldsymbol{\mu},\boldsymbol{\Lambda})&=
		\eta\boldsymbol{I}+\boldsymbol{A}_0+\boldsymbol{P}({\boldsymbol{\gamma}},{\boldsymbol{\mu}})-\boldsymbol{\Lambda}=0,\label{KKT_2}\\
		\!\!\!\gamma_k q_k(\boldsymbol{\boldsymbol{x}})&=0, \qquad\forall k\in\mathcal{I}\label{KKT_3}\\
		\langle\boldsymbol{\Lambda},\,\boldsymbol{x}\boldsymbol{x}^{\top}\!-\!\boldsymbol{X}\rangle&=0, \label{KKT_4}
		\end{align}
	\end{subequations}
	where 
	$\bar{\mathcal{L}}:\mathbb{R}^n\times\mathbb{S}_n
	\times\mathbb{R}^{|\mathcal{I}|}\times\mathbb{R}^{|\mathcal{E}|}\times\mathbb{S}_n\to\mathbb{R}$ 
	is the Lagrangian function, equations \eqref{KKT_1} and \eqref{KKT_2} account for stationarity with respect to $\boldsymbol{x}$ and $\boldsymbol{X}$, respectively, and equations \eqref{KKT_3} and \eqref{KKT_4} are the complementary slackness conditions for the constraints \eqref{prob_lifted_pen_ineq} and \eqref{prob_lifted_pen_conic}, respectively. Define
	\begin{align}
	\accentset{\ast}{\boldsymbol{\Lambda}}\triangleq\eta\boldsymbol{I}+\boldsymbol{A}_0+
	\boldsymbol{P}(\accentset{\ast}{\boldsymbol{\gamma}},\accentset{\ast}{\boldsymbol{\mu}}).
	\end{align}
	Due to Lemma \ref{lemma3}, if $\eta$ is sufficiently large, $\accentset{\ast}{\boldsymbol{x}}$ and $(\accentset{\ast}{\boldsymbol{\gamma}},\accentset{\ast}{\boldsymbol{\mu}})$ satisfy the equations \eqref{KKT1} and \eqref{KKT2}, which yield the optimality conditions \eqref{KKT_1}-\eqref{KKT_4}, if
	$\boldsymbol{x}=\accentset{\ast}{\boldsymbol{x}}$,
	$\boldsymbol{X}=\accentset{\ast}{\boldsymbol{x}}\accentset{\ast}{\boldsymbol{x}}^{\top}$,
	$\boldsymbol{\gamma}=\accentset{\ast}{\boldsymbol{\gamma}}$,
	$\boldsymbol{\mu}=\accentset{\ast}{\boldsymbol{\mu}}$, and $\boldsymbol{\Lambda}=\accentset{\ast}{\boldsymbol{\Lambda}}$.
	Therefore, the pair $(\accentset{\ast}{\boldsymbol{x}},\accentset{\ast}{\boldsymbol{x}}\accentset{\ast}{\boldsymbol{x}}^{\top})$ is a primal optimal points for the penalized {SDP} \eqref{prob_lifted_pen_obj}-\eqref{prob_lifted_pen_conic}. {\color{black}Note that due to positive semidefiniteness of $\accentset{\ast}{\boldsymbol{\Lambda}}$, the condition \eqref{KKT_4} implies $\boldsymbol{x}\boldsymbol{x}^{\top}\!-\!\boldsymbol{X}=0$.}
	
	Since the KKT conditions hold for every pair of primal and dual solutions, we have
	\begin{align}
	\accentset{\ast}{\boldsymbol{\boldsymbol{x}}}=\accentset{\ast}{\boldsymbol{\Lambda}}^{-1}\left(\eta\hat{\boldsymbol{x}}-\boldsymbol{b}_0-\!
	\sum_{k\in\mathcal{I}}\accentset{\ast}{\gamma}_{k} \boldsymbol{b}_k\!-\!\!
	\sum_{k\in\mathcal{E}}\accentset{\ast}{\mu}_{k} \boldsymbol{b}_k\right)
	\end{align}
	and $\accentset{\ast}{\boldsymbol{X}}=\accentset{\ast}{\boldsymbol{x}}\accentset{\ast}{\boldsymbol{x}}^{\top}$, according to the equations \eqref{KKT_1} and \eqref{KKT_4}, respectively, which implies the uniqueness of the solution.
\end{proof}

\begin{lemma}\label{lemma6}
	Consider an optimal solution $\accentset{\ast}{\boldsymbol{x}}$ of the problem \eqref{prob_pen_obj}-\eqref{prob_pen_eq}, and a pair of dual vectors $(\accentset{\ast}{\boldsymbol{\gamma}},\accentset{\ast}{\boldsymbol{\mu}})\in\mathbb{R}^{|\mathcal{I}|}_{+}\times\mathbb{R}^{|\mathcal{E}|}$ that satisfies the conditions \eqref{KKT1} and \eqref{KKT2}. If the inequality,
	\begin{align}\label{lemma6ieq}
	\frac{1}{\eta}\sqrt{\|\accentset{\ast}{\boldsymbol{\gamma}}\|^2_2+\|\accentset{\ast}{\boldsymbol{\mu}}\|^2_2}
	<{n-1 \choose r-1}^{-1}\frac{1}{\|\boldsymbol{P}\|_2}-\frac{\|\boldsymbol{A}_0\|_2}{\eta\|\boldsymbol{P}\|_2}
	\end{align}
	holds true, then the pair $(\accentset{\ast}{\boldsymbol{x}},\accentset{\ast}{\boldsymbol{x}}\accentset{\ast}{\boldsymbol{x}}^{\top})$ is the unique primal solution to the penalized {SDP} \eqref{prob_lifted_pen_obj}--\eqref{prob_lifted_pen_conic}.
\end{lemma}
\begin{proof}
	Based on Lemma \ref{lemma5}, it suffices to prove the conic inequality \eqref{dual_conic}. Define
	\begin{align}
	\boldsymbol{K}\triangleq\boldsymbol{A}_0+
	\boldsymbol{P}(\accentset{\ast}{\boldsymbol{\gamma}},\accentset{\ast}{\boldsymbol{\mu}}).
	\end{align}
	It follows that
	\begin{subequations}\label{Kbound}
		\begin{align}
		\|\boldsymbol{K}\|_2&\leq\|\boldsymbol{A}_0\|_2+
		\sum_{k\in\mathcal{I}}{\accentset{\ast}{\gamma}_k\|\boldsymbol{A}_k\|_2}+
		\sum_{k\in\mathcal{E}}{\accentset{\ast}{\mu}_k\|\boldsymbol{A}_k\|_2},\\
		&\leq\|\boldsymbol{A}_0\|_2+
		\|\boldsymbol{P}\|_2\sqrt{\|\accentset{\ast}{\boldsymbol{\gamma}}\|^2_2+\|\accentset{\ast}{\boldsymbol{\mu}}\|^2_2}\;.
		\end{align}
	\end{subequations}
	Let $\mathcal{R}$ be the set of all $r$-member subsets of $\{1,2,\ldots,n\}$. Hence,
	\begin{align}
	\eta\boldsymbol{I}+\boldsymbol{K}=\sum_{\mathcal{K}\in\mathcal{R}}
	{\boldsymbol{I}\{\mathcal{K}\}^{\top}\,\boldsymbol{R}_{\mathcal{K}}\,\boldsymbol{I}\{\mathcal{K}\}},
	\end{align}
	where
	\begin{align}
	\boldsymbol{R}_{\mathcal{K}}&={{n-1}\choose{r-1}}^{-1}
	[\eta\boldsymbol{I}\{\mathcal{K},\mathcal{K}\}+\boldsymbol{K}\{\mathcal{K},\mathcal{K}\}].
	\end{align}
	Due to the inequalities \eqref{lemma6ieq} and \eqref{Kbound}, we have $\boldsymbol{R}_{\mathcal{K}}\succ 0$ for every $\mathcal{K}\in\mathcal{R}$, which proves that $\eta\boldsymbol{I}+\boldsymbol{K}\succ_{\mathcal{D}_r}0$.
\end{proof}

\begin{proof}[Theorem \ref{thm2}]
	Let $\accentset{\ast}{\boldsymbol{x}}$ be an optimal solution of the problem \eqref{prob_pen_obj}--\eqref{prob_pen_eq}. 
	According to the assumption \eqref{thm2ineq}, the inequality \eqref{lemma4ieq2} holds true, and due to Lemma \ref{lemma4}, if $\eta$ is sufficiently large, there exists a corresponding pair of dual vectors $(\accentset{\ast}{\boldsymbol{\gamma}},\accentset{\ast}{\boldsymbol{\mu}})$ that satisfies the inequality \eqref{lemma4ieq}. Now, according to the inequality \eqref{thm2ineq}, we have
	\begin{align}
	\frac{2 d_{\mathcal{F}}(\hat{\boldsymbol{x}})}
	{ s(\hat{\boldsymbol{x}}) - 2 d_{\mathcal{F}}(\hat{\boldsymbol{x}})\|\boldsymbol{P}\|_2 } \leq
	\frac{1}{{n-1\choose r-1}\|\boldsymbol{P}\|_2}
	\end{align}
	and therefore \eqref{lemma4ieq} concludes \eqref{lemma6ieq}. Hence, according to Lemma \ref{lemma6}, the pair $(\accentset{\ast}{\boldsymbol{x}},\accentset{\ast}{\boldsymbol{x}}\accentset{\ast}{\boldsymbol{x}}^{\top})$ is the unique primal solution to the penalized {SDP} \eqref{prob_lifted_pen_obj}--\eqref{prob_lifted_pen_conic}.
\end{proof}

\begin{proof}[Theorem \ref{thm1}]
	If $\hat{\boldsymbol{x}}$ is feasible, then $d_{\mathcal{F}}(\hat{\boldsymbol{x}})=0$. Therefore, the tightness of the penalization for Theorem \ref{thm1} is a direct consequence of Theorem \ref{thm2}. Denote the unique optimal solution of the penalized {SDP} as $(\accentset{\ast}{\boldsymbol{x}},\accentset{\ast}{\boldsymbol{x}}\accentset{\ast}{\boldsymbol{x}}^{\top})$. Then it is straightforward to verify the inequality $q_0(\accentset{\ast}{\boldsymbol{x}})\leq q_0(\hat{\boldsymbol{x}})$ by evaluating the objective function \eqref{prob_lifted_pen_obj} at the point $(\hat{\boldsymbol{x}},\hat{\boldsymbol{x}}\hat{\boldsymbol{x}}^{\top})$. 
\end{proof}

\subsection{Sequential penalization procedure}
In practice, the penalized {SDP} \eqref{prob_lifted_pen_obj}--\eqref{prob_lifted_pen_conic} 
can be initialized by a point that may not satisfy the conditions of Theorem~\ref{thm1} or Theorem~\ref{thm2} as these conditions are only sufficient, but not necessary. If the chosen initial point $\hat{\boldsymbol{x}}$ does not result in a tight penalization, the penalized SDP\eqref{prob_lifted_pen_obj}--\eqref{prob_lifted_pen_conic} can be solved sequentially by updating the initial point until a feasible and near-optimal point is obtained. This heuristic procedure is described in Algorithm \ref{al:alg_1}.

\begin{algorithm}
	\caption{Sequential Penalized Conic Relaxation.}\label{alg:1}
	\begin{algorithmic}
		\STATE{initiate $\{q_k\}_{k\in\{0\}\cup\mathcal{I}\cup\mathcal{E}}$, $r\geq2$, $\hat{\boldsymbol{x}}\in\mathbb{R}^n$, and the fixed parameter $\eta>0$}
		\WHILE{stopping criterion is not met}
		\STATE{solve the problem \eqref{prob_lifted_pen_obj}--\eqref{prob_lifted_pen_conic} with the initial point $\hat{\boldsymbol{x}}$ to obtain $(\accentset{\ast}{\boldsymbol{x}},\accentset{\ast}{\boldsymbol{X}})$}
		\STATE{set $\hat{\boldsymbol{x}} \leftarrow \accentset{\ast}{\boldsymbol{x}}$}
		\ENDWHILE
		\RETURN $\accentset{\ast}{\boldsymbol{x}}$
	\end{algorithmic}\label{al:alg_1}
\end{algorithm}

According to Theorem \ref{thm2}, once $\hat{\boldsymbol{x}}$ is close enough to the feasible set $\mathcal{F}$, the penalization becomes tight, i.e., a feasible solution $\accentset{\ast}{\boldsymbol{x}}$ is recovered as the unique optimal solution to \eqref{prob_lifted_pen_obj}--\eqref{prob_lifted_pen_conic}. Afterwards, in the subsequent iterations, according to Theorem \ref{thm1}, feasibility is preserved and the objective value does not increase. {\color{black} Note that Theorems \ref{thm1} and \ref{thm2} do not guarantee the existence of a global $\eta$ that works for every member of the sequence generated by Algorithm \ref{al:alg_1}. For this reason, we regard this procedure as a heuristic.}

{
	The following example illustrates the 
	application of Algorithm \ref{al:alg_1} for a polynomial optimization. 
	\begin{example}
		Consider the following three-dimensional polynomial optimization:
		\begin{subequations}\begin{align}
			\underset{a,b,c\in\mathbb{R}}{\mathrm{minimize}} \ \ 
			& a  \label{poly_exmp_obj}\\
			\ \ \ \ \mathrm{s.t.} \ \ \ 
			& a^5-b^4-c^4+2a^3+2a^2 b-2a b^2+ 6abc-2=0 \label{poly_exmp_eq}
			\end{align}\end{subequations}
		To derive a QCQP reformulation of the problem \eqref{poly_exmp_obj}--\eqref{poly_exmp_eq}, we consider a variable $x\in\mathbb{R}^8$, whose elements account for the monomials $a$, $b$, $c$, $a^2$, $b^2$, $c^2$, $ab$, and $a^3$, respectively. This leads to the following QCQP: 
		\begin{subequations}\begin{align}
			\underset{\begin{subarray}{l} \boldsymbol{x}\in\mathbb{R}^8,
				\end{subarray}}{\mathrm{minimize}} \ \ &  x_1\  \label{qcqp_exmp_obj}\\
			\ \ \  \ \mathrm{s.t.} \ \ \ 
			&  x_4x_8 - x_5^2 - x_6^2 + 2x_1x_4 + 2 x_2 x_4 - 2 x_1 x_5 + 6 x_3 x_7 -2 =0 \label{qcqp_exmp_eq}\\
			&  x_4-x^2_1=0 \label{qcqp_exmp_aux1}\\
			&  x_5-x^2_2=0 \label{qcqp_exmp_aux2} \\
			&  x_6-x^3_3=0 \label{qcqp_exmp_aux3}\\
			&  x_7-x_1 x_2=0 \label{qcqp_exmp_aux4}\\
			&  x_8-x_1 x_4=0 \label{qcqp_exmp_aux5}
			\end{align}\end{subequations}
		The transformation of the polynomial optimization to QCQP is standard and it is described in Appendix \ref{apo} for completeness. The global optimal objective value of the above QCQP equals $-2.0198$ and the lower-bound, offered by the standard SDP relaxation equals $-89.8901$.
		In order to solve the above QCQP, we run Algorithm \ref{al:alg_1}, equipped with the SDP relaxation (no additional valid inequalities) and penalty term $\eta=0.025$. The trajectory with three different initializations $\hat{\boldsymbol{x}}^1=[0,0,0,0,0,0,0]^{\top}$,
		$\hat{\boldsymbol{x}}^2=[-3,0,2,9,0,4,0,27]^{\top}$, and
		$\hat{\boldsymbol{x}}^3=[0,4,0,0,16,0,0,0]^{\top}$ are given in Table \ref{table_traj} and shown in Fig. \ref{fig3}. In all three cases, the algorithm achieves feasibility in 1--8 rounds. Moreover, a feasible solution with less than $0.2\%$ gap from global optimality is attained within 10 rounds in all three cases. The example illustrates { a case for which the heuristic Algorithm} \ref{alg:1} is not sensitive to the choice of initial point.
		
\end{example}}

\begin{figure}[h!]
	\captionsetup[subfigure]{position=t}
	\centering
	\includegraphics[height=0.26\textwidth]{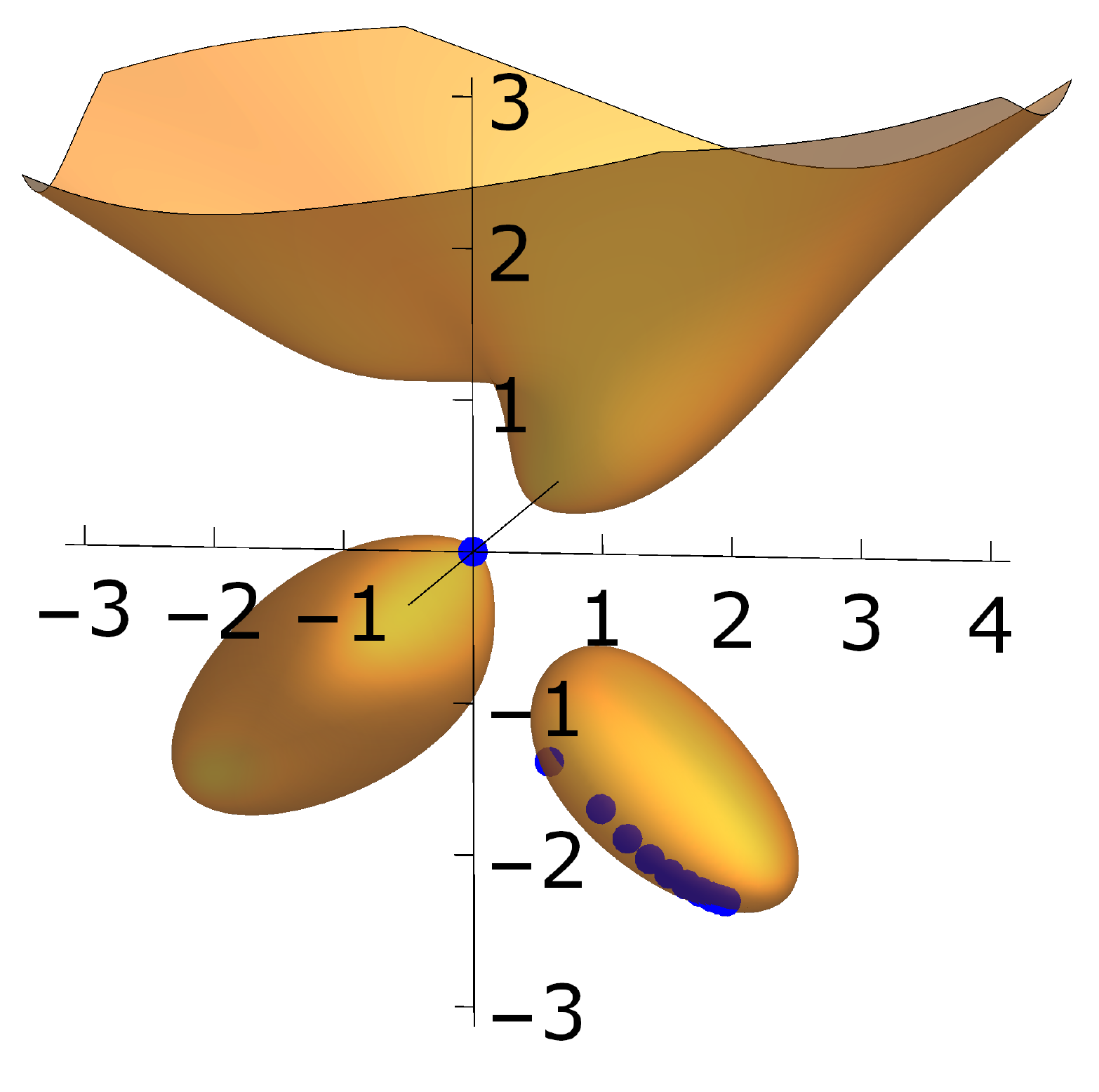}\hspace{5mm}
	\captionsetup[subfigure]{position=t}
	\includegraphics[height=0.26\textwidth]{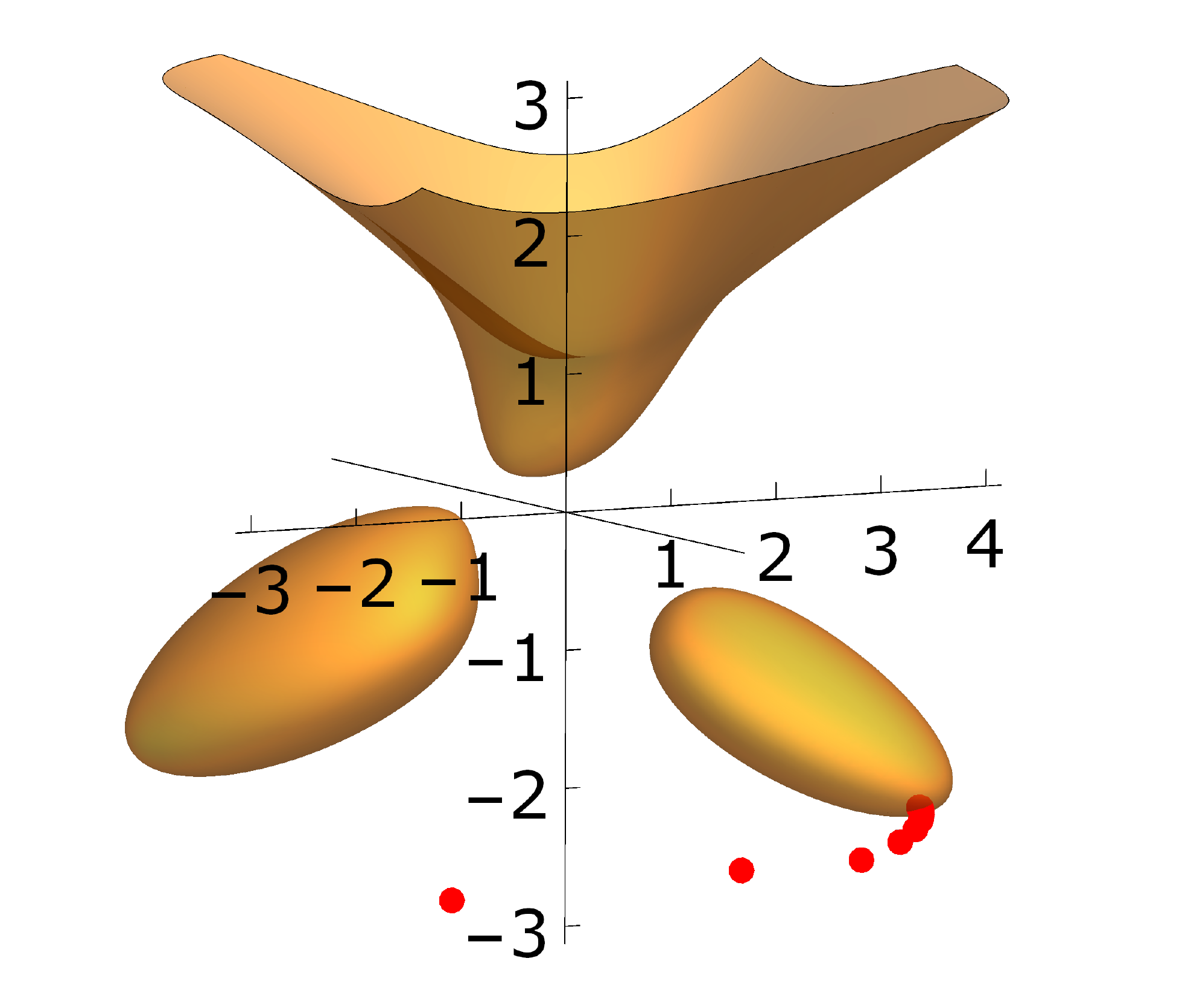}\hspace{5mm}
	\captionsetup[subfigure]{position=t}
	\includegraphics[height=0.26\textwidth]{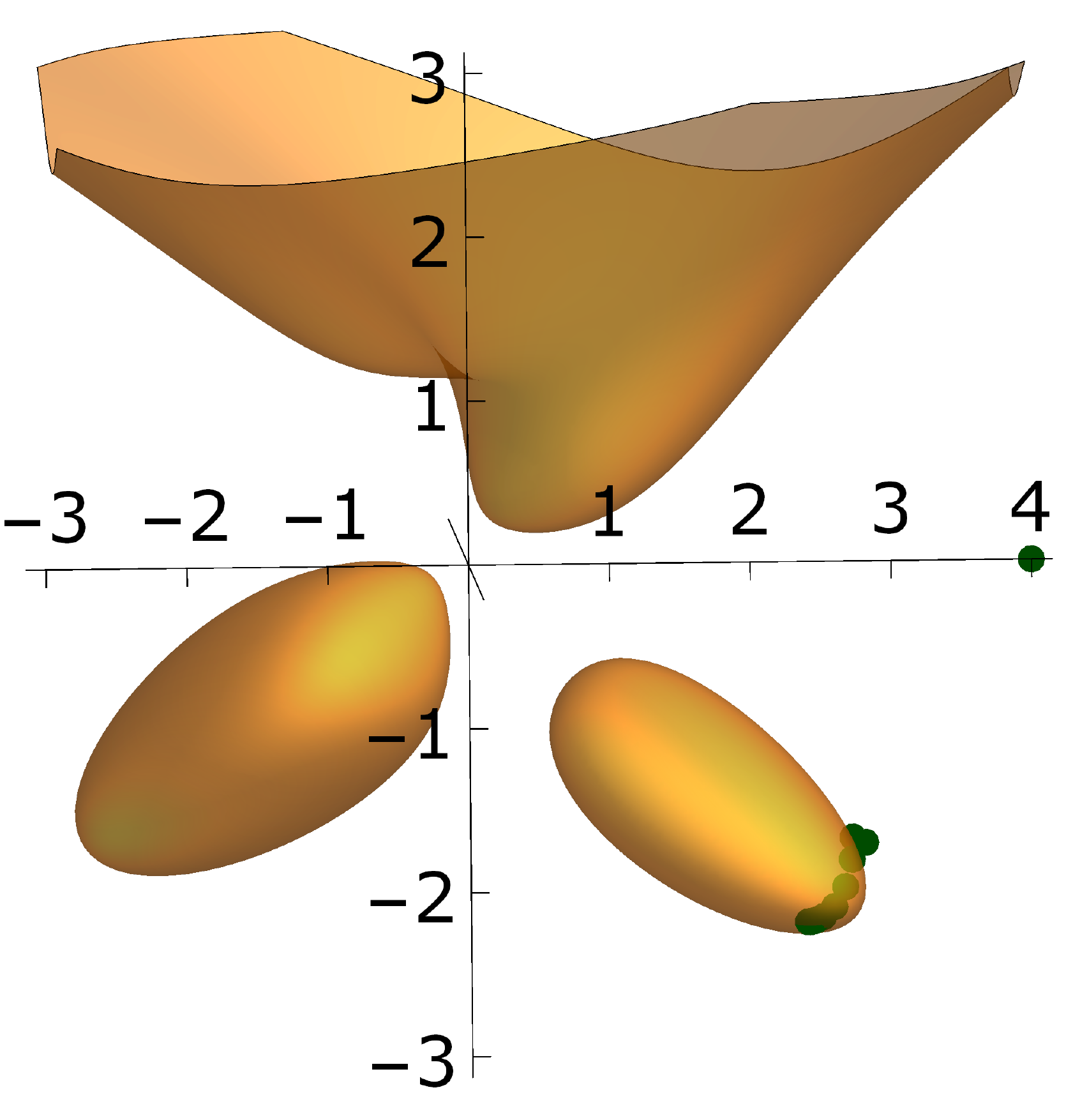}\\
	\captionsetup[subfigure]{position=t}
	\includegraphics[height=0.26\textwidth]{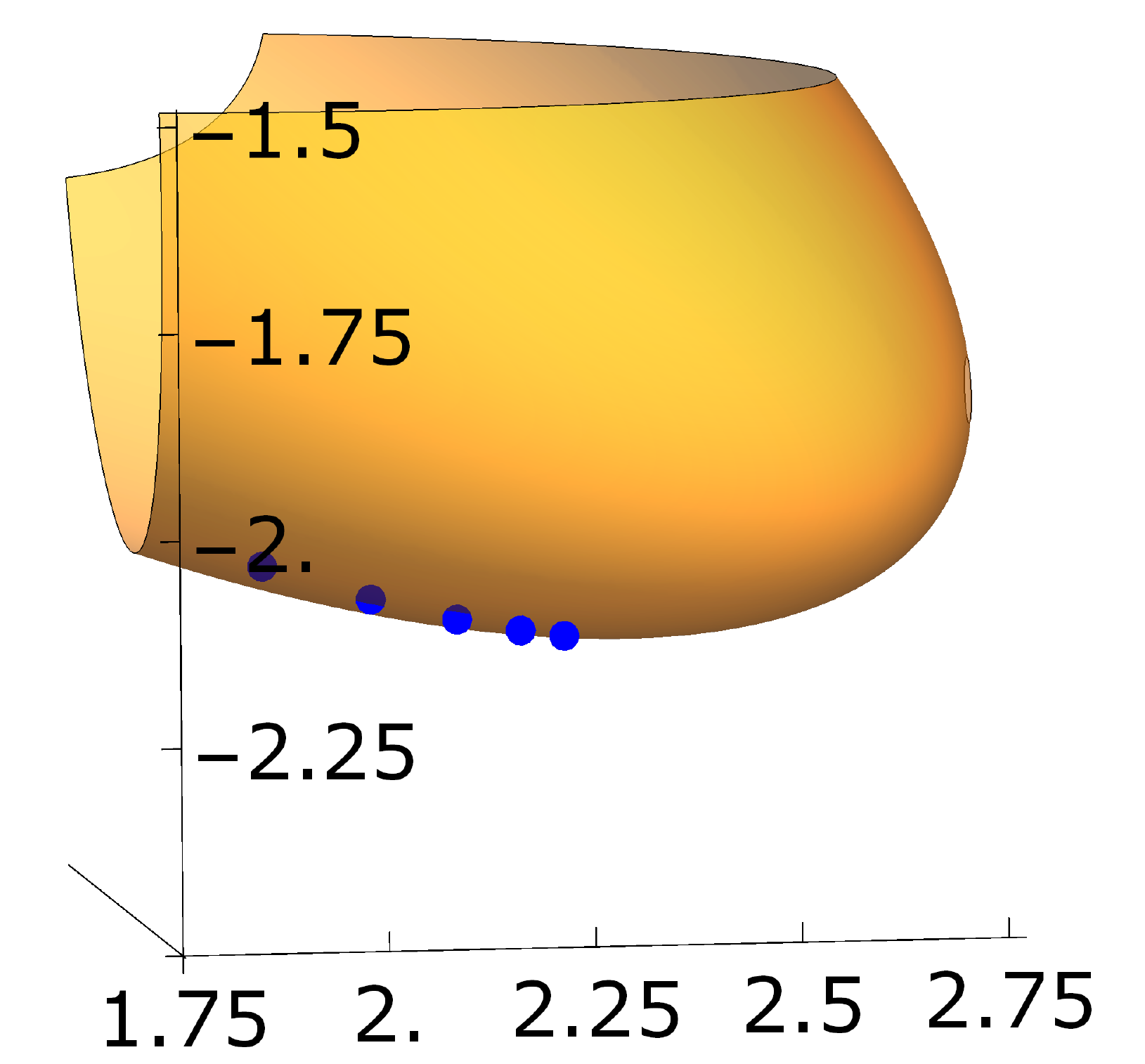}\hspace{5mm}
	\captionsetup[subfigure]{position=t}
	\includegraphics[height=0.26\textwidth]{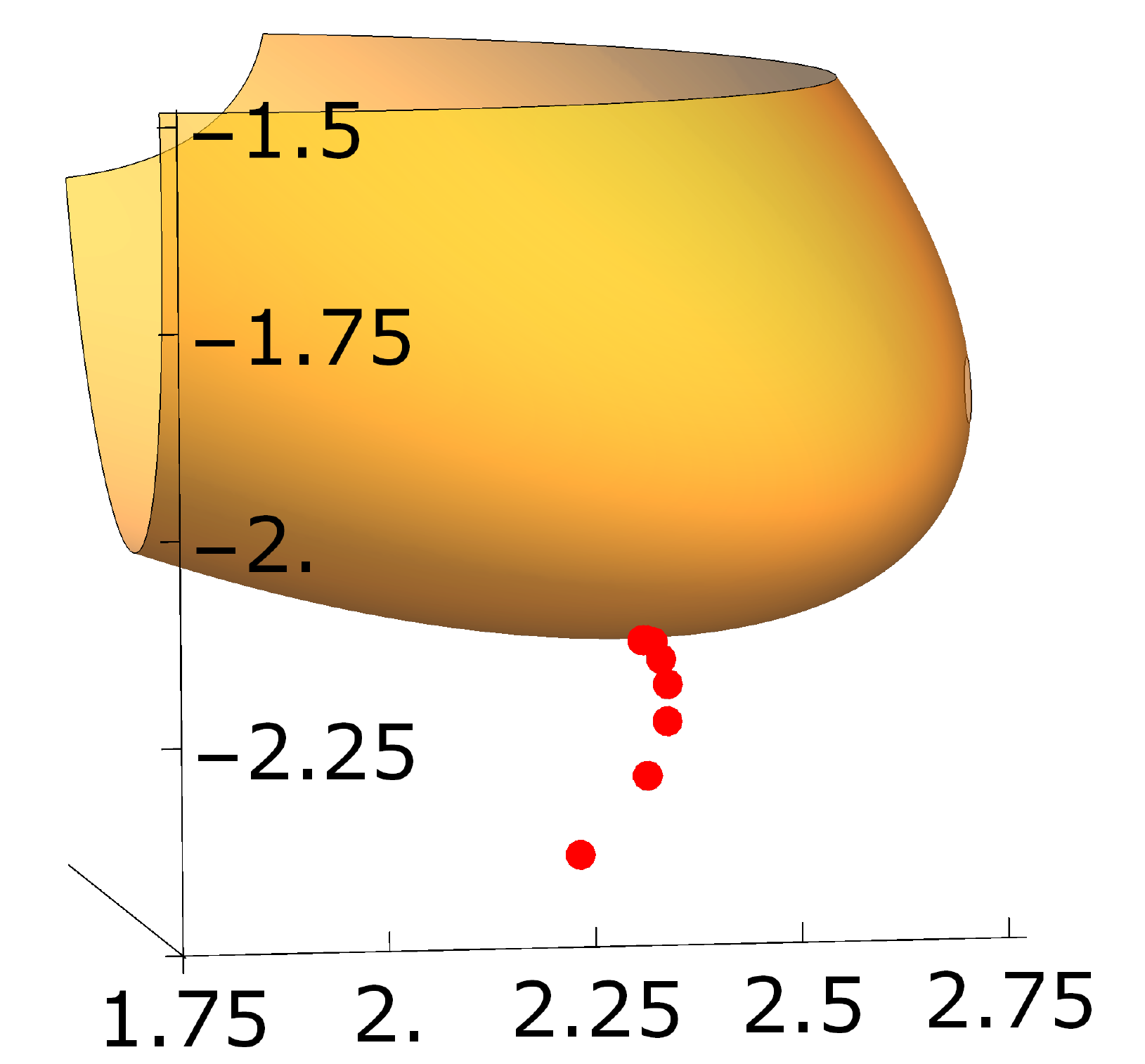}\hspace{5mm}
	\captionsetup[subfigure]{position=t}
	\includegraphics[height=0.26\textwidth]{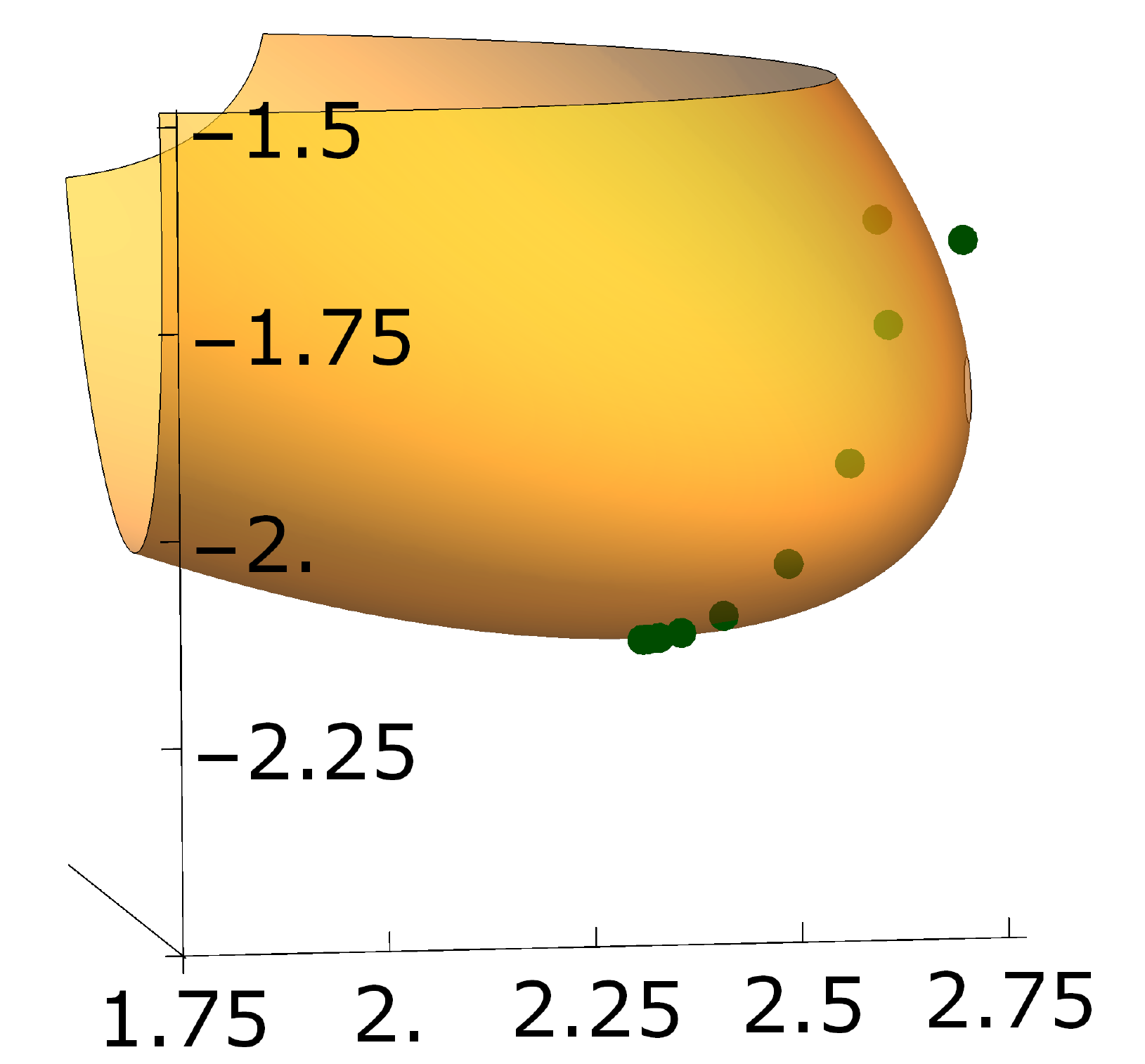}
	\caption{\small {Trajectory of Algorithm \ref{al:alg_1} for three different initializations. The yellow surface represents the feasible set and the , red and green points correspond to $\hat{\boldsymbol{x}}^1$, $\hat{\boldsymbol{x}}^2$ and $\hat{\boldsymbol{x}}^3$, respectively.}}
	\label{fig3}
\end{figure}


\begin{table*}[t]
	\centering
	\caption{\small { Trajectory of Algorithm \ref{al:alg_1} for three different initializations.}}
	\scalebox{0.56}{
		\begin{tabular}
			{c|cccc|cccc|cccc}
			\hline \hline
			\normalsize
			\multirow{2}{*}{Round} & 
			\multicolumn{4}{c|}{$\hat{\boldsymbol{x}}^1$} & 
			\multicolumn{4}{c|}{$\hat{\boldsymbol{x}}^2$} & 
			\multicolumn{4}{c}{$\hat{\boldsymbol{x}}^3$}\\
			\cline{2-13}
			& 
			$a$ (obj.) & $b$ & $c$ & \scalebox{0.80}{ $\mathrm{tr}\{\accentset{\ast}{\boldsymbol{X}}-\accentset{\ast}{\boldsymbol{x}}\accentset{\ast}{\boldsymbol{x}}^{\top}\}$} &
			$a$ (obj.) & $b$ & $c$ & \scalebox{0.80}{ $\mathrm{tr}\{\accentset{\ast}{\boldsymbol{X}}-\accentset{\ast}{\boldsymbol{x}}\accentset{\ast}{\boldsymbol{x}}^{\top}\}$} &
			$a$ (obj.) & $b$ & $c$ & \scalebox{0.80}{ $\mathrm{tr}\{\accentset{\ast}{\boldsymbol{X}}-\accentset{\ast}{\boldsymbol{x}}\accentset{\ast}{\boldsymbol{x}}^{\top}\}$}\\
			\hline
			\hline
			0&
			0.0000& 0.0000& 0.0000& -&
			-3.0000& 0.0000& 2.0000& -&
			0.0000& 4.0000& 0.0000& -
			\\
			1&
			-1.2739& 0.6601& -0.4697& 2.1884&
			-2.5377& 1.2831& -0.7380& 138.9796&
			-1.5721& 2.6848& -0.9492& 39.2455
			\\
			2&
			-1.5173& 1.1445& -1.0128& $<10^{-11}$&
			-2.4389& 2.0715& -1.3946& 51.1170&
			-1.5749& 2.7588& -1.3854& 13.5140
			\\
			3&
			-1.6882& 1.3773& -1.2015& $<10^{-11}$&
			-2.2889& 2.2685& -1.7098& 23.0050&
			-1.6678& 2.6583& -1.5228& 0.9995
			\\
			4&
			-1.8021& 1.5739& -1.3561& $<10^{-11}$&
			-2.1878& 2.3416& -1.8442& 11.4963&
			-1.8322& 2.6083& -1.5587& $<10^{-11}$
			\\
			5&
			-1.8824& 1.7447& -1.4873& $<10^{-11}$&
			-2.1194& 2.3621& -1.9007& 5.9206&
			-1.9460& 2.5261& -1.6624& $<10^{-11}$
			\\
			6&
			-1.9386& 1.8930& -1.5992& $<10^{-11}$&
			-2.0733& 2.3611& -1.9250& 2.9082&
			-2.0002& 2.4391& -1.7847& $<10^{-11}$
			\\
			7&
			-1.9760& 2.0180& -1.6923& $<10^{-11}$&
			-2.0423& 2.3526& -1.9352& 1.1594&
			-2.0156& 2.3824& -1.8598& $<10^{-11}$
			\\
			8&
			-1.9985& 2.1175& -1.7656& $<10^{-11}$&
			-2.0214& 2.3426& -1.9393& 0.0938&
			-2.0189& 2.3532& -1.8938& $<10^{-11}$
			\\
			9&
			-2.0104& 2.1907& -1.8193& $<10^{-11}$&
			-2.0197& 2.3352& -1.9302& $<10^{-11}$&
			-2.0196& 2.3387& -1.9079& $<10^{-11}$
			\\
			10&
			-2.0160& 2.2408& -1.8559& $<10^{-11}$&
			-2.0198& 2.3304& -1.9240& $<10^{-11}$&
			-2.0197& 2.3313& -1.9135& $<10^{-11}$
			\\ \hline \hline
	\end{tabular}}
	\label{table_traj}
\end{table*}

\section{Numerical experiments} \label{sec:exp} In this section we describe numerical experiments to test the effectiveness of the sequential penalization method for non-convex QCQPs from the library of quadratic programming instances (QPLIB) \cite{FuriniEtAl2017TR} as well as  large-scale system identification problems \cite{fattahi2018data}.

\subsection{QPLIB problems}
The experiments are performed on a desktop computer with a 12-core 3.0GHz CPU and 256GB RAM. MOSEK v8.1 \cite{mosek} is used through MATLAB 2017a to solve the resulting SDPs. { The size and number of constraints for each QPLIB instance are reported in Table \ref{tab:QPLIB}.}

\subsubsection{Sequential penalization}
Tables \ref{tab:Seq_SOCP}, \ref{tab:Seq_SOCP_RLT}, \ref{tab:Seq_SDP}, and \ref{tab:Seq_SDP_RLT} report the results of Algorithm \ref{al:alg_1} for $2\times 2$ SDP, $2\times 2$ SDP+RLT, SDP, and SDP+RLT relaxations, respectively. The following valid inequalities are imposed on all of the convex relaxations:
\begin{subequations}
	\begin{align}
	&X_{kk}-(x^{\mathrm{lb}\!\!\!\!\!\phantom{\mathrm{ub}}}_k+x^{\mathrm{ub}\!\!\!\!\!\phantom{\mathrm{lb}}}_k)x_k 
	+x^{\mathrm{lb}\!\!\!\!\!\phantom{\mathrm{ub}}}_k x^{\mathrm{ub}\!\!\!\!\!\phantom{\mathrm{lb}}}_k \leq 0,&\qquad\forall k\in\{1,\ldots,n\}\label{lbub}\\
	&X_{kk}-(x^{\mathrm{ub}\!\!\!\!\!\phantom{\mathrm{lb}}}_k+x^{\mathrm{ub}\!\!\!\!\!\phantom{\mathrm{lb}}}_k)x_k 
	+x^{\mathrm{ub}\!\!\!\!\!\phantom{\mathrm{lb}}}_k x^{\mathrm{ub}\!\!\!\!\!\phantom{\mathrm{lb}}}_k \geq 0,&\qquad\forall k\in\{1,\ldots,n\}\label{ubub}\\
	&X_{kk}-(x^{\mathrm{lb}\!\!\!\!\!\phantom{\mathrm{ub}}}_k+x^{\mathrm{lb}\!\!\!\!\!\phantom{\mathrm{ub}}}_k)x_k 
	+x^{\mathrm{lb}\!\!\!\!\!\phantom{\mathrm{ub}}}_k x^{\mathrm{lb}\!\!\!\!\!\phantom{\mathrm{ub}}}_k \geq 0,&\qquad\forall k\in\{1,\ldots,n\}\label{lblb}
	\end{align}
\end{subequations}
where $\boldsymbol{x}^{\mathrm{lb}},\boldsymbol{x}^{\mathrm{ub}}\in\mathbb{R}^n$ are given lower and upper bounds on $\boldsymbol{x}$. Problem \eqref{prob_lifted_obj}--\eqref{prob_lifted_conic} is solved with the following four settings:
\begin{itemize}
	\item {\it $2\times2$ SDP relaxation:} $r=2$ and valid inequalities \eqref{lbub} -- \eqref{lblb}.
	\item {\it $2\times2$ SDP+RLT relaxation:} $\mathcal{V}=\mathcal{H}\times\mathcal{H}$ and $r=2$.
	\item {\it SDP relaxation:} $r=n$ and valid inequalities \eqref{lbub} -- \eqref{lblb}.
	\item {\it SDP+RLT relaxation:} $\mathcal{V}=\mathcal{H}\times\mathcal{H}$ and $r=n$,
\end{itemize}
where $\mathcal{V}$ and $\mathcal{H}$ is defined in Appendix \ref{apn:RLT}. {\color{black}The assumption $\mathcal{V}=\mathcal{H}\times\mathcal{H}$ means that every pairs of linear constraints are used to generate RLT inequalities.}
Let $(\accentset{\ast}{\boldsymbol{x}},\accentset{\ast}{\boldsymbol{X}})$ denote the optimal solution of the convex relaxation \eqref{prob_lifted_obj}-\eqref{prob_lifted_conic}. We use the point $\hat{\boldsymbol{x}} = \accentset{\ast}{\boldsymbol{x}}$ as the initial point of the algorithm.


The penalty parameter $\eta$ is chosen via bisection as the smallest number of the form $\alpha\times10^\beta$, which results in a tight penalization during the first six rounds, where $\alpha\in\{1,2,5\}$ and $\beta$ is an integer. In all of the experiments, the value of $\eta$ has remained static throughout Algorithm \ref{al:alg_1}. Denote the sequence of penalized SDP solutions obtained by Algorithm \ref{al:alg_1} as
\begin{align}
(\boldsymbol{x}^{(1)},\boldsymbol{X}^{(1)}), \;\;
(\boldsymbol{x}^{(2)},\boldsymbol{X}^{(2)}), \;\;
(\boldsymbol{x}^{(3)},\boldsymbol{X}^{(3)}), \;\;\ldots\nonumber
\end{align}
The smallest $i$ such that
\begin{align}
\mathrm{tr}\{\boldsymbol{X}^{(i)}-\boldsymbol{x}^{(i)}(\boldsymbol{x}^{(i)})^{\top}\}<10^{-7}
\end{align}
is denoted by $i^{\mathrm{feas}}$, i.e., it is the number of rounds that Algorithm \ref{al:alg_1} needs to attain a tight penalization. 
Moreover, the smallest $i$ such that 
\begin{align}
\frac{q_0(\boldsymbol{x}^{(i-1)})-q_0(\boldsymbol{x}^{(i)})}{|q_0(\boldsymbol{x}^{(i)})|}\leq 5\times 10^{-4}\label{scri}
\end{align}
is denoted by $i^{\mathrm{stop}}$, and $\mathrm{UB}\triangleq q_0(\boldsymbol{x}^{(i^{\mathrm{stop}})})$. The following formula is used to calculate the final percentage gaps from the optimal costs reported by the QPLIB library:
\begin{align}
\mathrm{GAP}(\%)&=100\times\frac{q_{0}^{\mathrm{stop}}-
	q_0(\boldsymbol{x}^{\mathrm{QPLIB}})}{|q_0(\boldsymbol{x}^{\mathrm{QPLIB}})|}.\!\!\!\label{gapgap}
\end{align}
{Moreover, $t$(s) denotes the cumulative solver time in seconds for the $i^{\mathrm{stop}}$ rounds. Our results are compared with BARON \cite{BARON} and COUENNE \cite{COUENNE} by fixing the maximum solver times equal to the accumulative solver times spent by Algorithm \ref{al:alg_1}.} We ran BARON and COUENNE through GAMS v25.1.2 \cite{GamsSoftware2013}. The resulting lower bounds, upper bounds and GAPs (from the equation \eqref{gapgap}) are reported in Tables \ref{tab:Seq_SOCP}, \ref{tab:Seq_SOCP_RLT}, \ref{tab:Seq_SDP}, and \ref{tab:Seq_SDP_RLT}.

\begin{table}[t]
\centering
\caption{\small  QPLIB benchmark problems.}
\scalebox{0.65}{\begin{tabular}{c|ccc||c|ccc||c|ccc||c|ccc}
\hline \hline
\!\!\!\multirow{2}{*}{Inst}\!\!\! & Total & Quad & Total & 
\!\!\!\multirow{2}{*}{Inst}\!\!\! & Total & Quad & Total & 
\!\!\!\multirow{2}{*}{Inst}\!\!\! & Total & Quad & Total & 
\!\!\!\multirow{2}{*}{Inst}\!\!\! & Total & Quad & Total \\
	& Var & Cons & Cons &
	& Var & Cons & Cons &
	& Var & Cons & Cons &
	& Var & Cons & Cons \\
\hline \hline
\!\!\! 0343 \!\!\!  & 50 & 0 & 1 & \!\!\! 1353 \!\!\!  & 50 & 1 & 6 & \!\!\! 1535 \!\!\!  & 60 & 60 & 66 & \!\!\! 1773 \!\!\!  & 60 & 1 & 7 \\
\!\!\! 0911 \!\!\!  & 50 & 50 & 50 & \!\!\! 1423 \!\!\!  & 40 & 20 & 24 & \!\!\! 1619 \!\!\!  & 50 & 25 & 30 & \!\!\! 1886 \!\!\!  & 50 & 50 & 50 \\
\!\!\! 0975 \!\!\!  & 50 & 10 & 10 & \!\!\! 1437 \!\!\!  & 50 & 1 & 11 & \!\!\! 1661 \!\!\!  & 60 & 1 & 13 & \!\!\! 1913 \!\!\!  & 48 & 48 & 48 \\
\!\!\! 1055 \!\!\!  & 40 & 20 & 20 & \!\!\! 1451 \!\!\!  & 60 & 60 & 66 & \!\!\! 1675 \!\!\!  & 60 & 1 & 13 & \!\!\! 1922 \!\!\!  & 30 & 60 & 60 \\
\!\!\! 1143 \!\!\!  & 40 & 20 & 24 & \!\!\! 1493 \!\!\!  & 40 & 1 & 5 & \!\!\! 1703 \!\!\!  & 60 & 30 & 36 & \!\!\! 1931 \!\!\!  & 40 & 40 & 40 \\
\!\!\! 1157 \!\!\!  & 40 & 1 & 9 & \!\!\! 1507 \!\!\!  & 30 & 30 & 33 & \!\!\! 1745 \!\!\!  & 50 & 50 & 55 & \!\!\! 1967 \!\!\!  & 50 & 75 & 75 \\
\hline \hline
\end{tabular}}
\label{tab:QPLIB}
\end{table}

\begin{table}[t]
	\centering
	\caption{\small Sequential penalized $2\times 2$ SDP.}
	\scalebox{0.7}{\begin{tabular}{c|cccccc|ccc|ccc}
			\hline \hline
			\!\!\!\multirow{2}{*}{Inst}\!\!\! & 
			\multicolumn{6}{c|}{\bf Sequential penalized $2\times 2$ SDP} & 
			\multicolumn{3}{c|}{\bf BARON} & 
			\multicolumn{3}{c}{\bf COUENNE} \\
			\cline{2-13}
			& $\eta$ & $i^{\sm{\mathrm{feas}}{6}}$ & $i^{\sm{\mathrm{stop}}{6}}$ & $t$(s) 
			& UB & GAP(\%)\!\!\! &  
			LB & UB & GAP(\%)\!\!\! & LB & UB & GAP(\%)\!\!\! \\
			\hline
			\hline
			\!\!\!\! 0343$\dagger$ \!\!\!\! & \!\!
			\!\!\! 5e+2 \!\!\!\!\! & \!\!\!\!\!\!\!\! 1 \!\!\!\!\!  & \!\!\!\!\!\!  100 \!\!\!\!\! & \!\!\!\!\!\!  75.27 \!\!\!\!\! 
			& \!\!  -5.882 \!\!\!\!\! & \!\!\!\!\!\!  { 7.89} \!\!\!\!\!  & \!\!\!\!  
			\!  -95.372 \!\!\! & \!\!\!  -6.386 \!\!\! & \!\!\!\!\!\!  0.00 \!\!\!\!\! & \!\!\!\!
			\!  -7668.005 \!\!\!\!\! & \!\!\!  -6.386 \!\!\! & \!\!\!\!\!\!  0.00 \!\!\!\!\!\!\!\!\!\\ 
			\!\!\!\! 0911 \!\!\!\! & \!\!
			\!\!\! 1e+1 \!\!\!\!\! & \!\!\!\!\!\!\!\! 1 \!\!\!\!\!  & \!\!\!\!\!\!  29 \!\!\!\!\! & \!\!\!\!\!\!  22.91 \!\!\!\!\! 
			& \!\!  -30.675 \!\!\!\!\! & \!\!\!\!\!\!  { 4.58} \!\!\!\!\!  & \!\!\!\!  
			\!  -172.777 \!\!\! & \!\!\!  0.000 \!\!\! & \!\!\!\!\!\!  100 \!\!\!\!\! & \!\!\!\!
			\!  -172.777 \!\!\! & \!\!\!  -31.026 \!\!\! & \!\!\!\!\!\!  3.49 \!\!\!\!\!\!\!\!\!\\ 
			\!\!\!\! 0975 \!\!\!\! & \!\!
			\!\!\! 5e+0 \!\!\!\!\! & \!\!\!\!\!\!\!\! 6 \!\!\!\!\!  & \!\!\!\!\!\!  18 \!\!\!\!\! & \!\!\!\!\!\!  46.36 \!\!\!\!\! 
			& \!\!  -36.434 \!\!\!\!\! & \!\!\!\!\!\!  { 3.75} \!\!\!\!\!  & \!\!\!\!  
			\!  -47.428 \!\!\! & \!\!\!  -37.801 \!\!\! & \!\!\!\!\!\!  0.14 \!\!\!\!\! & \!\!\!\!
			\! -171.113 \!\!\! & \!\!\!  -37.213 \!\!\! & \!\!\!\!\!\!  1.69 \!\!\!\!\!\!\!\!\!\\ 
			\!\!\!\! 1055 \!\!\!\! & \!\!
			\!\!\! 1e+1 \!\!\!\!\! & \!\!\!\!\!\!\!\! 1 \!\!\!\!\!  & \!\!\!\!\!\!  22 \!\!\!\!\!  & \!\!\!\!\!\!  14.39 \!\!\!\!\! 
			& \!\!  -32.620 \!\!\!\!\! & \!\!\!\!\!\!  { 1.26} \!\!\!\!\!  & \!\!\!\!  
			\!  -37.841 \!\!\! & \!\!\!  -33.037 \!\!\! & \!\!\!\!\!\!  0.00 \!\!\!\!\! & \!\!\!\!
			\!  -199.457 \!\!\! & \!\!\!  -33.037 \!\!\! & \!\!\!\!\!\!  0.00 \!\!\!\!\!\!\!\!\!\\ 
			\!\!\!\! 1143 \!\!\!\! & \!\!
			\!\!\! 2e+1 \!\!\!\!\! & \!\!\!\!\!\!\!\! 1 \!\!\!\!\!  & \!\!\!\!\!\!  44 \!\!\!\!\! & \!\!\!\!\!\!  25.68 \!\!\!\!\! 
			& \!\!  -55.417 \!\!\!\!\! & \!\!\!\!\!\! {  3.20} \!\!\!\!\!  & \!\!\!\!  
			\!  -69.522 \!\!\! & \!\!\!  -57.247 \!\!\! & \!\!\!\!\!\!  0.00 \!\!\!\!\! & \!\!\!\!
			\! -384.45 \!\!\! & \!\!\!  -56.237 \!\!\! & \!\!\!\!\!\!  1.76 \!\!\!\!\!\!\!\!\!\\ 
			\!\!\!\! 1157 \!\!\!\! & \!\!
			\!\!\! 2e+0 \!\!\!\!\! & \!\!\!\!\!\!\!\! 2 \!\!\!\!\!  & \!\!\!\!\!\!  9 \!\!\!\!\!  & \!\!\!\!\!\!  9.01 \!\!\!\!\! 
			& \!\!  -10.938 \!\!\!\!\! & \!\!\!\!\!\!  { 0.10} \!\!\!\!\!  & \!\!\!\!  
			\!  -11.414 \!\!\! & \!\!\!  -10.948 \!\!\! & \!\!\!\!\!\!  0.00 \!\!\!\!\! & \!\!\!\!
			\!  -80.51 \!\!\! & \!\!\!  -10.948 \!\!\! & \!\!\!\!\!\!  0.00 \!\!\!\!\!\!\!\!\!\\ 
			\!\!\!\! 1353 \!\!\!\! & \!\!
			\!\!\! 5e+0 \!\!\!\!\! & \!\!\!\!\!\!\!\! 1 \!\!\!\!\!  & \!\!\!\!\!\!  48 \!\!\!\!\! & \!\!\!\!\!\!  84.90 \!\!\!\!\!  
			& \!\!  -7.700 \!\!\!\!\! & \!\!\!\!\!\!  { 0.19} \!\!\!\!\! & \!\!\!\!  
			\! -7.925 \!\!\! & \!\!\!  -7.714 \!\!\! & \!\!\!\!\!\!  0.00 \!\!\!\!\! & \!\!\!\!
			\! -73.28 \!\!\! & \!\!\!  -7.714 \!\!\! & \!\!\!\!\!\!  0.00 \!\!\!\!\!\!\!\!\!\\ 
			\!\!\!\! 1423 \!\!\!\! & \!\!
			\!\!\! 5e+0 \!\!\!\!\! & \!\!\!\!\!\!\!\! 1 \!\!\!\!\!  & \!\!\!\!\!\!  29 \!\!\!\!\!  & \!\!\!\!\!\!  17.44 \!\!\!\!\!  
			& \!\!  -14.684 \!\!\!\!\! & \!\!\!\!\!\!  { 1.90} \!\!\!\!\! & \!\!\!\!  
			\!  -16.313 \!\!\! & \!\!\!  -14.968 \!\!\! & \!\!\!\!\!\!  0.00 \!\!\!\!\! & \!\!\!\!
			\!  -76.13 \!\!\! & \!\!\!  -14.871 \!\!\! & \!\!\!\!\!\!  0.65 \!\!\!\!\!\!\!\!\!\\ 
			\!\!\!\! 1437 \!\!\!\! & \!\!
			\!\!\! 5e+0 \!\!\!\!\! & \!\!\!\!\!\!\!\! 1 \!\!\!\!\!  & \!\!\!\!\!\!  36 \!\!\!\!\!  & \!\!\!\!\!\!  54.57 \!\!\!\!\!  
			& \!\!  -7.785 \!\!\!\!\! & \!\!\!\!\!\!  { 0.06} \!\!\!\!\! & \!\!\!\!  
			\!  -9.601 \!\!\! & \!\!\!  -7.789 \!\!\! & \!\!\!\!\!\!  0.00 \!\!\!\!\! & \!\!\!\!
			\!  -87.58 \!\!\! & \!\!\!  -7.789 \!\!\! & \!\!\!\!\!\!  0.00 \!\!\!\!\!\!\!\!\!\\ 
			\!\!\!\! 1451 \!\!\!\! & \!\!
			\!\!\! 2e+1 \!\!\!\!\! & \!\!\!\!\!\!\!\! 4 \!\!\!\!\!  & \!\!\!\!\!\!  21 \!\!\!\!\!  & \!\!\!\!\!\!  20.86 \!\!\!\!\!  
			& \!\!  -85.598 \!\!\!\!\! & \!\!\!\!\!\! { 2.26} \!\!\!\!\! & \!\!\!\!  
			\!  -135.140 \!\!\! & \!\!\!  -87.577 \!\!\! & \!\!\!\!\!\!  0.00 \!\!\!\!\! & \!\!\!\!
			\!  -468.04 \!\!\! & \!\!\!  -86.860 \!\!\! & \!\!\!\!\!\!  0.82 \!\!\!\!\!\!\!\!\!\\ 
			\!\!\!\! 1493 \!\!\!\! & \!\!
			\!\!\! 2e+1 \!\!\!\!\! & \!\!\!\!\!\!\!\! 1 \!\!\!\!\!  & \!\!\!\!\!\!  18 \!\!\!\!\!  & \!\!\!\!\!\!  14.49 \!\!\!\!\!  
			& \!\!  -41.910 \!\!\!\!\! & \!\!\!\!\!\!  { 2.90} \!\!\!\!\! & \!\!\!\!  
			\!  -47.239 \!\!\! & \!\!\!  -43.160 \!\!\! & \!\!\!\!\!\!  0.00 \!\!\!\!\! & \!\!\!\!
			\!  -395.69 \!\!\! & \!\!\!  -43.160 \!\!\! & \!\!\!\!\!\!  0.00 \!\!\!\!\!\!\!\!\!\\ 
			\!\!\!\! 1507 \!\!\!\! & \!\!
			\!\!\!\! 2e+0 \!\!\!\!\! & \!\!\!\!\!\!\!\! 1 \!\!\!\!\!  & \!\!\!\!\!\! 15 \!\!\!\!\!  & \!\!\!\!\!\!  8.98 \!\!\!\!\!  
			& \!\!  -8.289 \!\!\!\!\! & \!\!\!\!\!\!  { 0.15} \!\!\!\!\! & \!\!\!\!  
			\! -49.709 \!\!\! & \!\!\!  -8.301 \!\!\! & \!\!\!\!\!\!  0.00 \!\!\!\!\! & \!\!\!\!
			\! -44.37 \!\!\! & \!\!\!  -8.301 \!\!\! & \!\!\!\!\!\!  0.00 \!\!\!\!\!\!\!\!\!\\ 
			\!\!\!\! 1535 \!\!\!\! & \!\!
			\!\!\! 5e+0 \!\!\!\!\! & \!\!\!\!\!\!\!\! 1 \!\!\!\!\!  & \!\!\!\!\!\!  26 \!\!\!\!\!  & \!\!\!\!\!\!  28.16 \!\!\!\!\!  
			& \!\!  -10.948 \!\!\!\!\! & \!\!\!\!\!\!  { 5.51} \!\!\!\!\! & \!\!\!\!  
			\!  -13.407 \!\!\! & \!\!\!  -11.397 \!\!\! & \!\!\!\!\!\!  1.63 \!\!\!\!\! & \!\!\!\!
			\!  -107.86 \!\!\! & \!\!\!  -11.398 \!\!\! & \!\!\!\!\!\!  1.63 \!\!\!\!\!\!\!\!\!\\ 
			\!\!\!\! 1619 \!\!\!\! & \!\!
			\!\!\! 5e+0 \!\!\!\!\! & \!\!\!\!\!\!\!\! 1 \!\!\!\!\!  & \!\!\!\!\!\!  39 \!\!\!\!\!  & \!\!\!\!\!\!  32.34 \!\!\!\!\! 
			& \!\!  -9.210 \!\!\!\!\! & \!\!\!\!\!\!  { 0.08} \!\!\!\!\! & \!\!\!\!  
			\! -10.302 \!\!\! & \!\!\!  -9.217 \!\!\! & \!\!\!\!\!\!  0.00 \!\!\!\!\! & \!\!\!\!
			\! -74.55 \!\!\! & \!\!\!  -9.217 \!\!\! & \!\!\!\!\!\!  0.00 \!\!\!\!\!\!\!\!\!\\ 
			\!\!\!\! 1661 \!\!\!\! & \!\!
			\!\!\! 5e+0 \!\!\!\!\! & \!\!\!\!\!\!\!\! 1 \!\!\!\!\!  & \!\!\!\!\!\!   32 \!\!\!\!\!  & \!\!\!\!\!\!  87.50  \!\!\!\!\!  
			& \!\!  -15.666 \!\!\!\!\! & \!\!\!\!\!\!  { 1.81} \!\!\!\!\! & \!\!\!\!  
			\!  -19.667 \!\!\! & \!\!\!  -15.955 \!\!\! & \!\!\!\!\!\!  0.00 \!\!\!\!\! & \!\!\!\!
			\!  -139.25 \!\!\! & \!\!\!  -15.955 \!\!\! & \!\!\!\!\!\!  0.00 \!\!\!\!\!\!\!\!\!\\ 
			\!\!\!\! 1675 \!\!\!\! & \!\!
			\!\!\! 2e+1 \!\!\!\!\! & \!\!\!\!\!\!\!\! 1 \!\!\!\!\!  & \!\!\!\!\!\!  21 \!\!\!\!\!  & \!\!\!\!\!\!  36.38 \!\!\!\!\! 
			& \!\!  -75.485 \!\!\!\!\! & \!\!\!\!\!\! { 0.24} \!\!\!\!\! & \!\!\!\!  
			\! -96.864 \!\!\! & \!\!\!  -75.669 \!\!\! & \!\!\!\!\!\!  0.00 \!\!\!\!\! & \!\!\!\!
			\! -435.48 \!\!\! & \!\!\!  -75.669 \!\!\! & \!\!\!\!\!\!  0.00 \!\!\!\!\!\!\!\!\!\\ 
			\!\!\!\! 1703 \!\!\!\! & \!\!
			\!\!\! 5e+1 \!\!\!\!\! & \!\!\!\!\!\!\!\! 2 \!\!\!\!\!  & \!\!\!\!\!\!  30 \!\!\!\!\!  & \!\!\!\!\!\!  31.82 \!\!\!\!\!  
			& \!\!  -130.902 \!\!\!\!\! & \!\!\!\!\!\! { 1.43} \!\!\!\!\! & \!\!\!\!  
			\!  -180.935 \!\!\! & \!\!\!  -132.802 \!\!\! & \!\!\!\!\!\!  0.00 \!\!\!\!\! & \!\!\!\!
			\!  -929.92 \!\!\! & \!\!\!  -132.802 \!\!\! & \!\!\!\!\!\!  0.00 \!\!\!\!\!\!\!\!\!\\ 
			\!\!\!\! 1745 \!\!\!\! & \!\!
			\!\!\! 2e+1 \!\!\!\!\! & \!\!\!\!\!\!\!\! 1 \!\!\!\!\!  & \!\!\!\!\!\!  26 \!\!\!\!\!  & \!\!\!\!\!\!  22.15 \!\!\!\!\!  
			& \!\!  -71.704 \!\!\!\!\! & \!\!\!\!\!\!  { 0.93} \!\!\!\!\! & \!\!\!\!  
			\!  -77.465 \!\!\! & \!\!\!  -72.377 \!\!\! & \!\!\!\!\!\!  0.00 \!\!\!\!\! & \!\!\!\!
			\!  -317.99 \!\!\! & \!\!\!  -72.377 \!\!\! & \!\!\!\!\!\!  0.00 \!\!\!\!\!\!\!\!\!\\ 
			\!\!\!\! 1773 \!\!\!\! & \!\!
			\!\!\! 5e+0 \!\!\!\!\! & \!\!\!\!\!\!\!\! 1 \!\!\!\!\!  & \!\!\!\!\!\!  56 \!\!\!\!\!  & \!\!\!\!\!\!  148.79 \!\!\!\!\!  
			& \!\!  -14.154 \!\!\!\!\! & \!\!\!\!\!\!  { 3.34} \!\!\!\!\! & \!\!\!\!  
			\!  -21.581 \!\!\! & \!\!\!  -14.642 \!\!\! & \!\!\!\!\!\!  0.00 \!\!\!\!\! & \!\!\!\!
			\!  -118.65 \!\!\! & \!\!\!  -14.642 \!\!\! & \!\!\!\!\!\!  0.00 \!\!\!\!\!\!\!\!\!\\ 
			\!\!\!\! 1886 \!\!\!\! & \!\!
			\!\!\! 2e+1 \!\!\!\!\! & \!\!\!\!\!\!\!\! 1 \!\!\!\!\!  & \!\!\!\!\!\! 34 \!\!\!\!\!  & \!\!\!\!\!\!  26.82 \!\!\!\!\!  
			& \!\!  -78.604 \!\!\!\!\! & \!\!\!\!\!\!  { 0.09} \!\!\!\!\! & \!\!\!\!  
			\!  -135.615 \!\!\! & \!\!\!  -78.672 \!\!\! & \!\!\!\!\!\!  0.00 \!\!\!\!\! & \!\!\!\!
			\!  -324.87 \!\!\! & \!\!\!  -78.672 \!\!\! & \!\!\!\!\!\!  0.00 \!\!\!\!\!\!\!\!\!\\ 
			\!\!\!\! 1913 \!\!\!\! & \!\!
			\!\!\! 1e+1 \!\!\!\!\! & \!\!\!\!\!\!\!\! 1 \!\!\!\!\!  & \!\!\!\!\!\!  28 \!\!\!\!\!  & \!\!\!\!\!\!  21.91 \!\!\!\!\! 
			& \!\!  -51.889 \!\!\!\!\! & \!\!\!\!\!\!  { 0.42} \!\!\!\!\! & \!\!\!\!  
			\!  -68.555 \!\!\! & \!\!\!  -52.109 \!\!\! & \!\!\!\!\!\!  0.00 \!\!\!\!\! & \!\!\!\!
			\!  -164.26 \!\!\! & \!\!\!  -51.478 \!\!\! & \!\!\!\!\!\!  1.21 \!\!\!\!\!\!\!\!\!\\ 
			\!\!\!\! 1922 \!\!\!\! & \!\!
			\!\!\! 1e+1 \!\!\!\!\! & \!\!\!\!\!\!\!\! 1 \!\!\!\!\!  & \!\!\!\!\!\!  23 \!\!\!\!\!  & \!\!\!\!\!\!  11.16 \!\!\!\!\!  
			& \!\!  -35.437 \!\!\!\!\! & \!\!\!\!\!\!  { 1.43} \!\!\!\!\! & \!\!\!\!  
			\!  -121.872 \!\!\! & \!\!\!  -35.951 \!\!\! & \!\!\!\!\!\!  0.00 \!\!\!\!\! & \!\!\!\!
			\!  -123.2 \!\!\! & \!\!\!  -35.951 \!\!\! & \!\!\!\!\!\!  0.00  \!\!\!\!\!\!\!\!\!\\ 
			\!\!\!\! 1931 \!\!\!\! & \!\!
			\!\!\! 1e+1 \!\!\!\!\! & \!\!\!\!\!\!\!\! 1 \!\!\!\!\!  & \!\!\!\!\!\!  13 \!\!\!\!\!  & \!\!\!\!\!\!  8.78 \!\!\!\!\!  
			& \!\!  -53.684 \!\!\!\!\! & \!\!\!\!\!\!  { 3.64} \!\!\!\!\!& \!\!\!\!  
			\!  -85.196 \!\!\! & \!\!\!  -55.709 \!\!\! & \!\!\!\!\!\!  0.00 \!\!\!\!\! & \!\!\!\!
			\!  -204.08 \!\!\! & \!\!\!  -54.290 \!\!\! & \!\!\!\!\!\!  2.55 \!\!\!\!\!\!\!\!\!\\ 
			\!\!\!\! 1967 \!\!\!\! & \!\!
			\!\!\! 5e+1 \!\!\!\!\! & \!\!\!\!\!\!\!\! 1 \!\!\!\!\!  & \!\!\!\!\!\!  32 \!\!\!\!\!  & \!\!\!\!\!\!  27.23 \!\!\!\!\! 
			& \!\!  -105.570 \!\!\!\!\! & \!\!\!\!\!\!  { 1.87} \!\!\!\!\!  & \!\!\!\!  
			\!  -136.098 \!\!\! & \!\!\!  0.000 \!\!\! & \!\!\!\!\!\!  100 \!\!\!\!\! & \!\!\!\!
			\!  -622.57 \!\!\! & \!\!\! -107.581 \!\!\! & \!\!\!\!\!\!  0.00 \!\!\!\!\!\!\!\!\!\\ 
			\hline
			\!\!\!\! Max \!\!\!\! & \!\!
			\!\!\! 500 \!\!\!\!\! & \!\!\!\!\!\!\!\! 6 \!\!\!\!\!  & \!\!\!\!\!\!  100 \!\!\!\!\!  & \!\!\!\!\!\! 148.79 \!\!\!\!\! 
			& \!\!   \!\!\!\!\! & \!\!\!\!\!\!  7.89 \!\!\!\!\! &  \!\!\!\!    
			\!    \!\!\! & \!\!\!    \!\!\! & \!\!\!\!\!\!  100 \!\!\!\!\! & \!\!\!\!
			\!    \!\!\! & \!\!\!    \!\!\! & \!\!\!\!\!\!  3.34 \!\!\!\!\!\!\!\!\!\\ 
			\hline \hline
	\end{tabular}}
	\label{tab:Seq_SOCP}
\end{table}

\begin{table}[t]
	\centering
	\caption{\small  Sequential penalized $2\times 2$ SDP+RLT.}
	\scalebox{0.70}{\begin{tabular}{c|cccccc|ccc|ccc}
			\hline \hline
			\!\!\!\multirow{2}{*}{Inst}\!\!\! & 
			\multicolumn{6}{c|}{\bf Sequential $2\times 2$ SDP+RLT} & 
			\multicolumn{3}{c|}{\bf BARON} & 
			\multicolumn{3}{c}{\bf COUENNE} \\
			\cline{2-13}
			& $\eta$ & $i^{\sm{\mathrm{feas}}{6}}$ & $i^{\sm{\mathrm{stop}}{6}}$ & $t$(s) 
			& UB & GAP(\%)\!\!\! & 
			LB & UB & GAP(\%)\!\!\! & LB & UB & GAP(\%)\!\!\! \\
			\hline
			\hline
			\!\!\!\! 0343 \!\!\!\! & \!\!
			\!\!\! 1e+2 \!\!\!\!\! & \!\!\!\!\!\!\!\! 4 \!\!\!\!\!  & \!\!\!\!\!\!  24 \!\!\!\!\!  & \!\!\!\!\!\!  25.23 \!\!\!\!\! 
			& \!\!  -5.945 \!\!\!\!\! & \!\!\!\!\!\!  6.91 \!\!\!\!\! &  \!\!\!\!  
			\!  -95.372 \!\!\! & \!\!\!  -6.386 \!\!\! & \!\!\!\!\!\!  0.00 \!\!\!\!\! & \!\!\!\!
			\!  -7668.005 \!\!\! & \!\!\!  -6.386 \!\!\! & \!\!\!\!\!\!  0.00 \!\!\!\!\!\!\!\!\!\\ 
			\!\!\!\! 0911 \!\!\!\! & \!\!
			\!\!\! 1e+1 \!\!\!\!\! & \!\!\!\!\!\!\!\! 1 \!\!\!\!\!  & \!\!\!\!\!\!  33 \!\!\!\!\!  & \!\!\!\!\!\!  27.69 \!\!\!\!\! 
			& \!\!  -30.923 \!\!\!\!\! & \!\!\!\!\!\!  3.81 \!\!\!\!\!  & \!\!\!\!  
			\!  -172.777 \!\!\! & \!\!\! -32.148 \!\!\! & \!\!\!\!\!\!  0.00 \!\!\!\!\! & \!\!\!\!
			\!  -172.777 \!\!\! & \!\!\!  -31.026 \!\!\! & \!\!\!\!\!\!  3.49 \!\!\!\!\!\!\!\!\!\\ 
			\!\!\!\! 0975 \!\!\!\! & \!\!
			\!\!\! 5e+0 \!\!\!\!\! & \!\!\!\!\!\!\!\! 6 \!\!\!\!\!  & \!\!\!\!\!\!  15 \!\!\!\!\!  & \!\!\!\!\!\!  4.10 \!\!\!\!\! 
			& \!\!  -36.300 \!\!\!\!\! &  \!\!\!\!\!\!  13.17 \!\!\!\!\! & \!\!\!\!  
			\!  -47.428 \!\!\! & \!\!\!  -37.794 \!\!\! & \!\!\!\!\!\!  0.16 \!\!\!\!\! & \!\!\!\!
			\!  -171.113 \!\!\! & \!\!\! -36.812 \!\!\! & \!\!\!\!\!\! 2.75 \!\!\!\!\!\!\!\!\!\\ 
			\!\!\!\! 1055 \!\!\!\! & \!\!
			\!\!\! 1e+1 \!\!\!\!\! & \!\!\!\!\!\!\!\! 1 \!\!\!\!\!  & \!\!\!\!\!\!  24 \!\!\!\!\!  & \!\!\!\!\!\!  16.78 \!\!\!\!\! 
			& \!\!  -32.666 \!\!\!\!\! & \!\!\!\!\!\!  1.12 \!\!\!\!\!  & \!\!\!\!  
			\!  -37.841 \!\!\! & \!\!\!  -33.037 \!\!\! & \!\!\!\!\!\!  0.00 \!\!\!\!\! & \!\!\!\!
			\!  -199.457 \!\!\! & \!\!\!  -33.037 \!\!\! & \!\!\!\!\!\!  0.00 \!\!\!\!\!\!\!\!\!\\ 
			\!\!\!\! 1143 \!\!\!\! & \!\!
			\!\!\! 2e+1 \!\!\!\!\! & \!\!\!\!\!\!\!\! 1 \!\!\!\!\!  & \!\!\!\!\!\!  30 \!\!\!\!\!  & \!\!\!\!\!\!  32.66 \!\!\!\!\! 
			& \!\!  -55.507 \!\!\!\!\! & \!\!\!\!\!\!  3.04 \!\!\!\!\!  & \!\!\!\!  
			\!  -69.522 \!\!\! & \!\!\!  -57.247 \!\!\! & \!\!\!\!\!\!  0.00 \!\!\!\!\! & \!\!\!\!
			\!  -384.45 \!\!\! & \!\!\!  -56.237 \!\!\! & \!\!\!\!\!\! 1.76 \!\!\!\!\!\!\!\!\!\\ 
			\!\!\!\! 1157 \!\!\!\! & \!\!
			\!\!\! 2e+0 \!\!\!\!\! & \!\!\!\!\!\!\!\! 1 \!\!\!\!\!  & \!\!\!\!\!\! 0  \!\!\!\!\!  & \!\!\!\!\!\!  1.14 \!\!\!\!\! 
			& \!\!  -10.948 \!\!\!\!\! & \!\!\!\!\!\!  0.00 \!\!\!\!\!  & \!\!\!\!  
			\!  -11.414 \!\!\! & \!\!\!  -10.948 \!\!\! & \!\!\!\!\!\!  0.00 \!\!\!\!\! & \!\!\!\!
			\!  -80.51 \!\!\! & \!\!\!  -10.948 \!\!\! & \!\!\!\!\!\!  0.00 \!\!\!\!\!\!\!\!\!\\ 
			\!\!\!\! 1353 \!\!\!\! & \!\!
			\!\!\! 1e+0 \!\!\!\!\! & \!\!\!\!\!\!\!\! 3 \!\!\!\!\!  & \!\!\!\!\!\!  11 \!\!\!\!\!  & \!\!\!\!\!\!  19.41 \!\!\!\!\! 
			& \!\!  -7.711 \!\!\!\!\! & \!\!\!\!\!\!  0.05 \!\!\!\!\!  & \!\!\!\!  
			\!  -7.925 \!\!\! & \!\!\!  -7.714 \!\!\! & \!\!\!\!\!\!  0.00 \!\!\!\!\! & \!\!\!\!
			\!  -73.28 \!\!\! & \!\!\!  -7.714 \!\!\! & \!\!\!\!\!\!  0.00 \!\!\!\!\!\!\!\!\!\\ 
			\!\!\!\! 1423 \!\!\!\! & \!\!
			\!\!\! 2e+0 \!\!\!\!\! & \!\!\!\!\!\!\!\! 3 \!\!\!\!\!  & \!\!\!\!\!\!  14 \!\!\!\!\!  & \!\!\!\!\!\!   16.41 \!\!\!\!\! 
			& \!\!  -14.730 \!\!\!\!\! & \!\!\!\!\!\!  1.59 \!\!\!\!\!  & \!\!\!\!  
			\!  -16.313 \!\!\! & \!\!\!  -14.968 \!\!\! & \!\!\!\!\!\!  0.00 \!\!\!\!\! & \!\!\!\!
			\!  -76.13 \!\!\! & \!\!\!  -14.871 \!\!\! & \!\!\!\!\!\!  0.65 \!\!\!\!\!\!\!\!\!\\ 
			\!\!\!\! 1437 \!\!\!\! & \!\!
			\!\!\! 5e-1 \!\!\!\!\! & \!\!\!\!\!\!\!\! 4 \!\!\!\!\!  & \!\!\!\!\!\!  8  \!\!\!\!\!  & \!\!\!\!\!\!  21.62  \!\!\!\!\! 
			& \!\!  -7.788 \!\!\!\!\! & \!\!\!\!\!\!  0.02 \!\!\!\!\!  & \!\!\!\!  
			\!  -9.601 \!\!\! & \!\!\! -7.789 \!\!\! & \!\!\!\!\!\!  0.00 \!\!\!\!\! & \!\!\!\!
			\!  -87.58 \!\!\! & \!\!\!  -7.789 \!\!\! & \!\!\!\!\!\! 0.00 \!\!\!\!\!\!\!\!\!\\ 
			\!\!\!\! 1451 \!\!\!\! & \!\!
			\!\!\! 2e+1 \!\!\!\!\! & \!\!\!\!\!\!\!\! 2 \!\!\!\!\!  & \!\!\!\!\!\!  36 \!\!\!\!\!  & \!\!\!\!\!\!  100.50 \!\!\!\!\! 
			& \!\!  -87.502 \!\!\!\!\! & \!\!\!\!\!\!  0.09 \!\!\!\!\!  & \!\!\!\!  
			\! -135.140 \!\!\! & \!\!\! -87.577 \!\!\! & \!\!\!\!\!\!  0.00 \!\!\!\!\! & \!\!\!\!
			\! -468.04 \!\!\! & \!\!\!  -87.283  \!\!\! & \!\!\!\!\!\!   0.34 \!\!\!\!\!\!\!\!\!\\ 
			\!\!\!\! 1493 \!\!\!\! & \!\!
			\!\!\! 1e+1 \!\!\!\!\! & \!\!\!\!\!\!\!\! 3 \!\!\!\!\!  & \!\!\!\!\!\!  13 \!\!\!\!\!  & \!\!\!\!\!\!  13.69  \!\!\!\!\! 
			& \!\!  -41.804 \!\!\!\!\! & \!\!\!\!\!\!  3.14 \!\!\!\!\!  & \!\!\!\!  
			\!  -47.239 \!\!\! & \!\!\! -43.160 \!\!\! & \!\!\!\!\!\!  0.00 \!\!\!\!\! & \!\!\!\!
			\!  -395.69 \!\!\! & \!\!\! -43.160 \!\!\! & \!\!\!\!\!\!  0.00 \!\!\!\!\!\!\!\!\!\\ 
			\!\!\!\! 1507 \!\!\!\! & \!\!
			\!\!\! 1e+0 \!\!\!\!\! & \!\!\!\!\!\!\!\! 6 \!\!\!\!\!  & \!\!\!\!\!\!  13 \!\!\!\!\!  & \!\!\!\!\!\!  10.31  \!\!\!\!\! 
			& \!\!  -8.295 \!\!\!\!\! & \!\!\!\!\!\!  0.08 \!\!\!\!\!  & \!\!\!\!  
			\!  -49.709 \!\!\! & \!\!\!  -8.301 \!\!\! & \!\!\!\!\!\!  0.00 \!\!\!\!\! & \!\!\!\!
			\!  -44.37 \!\!\! & \!\!\!  -8.301 \!\!\! & \!\!\!\!\!\!  0.00 \!\!\!\!\!\!\!\!\!\\ 
			\!\!\!\! 1535 \!\!\!\! & \!\!
			\!\!\! 2e+0 \!\!\!\!\! & \!\!\!\!\!\!\!\! 3 \!\!\!\!\!  & \!\!\!\!\!\!  23 \!\!\!\!\!  & \!\!\!\!\!\!  83.47 \!\!\!\!\! 
			& \!\!  -11.241 \!\!\!\!\! & \!\!\!\!\!\!  2.98 \!\!\!\!\!  & \!\!\!\!  
			\!  -13.407 \!\!\! & \!\!\! -11.586 \!\!\! & \!\!\!\!\!\!  0.00 \!\!\!\!\! & \!\!\!\!
			\!  -107.86 \!\!\! & \!\!\! -11.398 \!\!\! & \!\!\!\!\!\!  1.62 \!\!\!\!\!\!\!\!\!\\ 
			\!\!\!\! 1619 \!\!\!\! & \!\!
			\!\!\! 2e+0 \!\!\!\!\! & \!\!\!\!\!\!\!\! 3 \!\!\!\!\!  & \!\!\!\!\!\!  20 \!\!\!\!\!  & \!\!\!\!\!\!  35.62 \!\!\!\!\! 
			& \!\!  -9.213 \!\!\!\!\! & \!\!\!\!\!\!  0.05 \!\!\!\!\!  & \!\!\!\!  
			\!  -10.302 \!\!\! & \!\!\! -9.217 \!\!\! & \!\!\!\!\!\!  0.00 \!\!\!\!\! & \!\!\!\!
			\!  -74.55 \!\!\! & \!\!\!  -9.217 \!\!\! & \!\!\!\!\!\!  0.00 \!\!\!\!\!\!\!\!\!\\ 
			\!\!\!\! 1661 \!\!\!\! & \!\!
			\!\!\! 1e+0 \!\!\!\!\! & \!\!\!\!\!\!\!\! 3 \!\!\!\!\!  & \!\!\!\!\!\!  8  \!\!\!\!\!  & \!\!\!\!\!\!  35.85 \!\!\!\!\! 
			& \!\!  -15.666 \!\!\!\!\! & \!\!\!\!\!\!  1.81 \!\!\!\!\!  & \!\!\!\!  
			\!  -19.667 \!\!\! & \!\!\!  -15.955 \!\!\! & \!\!\!\!\!\!  0.00 \!\!\!\!\! & \!\!\!\!
			\!  -139.25 \!\!\! & \!\!\!  -15.955 \!\!\! & \!\!\!\!\!\!  0.00 \!\!\!\!\!\!\!\!\!\\ 
			\!\!\!\! 1675 \!\!\!\! & \!\!
			\!\!\! 1e+1 \!\!\!\!\! & \!\!\!\!\!\!\!\! 3 \!\!\!\!\!  & \!\!\!\!\!\!  11 \!\!\!\!\!  & \!\!\!\!\!\!  41.30 \!\!\!\!\! 
			& \!\!  -75.537 \!\!\!\!\! & \!\!\!\!\!\!  0.17 \!\!\!\!\!  & \!\!\!\!  
			\! -96.864 \!\!\! & \!\!\! -75.669 \!\!\! & \!\!\!\!\!\!  0.00 \!\!\!\!\! & \!\!\!\!
			\! -435.48 \!\!\! & \!\!\!  -75.669 \!\!\! & \!\!\!\!\!\!  0.00 \!\!\!\!\!\!\!\!\!\\ 
			\!\!\!\! 1703 \!\!\!\! & \!\!
			\!\!\! 2e+1 \!\!\!\!\! & \!\!\!\!\!\!\!\! 5 \!\!\!\!\!  & \!\!\!\!\!\!  22 \!\!\!\!\!  & \!\!\!\!\!\!  62.63 \!\!\!\!\! 
			& \!\!  -131.330 \!\!\!\!\! & \!\!\!\!\!\!  1.11 \!\!\!\!\!  & \!\!\!\!  
			\!  -180.935 \!\!\! & \!\!\!  -132.802 \!\!\! & \!\!\!\!\!\!  0.00 \!\!\!\!\! & \!\!\!\!
			\!  -929.92 \!\!\! & \!\!\!  -132.802 \!\!\! & \!\!\!\!\!\!  0.00 \!\!\!\!\!\!\!\!\!\\ 
			\!\!\!\! 1745 \!\!\!\! & \!\!
			\!\!\! 5e+0 \!\!\!\!\! & \!\!\!\!\!\!\!\! 4 \!\!\!\!\!  & \!\!\!\!\!\!  19 \!\!\!\!\!  & \!\!\!\!\!\!  40.44 \!\!\!\!\! 
			& \!\!  -72.351 \!\!\!\!\! & \!\!\!\!\!\!  0.04 \!\!\!\!\!  & \!\!\!\!  
			\!  -77.465 \!\!\! & \!\!\!  -72.377 \!\!\! & \!\!\!\!\!\!  0.00 \!\!\!\!\! & \!\!\!\!
			\!  -317.99 \!\!\! & \!\!\!  -72.377 \!\!\! & \!\!\!\!\!\!  0.00 \!\!\!\!\!\!\!\!\!\\ 
			\!\!\!\! 1773 \!\!\!\! & \!\!
			\!\!\! 5e+0 \!\!\!\!\! & \!\!\!\!\!\!\!\! 1 \!\!\!\!\!  & \!\!\!\!\!\!  56 \!\!\!\!\!  & \!\!\!\!\!\!   120.65 \!\!\!\!\! 
			& \!\!  -14.176 \!\!\!\!\! & \!\!\!\!\!\!  3.19 \!\!\!\!\!  & \!\!\!\!  
			\!  -21.581 \!\!\! & \!\!\! -14.642 \!\!\! & \!\!\!\!\!\!  0.00 \!\!\!\!\! & \!\!\!\!
			\!  -118.65 \!\!\! & \!\!\! -14.642 \!\!\! & \!\!\!\!\!\!  0.00 \!\!\!\!\!\!\!\!\!\\ 
			\!\!\!\! 1886 \!\!\!\! & \!\!
			\!\!\! 2e+1 \!\!\!\!\! & \!\!\!\!\!\!\!\! 1 \!\!\!\!\!  & \!\!\!\!\!\!  35 \!\!\!\!\!  & \!\!\!\!\!\!  28.19 \!\!\!\!\! 
			& \!\!  -78.620 \!\!\!\!\! & \!\!\!\!\!\!  0.07 \!\!\!\!\!  & \!\!\!\!  
			\!  -135.615 \!\!\! & \!\!\!  -78.672 \!\!\! & \!\!\!\!\!\!  0.00 \!\!\!\!\! & \!\!\!\!
			\!  -324.87 \!\!\! & \!\!\!  -78.672 \!\!\! & \!\!\!\!\!\!  0.00 \!\!\!\!\!\!\!\!\!\\ 
			\!\!\!\! 1913 \!\!\!\! & \!\!
			\!\!\! 5e+0 \!\!\!\!\! & \!\!\!\!\!\!\!\! 4 \!\!\!\!\!  & \!\!\!\!\!\!  18 \!\!\!\!\!  & \!\!\!\!\!\!  15.10 \!\!\!\!\! 
			& \!\!  -51.879 \!\!\!\!\! & \!\!\!\!\!\!  0.44 \!\!\!\!\!  & \!\!\!\!  
			\!  -68.555 \!\!\! & \!\!\!  -52.109 \!\!\! & \!\!\!\!\!\!  0.00 \!\!\!\!\! & \!\!\!\!
			\!  -164.26 \!\!\! & \!\!\!  -51.348 \!\!\! & \!\!\!\!\!\!  1.46 \!\!\!\!\!\!\!\!\!\\ 
			\!\!\!\! 1922 \!\!\!\! & \!\!
			\!\!\! 1e+1 \!\!\!\!\! & \!\!\!\!\!\!\!\! 1 \!\!\!\!\!  & \!\!\!\!\!\!  26 \!\!\!\!\!  & \!\!\!\!\!\!  13.22 \!\!\!\!\! 
			& \!\!  -35.451 \!\!\!\!\! & \!\!\!\!\!\!  1.39 \!\!\!\!\!  & \!\!\!\!  
			\!  -121.872 \!\!\! & \!\!\!  -35.951 \!\!\! & \!\!\!\!\!\!  0.00 \!\!\!\!\! & \!\!\!\!
			\!  -123.2 \!\!\! & \!\!\!  -35.951 \!\!\! & \!\!\!\!\!\!  0.00 \!\!\!\!\!\!\!\!\!\\ 
			\!\!\!\! 1931 \!\!\!\! & \!\!
			\!\!\! 1e+1 \!\!\!\!\! & \!\!\!\!\!\!\!\! 1 \!\!\!\!\!  & \!\!\!\!\!\!  13 \!\!\!\!\!  & \!\!\!\!\!\!   8.59 \!\!\!\!\! 
			& \!\!  -53.709 \!\!\!\!\! & \!\!\!\!\!\!  3.59 \!\!\!\!\!  & \!\!\!\!  
			\!  -85.196 \!\!\! & \!\!\!  -55.709 \!\!\! & \!\!\!\!\!\!  0.00 \!\!\!\!\! & \!\!\!\!
			\!  -204.08 \!\!\! & \!\!\!  -54.290 \!\!\! & \!\!\!\!\!\!  2.55 \!\!\!\!\!\!\!\!\!\\ 
			\!\!\!\! 1967 \!\!\!\! & \!\!
			\!\!\! 5e+1 \!\!\!\!\! & \!\!\!\!\!\!\!\! 1 \!\!\!\!\!  & \!\!\!\!\!\!  38 \!\!\!\!\!  & \!\!\!\!\!\!  33.01 \!\!\!\!\! 
			& \!\!  -105.616 \!\!\!\!\! & \!\!\!\!\!\!  1.83 \!\!\!\!\!  & \!\!\!\!  
			\!  -136.098 \!\!\! & \!\!\!  0.000 \!\!\! & \!\!\!\!\!\!  100 \!\!\!\!\! & \!\!\!\!
			\!  -622.57 \!\!\! & \!\!\! -107.581 \!\!\! & \!\!\!\!\!\!  0.00 \!\!\!\!\!\!\!\!\!\\ 
			\hline
			\!\!\!\! Max \!\!\!\! & \!\!
			\!\!\! 100 \!\!\!\!\! & \!\!\!\!\!\!\!\! 6 \!\!\!\!\!  & \!\!\!\!\!\!  56 \!\!\!\!\!  & \!\!\!\!\!\! 120.65 \!\!\!\!\! 
			& \!\!   \!\!\!\!\! & \!\!\!\!\!\!  13.17 \!\!\!\!\! &  \!\!\!\!    
			\!    \!\!\! & \!\!\!    \!\!\! & \!\!\!\!\!\!  100 \!\!\!\!\! & \!\!\!\!
			\!    \!\!\! & \!\!\!    \!\!\! & \!\!\!\!\!\!  3.49 \!\!\!\!\!\!\!\!\!\\ 
			\hline \hline
	\end{tabular}}
	\label{tab:Seq_SOCP_RLT}
\end{table}

\begin{table}[t]
	\centering
	\caption{\small Sequential penalized SDP.}
	\scalebox{0.70}{\begin{tabular}{c|cccccc|ccc|ccc}
			\hline \hline
			\!\!\!\multirow{2}{*}{Inst}\!\!\! & 
			\multicolumn{6}{c|}{\bf Sequential SDP} & 
			\multicolumn{3}{c|}{\bf BARON} & 
			\multicolumn{3}{c}{\bf COUENNE} \\
			\cline{2-13}
			& $\eta$ & $i^{\sm{\mathrm{feas}}{6}}$ & $i^{\sm{\mathrm{stop}}{6}}$ & $t$(s) 
			& UB & GAP(\%)\!\!\! &  
			LB & UB & GAP(\%)\!\!\! & LB & UB & GAP(\%)\!\!\! \\
			\hline
			\hline
			\!\!\!\! 0343$^\ast$ \!\!\!\! & \!\!
			\!\!\! 1e+2 \!\!\!\!\! & \!\!\!\!\!\!\!\! 1 \!\!\!\!\!  & \!\!\!\!\!\!  53 \!\!\!\!\!  & \!\!\!\!\!\!  29.24 \!\!\!\!\! 
			& \!\!  -6.379 \!\!\!\!\! & \!\!\!\!\!\!  0.12 \!\!\!\!\! &  \!\!\!\!  
			\!  -95.372 \!\!\! & \!\!\!  -6.386 \!\!\! & \!\!\!\!\!\!  0.00 \!\!\!\!\! & \!\!\!\!
			\!  -7668.005 \!\!\! & \!\!\!  -6.386 \!\!\! & \!\!\!\!\!\!  0.00 \!\!\!\!\!\!\!\!\!\\ 
			\!\!\!\! 0911 \!\!\!\! & \!\!
			\!\!\! 2e+0 \!\!\!\!\! & \!\!\!\!\!\!\!\! 1 \!\!\!\!\!  & \!\!\!\!\!\!  9 \!\!\!\!\!  & \!\!\!\!\!\!  5.19 \!\!\!\!\! 
			&\!\!  -31.811 \!\!\!\!\! & \!\!\!\!\!\!  1.05 \!\!\!\!\! &  \!\!\!\!  
			\!  -172.777 \!\!\! & \!\!\!  0.000 \!\!\! & \!\!\!\!\!\!  100 \!\!\!\!\! & \!\!\!\!
			\!  -172.777 \!\!\! & \!\!\!  -31.026 \!\!\! & \!\!\!\!\!\!  3.49 \!\!\!\!\!\!\!\!\!\\ 
			\!\!\!\! 0975 \!\!\!\! & \!\!
			\!\!\! 2e+0 \!\!\!\!\! & \!\!\!\!\!\!\!\! 2 \!\!\!\!\!  & \!\!\!\!\!\!  13 \!\!\!\!\!  & \!\!\!\!\!\!  8.18 \!\!\!\!\! 
			&\!\!  -37.845 \!\!\!\!\! & \!\!\!\!\!\!  0.02 \!\!\!\!\! &  \!\!\!\! 
			\!  -47.428 \!\!\! & \!\!\!  -37.794 \!\!\! & \!\!\!\!\!\!  0.16 \!\!\!\!\! & \!\!\!\!
			\!  -171.113 \!\!\! & \!\!\!  -36.812 \!\!\! & \!\!\!\!\!\!  2.75 \!\!\!\!\!\!\!\!\!\\ 
			\!\!\!\! 1055 \!\!\!\! & \!\!
			\!\!\! 5e+0 \!\!\!\!\! & \!\!\!\!\!\!\!\! 1 \!\!\!\!\!  & \!\!\!\!\!\!  8  \!\!\!\!\!  & \!\!\!\!\!\!  4.36 \!\!\!\!\! 
			&\!\!  -32.528 \!\!\!\!\! & \!\!\!\!\!\!  1.54 \!\!\!\!\! &  \!\!\!\! 
			\!  -37.841 \!\!\! & \!\!\!  -33.037 \!\!\! & \!\!\!\!\!\!  0.00 \!\!\!\!\! & \!\!\!\!
			\!  -199.457 \!\!\! & \!\!\!  -33.037 \!\!\! & \!\!\!\!\!\!  0.00 \!\!\!\!\!\!\!\!\!\\ 
			\!\!\!\! 1143 \!\!\!\! & \!\!
			\!\!\! 5e+0 \!\!\!\!\! & \!\!\!\!\!\!\!\! 4 \!\!\!\!\!  & \!\!\!\!\!\!  15 \!\!\!\!\!  & \!\!\!\!\!\!  7.89 \!\!\!\!\! 
			&\!\!  -55.606 \!\!\!\!\! & \!\!\!\!\!\!  2.87 \!\!\!\!\! &  \!\!\!\!  
			\!  -69.522 \!\!\! & \!\!\!  -57.247 \!\!\! & \!\!\!\!\!\!  0.00 \!\!\!\!\! & \!\!\!\!
			\!  -384.45 \!\!\! & \!\!\!  -53.367 \!\!\! & \!\!\!\!\!\!  6.78 \!\!\!\!\!\!\!\!\!\\ 
			\!\!\!\! 1157 \!\!\!\! & \!\!
			\!\!\! 1e+0 \!\!\!\!\! & \!\!\!\!\!\!\!\! 1 \!\!\!\!\!  & \!\!\!\!\!\!  5  \!\!\!\!\!  & \!\!\!\!\!\!  3.15 \!\!\!\!\! 
			&\!\!  -10.945 \!\!\!\!\! & \!\!\!\!\!\!  0.03 \!\!\!\!\! &  \!\!\!\!  
			\!  -11.414 \!\!\! & \!\!\!  -10.948 \!\!\! & \!\!\!\!\!\!  0.00 \!\!\!\!\! & \!\!\!\!
			\!  -80.51 \!\!\! & \!\!\!  -10.948 \!\!\! & \!\!\!\!\!\!  0.00 \!\!\!\!\!\!\!\!\!\\ 
			\!\!\!\! 1353$^\ast$ \!\!\!\! & \!\!
			\!\!\! 1e+0 \!\!\!\!\! & \!\!\!\!\!\!\!\! 1 \!\!\!\!\!  & \!\!\!\!\!\!  10 \!\!\!\!\!  & \!\!\!\!\!\!  6.12 \!\!\!\!\! 
			&\!\!  -7.712 \!\!\!\!\! & \!\!\!\!\!\!  0.03 \!\!\!\!\! &  \!\!\!\!  
			\!  -7.925 \!\!\! & \!\!\!  -7.714 \!\!\! & \!\!\!\!\!\!  0.00 \!\!\!\!\! & \!\!\!\!
			\!  -73.28 \!\!\! & \!\!\!  -7.714 \!\!\! & \!\!\!\!\!\!  0.00 \!\!\!\!\!\!\!\!\!\\ 
			\!\!\!\! 1423$^\ast$ \!\!\!\! & \!\!
			\!\!\! 1e+0 \!\!\!\!\! & \!\!\!\!\!\!\!\! 1 \!\!\!\!\!  & \!\!\!\!\!\!  5  \!\!\!\!\!  & \!\!\!\!\!\!  3.28 \!\!\!\!\! 
			&\!\!  -14.676 \!\!\!\!\! & \!\!\!\!\!\!  1.95 \!\!\!\!\! &  \!\!\!\!  
			\!  -16.313 \!\!\! & \!\!\! -14.968 \!\!\! & \!\!\!\!\!\!  0.00 \!\!\!\!\! & \!\!\!\!
			\!  -76.13 \!\!\! & \!\!\!  -14.078 \!\!\! & \!\!\!\!\!\!  5.94 \!\!\!\!\!\!\!\!\!\\ 
			\!\!\!\! 1437$^\ast$ \!\!\!\! & \!\!
			\!\!\! 1e+0 \!\!\!\!\! & \!\!\!\!\!\!\!\! 1 \!\!\!\!\!  & \!\!\!\!\!\!  7  \!\!\!\!\!  & \!\!\!\!\!\!  4.30 \!\!\!\!\!
			&\!\!  -7.787 \!\!\!\!\! & \!\!\!\!\!\!  0.03 \!\!\!\!\! &  
			\!\!\!\!  
			\!  -9.601 \!\!\! & \!\!\!  -7.789 \!\!\! & \!\!\!\!\!\!  0.00 \!\!\!\!\! & \!\!\!\!
			\!  -87.58 \!\!\! & \!\!\!  -7.789 \!\!\! & \!\!\!\!\!\!  0.00 \!\!\!\!\!\!\!\!\!\\ 
			\!\!\!\! 1451$^\dagger$ \!\!\!\! & \!\!
			\!\!\! 5e+0 \!\!\!\!\! & \!\!\!\!\!\!\!\! 2 \!\!\!\!\!  & \!\!\!\!\!\!  6 \!\!\!\!\!  & \!\!\!\!\!\! 5.09 \!\!\!\!\! 
			&\!\!  -85.972 \!\!\!\!\! & \!\!\!\!\!\!  1.83 \!\!\!\!\! &  \!\!\!\!  
			\!  -135.140 \!\!\! & \!\!\!  - \!\!\! & \!\!\!\!\!\!  - \!\!\!\!\! & \!\!\!\!
			\!  -468.04 \!\!\! & \!\!\!  - \!\!\! & \!\!\!\!\!\!  - \!\!\!\!\!\!\!\!\!\\ 
			\!\!\!\! 1493$^\ast$ \!\!\!\! & \!\!
			\!\!\! 5e+0 \!\!\!\!\! & \!\!\!\!\!\!\!\! 1 \!\!\!\!\!  & \!\!\!\!\!\!  6  \!\!\!\!\!  & \!\!\!\!\!\!  4.10 \!\!\!\!\! 
			& \!\!  -43.160 \!\!\!\!\! & \!\!\!\!\!\!  0.00 \!\!\!\!\! &  \!\!\!\!  
			\!  -47.239 \!\!\! & \!\!\!  -43.160 \!\!\! & \!\!\!\!\!\!  0.00 \!\!\!\!\! & \!\!\!\!
			\!  -395.69 \!\!\! & \!\!\!  -43.160 \!\!\! & \!\!\!\!\!\!  0.00 \!\!\!\!\!\!\!\!\!\\ 
			\!\!\!\! 1507 \!\!\!\! & \!\!
			\!\!\! 5e-1 \!\!\!\!\! & \!\!\!\!\!\!\!\! 3 \!\!\!\!\!  & \!\!\!\!\!\!  6  \!\!\!\!\!  & \!\!\!\!\!\!  3.28 \!\!\!\!\! 
			&\!\!  -8.291 \!\!\!\!\! & \!\!\!\!\!\!  0.12 \!\!\!\!\! & 
			\!\!\!\!
			\!  -49.709 \!\!\! & \!\!\!  -8.301 \!\!\! & \!\!\!\!\!\!  0.00 \!\!\!\!\! & \!\!\!\!
			\!  -44.37 \!\!\! & \!\!\!  -8.301 \!\!\! & \!\!\!\!\!\!  0.00 \!\!\!\!\!\!\!\!\!\\ 
			\!\!\!\! 1535 \!\!\!\! & \!\!
			\!\!\! 1e+0 \!\!\!\!\! & \!\!\!\!\!\!\!\! 1 \!\!\!\!\!  & \!\!\!\!\!\!  16 \!\!\!\!\!  & \!\!\!\!\!\!  13.05 \!\!\!\!\!
			& \!\! -11.363 \!\!\!\!\! & \!\!\!\!\!\!  1.93 \!\!\!\!\! &  \!\!\!\!
			\!  -13.407 \!\!\! & \!\!\!  -11.397 \!\!\! & \!\!\!\!\!\!  1.63 \!\!\!\!\! & \!\!\!\!
			\!  -107.86 \!\!\! & \!\!\!  -11.398 \!\!\! & \!\!\!\!\!\!  1.63 \!\!\!\!\!\!\!\!\!\\ 
			\!\!\!\! 1619$^\ast$ \!\!\!\! & \!\!
			\!\!\! 1e+0 \!\!\!\!\! & \!\!\!\!\!\!\!\! 1 \!\!\!\!\!  & \!\!\!\!\!\!  7 \!\!\!\!\!  & \!\!\!\!\!\!  4.64  \!\!\!\!\! 
			&\!\!  -9.213 \!\!\!\!\! & \!\!\!\!\!\!  0.05 \!\!\!\!\! &  
			\!\!\!\!
			\!  -10.302 \!\!\! & \!\!\!  -9.217 \!\!\! & \!\!\!\!\!\!  0.00 \!\!\!\!\! & \!\!\!\!
			\!  -74.55 \!\!\! & \!\!\!  -9.217 \!\!\! & \!\!\!\!\!\!  0.00 \!\!\!\!\!\!\!\!\!\\ 
			\!\!\!\! 1661$^\ast$ \!\!\!\! & \!\!
			\!\!\! 1e+0 \!\!\!\!\! & \!\!\!\!\!\!\!\! 1 \!\!\!\!\!  & \!\!\!\!\!\!  12 \!\!\!\!\!  & \!\!\!\!\!\!  7.57 \!\!\!\!\! 
			&\!\!  -15.955 \!\!\!\!\! & \!\!\!\!\!\!  0.00 \!\!\!\!\! &  \!\!\!\!
			\!  -19.667 \!\!\! & \!\!\! -15.955 \!\!\! & \!\!\!\!\!\!  0.00 \!\!\!\!\! & \!\!\!\!
			\!  -139.25 \!\!\! & \!\!\!  -15.955 \!\!\! & \!\!\!\!\!\!  0.00 \!\!\!\!\!\!\!\!\!\\ 
			\!\!\!\! 1675$^\ast$ \!\!\!\! & \!\!
			\!\!\! 5e+0 \!\!\!\!\! & \!\!\!\!\!\!\!\! 1 \!\!\!\!\!  & \!\!\!\!\!\!  5  \!\!\!\!\!  & \!\!\!\!\!\!  3.75 \!\!\!\!\! 
			&\!\!  -75.550 \!\!\!\!\! & \!\!\!\!\!\!  0.16 \!\!\!\!\! & 
			\!\!\!\!
			\!  -96.864 \!\!\! & \!\!\!  -75.669 \!\!\! & \!\!\!\!\!\!  0.00 \!\!\!\!\! & \!\!\!\!
			\!  -435.48 \!\!\! & \!\!\!  -75.669 \!\!\! & \!\!\!\!\!\!  0.00 \!\!\!\!\!\!\!\!\!\\ 
			\!\!\!\! 1703$^\dagger$ \!\!\!\! & \!\!
			\!\!\! 1e+1 \!\!\!\!\! & \!\!\!\!\!\!\!\! 1 \!\!\!\!\!  & \!\!\!\!\!\!  10 \!\!\!\!\!  & \!\!\!\!\!\!  6.96 \!\!\!\!\! 
			&\!\!  -132.539 \!\!\!\!\! & \!\!\!\!\!\!  0.20 \!\!\!\!\! &  \!\!\!\!
			\!  -180.935 \!\!\! & \!\!\!  -131.466 \!\!\! & \!\!\!\!\!\!  1.01 \!\!\!\!\! & \!\!\!\!
			\!  -929.92 \!\!\! & \!\!\!  - \!\!\! & \!\!\!\!\!\!  - \!\!\!\!\!\!\!\!\!\\ 
			\!\!\!\! 1745 \!\!\!\! & \!\!
			\!\!\! 5e+0 \!\!\!\!\! & \!\!\!\!\!\!\!\! 1 \!\!\!\!\!  & \!\!\!\!\!\!  8  \!\!\!\!\!  & \!\!\!\!\!\!  4.75 \!\!\!\!\! 
			&\!\! -71.828 \!\!\!\!\! & \!\!\!\!\!\!  0.76 \!\!\!\!\! &  
			\!\!\!\!
			\!  -77.465 \!\!\! & \!\!\!  -72.377 \!\!\! & \!\!\!\!\!\!  0.00 \!\!\!\!\! & \!\!\!\!
			\!  -317.99 \!\!\! & \!\!\!  -72.377 \!\!\! & \!\!\!\!\!\!  0.00 \!\!\!\!\!\!\!\!\!\\ 
			\!\!\!\! 1773$^\ast$ \!\!\!\! & \!\!
			\!\!\! 1e+0 \!\!\!\!\! & \!\!\!\!\!\!\!\! 1 \!\!\!\!\!  & \!\!\!\!\!\!  8  \!\!\!\!\!  & \!\!\!\!\!\!  5.44 \!\!\!\!\! 
			&\!\! -14.633 \!\!\!\!\! & \!\!\!\!\!\!  0.06 \!\!\!\!\! &  \!\!\!\! 
			\!  -21.581 \!\!\! & \!\!\!  -14.642 \!\!\! & \!\!\!\!\!\!  0.00 \!\!\!\!\! & \!\!\!\!
			\!  -118.65 \!\!\! & \!\!\!  -14.636 \!\!\! & \!\!\!\!\!\!  0.04 \!\!\!\!\!\!\!\!\!\\ 
			\!\!\!\! 1886 \!\!\!\! & \!\!
			\!\!\! 5e+0 \!\!\!\!\! & \!\!\!\!\!\!\!\! 2 \!\!\!\!\!  & \!\!\!\!\!\!  9 \!\!\!\!\!  & \!\!\!\!\!\!   5.84 \!\!\!\!\! 
			&\!\!  -78.659 \!\!\!\!\! & \!\!\!\!\!\!  0.02 \!\!\!\!\! & 
			\!\!\!\!
			\!  -135.615 \!\!\! & \!\!\!   -49.684 \!\!\! & \!\!\!\!\!\!  36.84 \!\!\!\!\! & \!\!\!\!
			\!  -324.87 \!\!\! & \!\!\!  -78.672 \!\!\! & \!\!\!\!\!\!  0.00 \!\!\!\!\!\!\!\!\!\\ 
			\!\!\!\! 1913 \!\!\!\! & \!\!
			\!\!\! 5e+0 \!\!\!\!\! & \!\!\!\!\!\!\!\! 1 \!\!\!\!\!  & \!\!\!\!\!\!  20 \!\!\!\!\!  & \!\!\!\!\!\!  12.48 \!\!\!\!\!
			&\!\!  -51.866 \!\!\!\!\! & \!\!\!\!\!\!  0.47 \!\!\!\!\! &  \!\!\!\!
			\!  -68.555 \!\!\! & \!\!\!  -52.109 \!\!\! & \!\!\!\!\!\!  0.00 \!\!\!\!\! & \!\!\!\!
			\!  -164.26 \!\!\! & \!\!\!  -51.348 \!\!\! & \!\!\!\!\!\!  1.46 \!\!\!\!\!\!\!\!\!\\ 
			\!\!\!\! 1922$^\ast$ \!\!\!\! & \!\!
			\!\!\! 5e+0 \!\!\!\!\! & \!\!\!\!\!\!\!\! 1 \!\!\!\!\!  & \!\!\!\!\!\!  7 \!\!\!\!\!  & \!\!\!\!\!\!  4.34 \!\!\!\!\! 
			&\!\!  -35.452 \!\!\!\!\! & \!\!\!\!\!\!  1.39 \!\!\!\!\! &  
			\!\!\!\!
			\!  -121.872 \!\!\! & \!\!\! -35.916 \!\!\! & \!\!\!\!\!\!  0.10 \!\!\!\!\! & \!\!\!\!
			\!  -123.2 \!\!\! & \!\!\!  -35.951 \!\!\! & \!\!\!\!\!\!  0.00 \!\!\!\!\!\!\!\!\!\\ 
			\!\!\!\! 1931 \!\!\!\! & \!\!
			\!\!\! 5e+0 \!\!\!\!\! & \!\!\!\!\!\!\!\! 1 \!\!\!\!\!  & \!\!\!\!\!\!  10 \!\!\!\!\!  & \!\!\!\!\!\!  5.87 \!\!\!\!\! 
			&\!\!  -54.894 \!\!\!\!\! & \!\!\!\!\!\!  1.46 \!\!\!\!\! &  \!\!\!\!
			\!  -85.196 \!\!\! & \!\!\!  -55.709 \!\!\! & \!\!\!\!\!\!  0.00 \!\!\!\!\! & \!\!\!\!
			\!  -204.08 \!\!\! & \!\!\!  -54.290 \!\!\! & \!\!\!\!\!\!  2.55 \!\!\!\!\!\!\!\!\!\\ 
			\!\!\!\! 1967 \!\!\!\! & \!\!
			\!\!\! 1e+1 \!\!\!\!\! & \!\!\!\!\!\!\!\! 1 \!\!\!\!\!  & \!\!\!\!\!\!  6 \!\!\!\!\!  & \!\!\!\!\!\!  5.49 \!\!\!\!\! 
			&\!\!  -104.752 \!\!\!\!\! & \!\!\!\!\!\!  2.63 \!\!\!\!\! &  \!\!\!\!  
			\!  -136.098 \!\!\! & \!\!\!  0.000 \!\!\! & \!\!\!\!\!\!  100 \!\!\!\!\! & \!\!\!\!
			\!  -622.57 \!\!\! & \!\!\!  -107.581 \!\!\! & \!\!\!\!\!\!  0.00 \!\!\!\!\!\!\!\!\!\\ 
			\hline
			\!\!\!\! Max \!\!\!\! & \!\!
			\!\!\! 100 \!\!\!\!\! & \!\!\!\!\!\!\!\! 4 \!\!\!\!\!  & \!\!\!\!\!\!  53 \!\!\!\!\!  & \!\!\!\!\!\! 29.24 \!\!\!\!\! 
			& \!\!   \!\!\!\!\! & \!\!\!\!\!\!  2.87 \!\!\!\!\! &  \!\!\!\!    
			\!    \!\!\! & \!\!\!    \!\!\! & \!\!\!\!\!\!  100 \!\!\!\!\! & \!\!\!\!
			\!    \!\!\! & \!\!\!    \!\!\! & \!\!\!\!\!\!  6.78 \!\!\!\!\!\!\!\!\!\\ 
			\hline \hline
	\end{tabular}}
	\begin{flushleft}
		{\scriptsize $^\dagger$ Rows 1451 and 1703 are excluded from 
			maximum computations due to missing entries.}\\
		{\scriptsize $^\ast$ $i^{\mathrm{feas}}=1$ is predicted by Theorem \ref{thm2}.}
	\end{flushleft}
	\label{tab:Seq_SDP}
\end{table}

\begin{table}[t]	
	\centering
	\caption{\small Sequential penalized SDP+RLT.}
	\scalebox{0.70}{\begin{tabular}{c|cccccc|ccc|ccc}
			\hline \hline
			\!\!\!\multirow{2}{*}{Inst}\!\!\! & 
			\multicolumn{6}{c|}{\bf Sequential SDP+RLT} & 
			\multicolumn{3}{c|}{\bf BARON} & 
			\multicolumn{3}{c}{\bf COUENNE} \\
			\cline{2-13}
			& $\eta$ & $i^{\sm{\mathrm{feas}}{6}}$ & $i^{\sm{\mathrm{stop}}{6}}$ & $t$(s) 
			& UB & GAP(\%)\!\!\!&
			LB & UB & GAP(\%)\!\!\! & LB & UB & GAP(\%)\!\!\! \\
			\hline
			\hline
			\!\!\!\! 0343 \!\!\!\! & \!\!
			\!\!\! 0e+0 \!\!\!\!\! & \!\!\!\!\!\!\!\! 0 \!\!\!\!\!  & \!\!\!\!\!\!  0 \!\!\!\!\!  & \!\!\!\!\!\!  1.42 \!\!\!\!\! 
			&\!\!  -6.386 \!\!\!\!\! & \!\!\!\!\!\!  0.00 \!\!\!\!\! &  \!\!\!\!  
			\!  -95.372 \!\!\! & \!\!\!  -6.386 \!\!\! & \!\!\!\!\!\!  0.00 \!\!\!\!\! & \!\!\!\!
			\!  -7668.005 \!\!\! & \!\!\!  -6.386 \!\!\! & \!\!\!\!\!\!  0.00 \!\!\!\!\!\!\!\!\!\\ 
			\!\!\!\! 0911 \!\!\!\! & \!\!
			\!\!\! 2e-1 \!\!\!\!\! & \!\!\!\!\!\!\!\! 4 \!\!\!\!\!  & \!\!\!\!\!\!  5 \!\!\!\!\!  & \!\!\!\!\!\! 13.08 \!\!\!\!\! 
			& \!\!  -32.147 \!\!\!\!\! & \!\!\!\!\!\!  0.00 \!\!\!\!\! &  \!\!\!\!  
			\!  -172.777 \!\!\! & \!\!\!  0.000 \!\!\! & \!\!\!\!\!\!  100 \!\!\!\!\! & \!\!\!\!
			\!  -172.777 \!\!\! & \!\!\!  -31.026 \!\!\! & \!\!\!\!\!\!  3.49 \!\!\!\!\!\!\!\!\!\\ 
			\!\!\!\! 0975 \!\!\!\! & \!\!
			\!\!\! 2e-1 \!\!\!\!\! & \!\!\!\!\!\!\!\! 3 \!\!\!\!\!  & \!\!\!\!\!\!  5 \!\!\!\!\!  & \!\!\!\!\!\!  12.75 \!\!\!\!\! 
			& \!\!  -37.852 \!\!\!\!\! & \!\!\!\!\!\!  0.00 \!\!\!\!\! &  \!\!\!\!  
			\! -47.428 \!\!\! & \!\!\!  -37.794 \!\!\! & \!\!\!\!\!\!  0.16 \!\!\!\!\! & \!\!\!\!
			\! -171.113 \!\!\! & \!\!\!  -36.812 \!\!\! & \!\!\!\!\!\!  2.75 \!\!\!\!\!\!\!\!\!\\ 
			\!\!\!\! 1055 \!\!\!\! & \!\!
			\!\!\! 1e+0 \!\!\!\!\! & \!\!\!\!\!\!\!\! 5 \!\!\!\!\!  & \!\!\!\!\!\!  8 \!\!\!\!\!  & \!\!\!\!\!\!  9.56 \!\!\!\!\! 
			& \!\!  -32.874 \!\!\!\!\! & \!\!\!\!\!\!  0.49 \!\!\!\!\! &  \!\!\!\!  
			\!  -37.841 \!\!\! & \!\!\!  -33.037 \!\!\! & \!\!\!\!\!\!  0.00 \!\!\!\!\! & \!\!\!\!
			\!  -199.457 \!\!\! & \!\!\!  -33.037 \!\!\! & \!\!\!\!\!\!  0.00 \!\!\!\!\!\!\!\!\!\\ 
			\!\!\!\! 1143 \!\!\!\! & \!\!
			\!\!\! 5e-1 \!\!\!\!\! & \!\!\!\!\!\!\!\! 4 \!\!\!\!\!  & \!\!\!\!\!\!  5 \!\!\!\!\!  & \!\!\!\!\!\!   7.27  \!\!\!\!\! 
			& \!\!  -57.241 \!\!\!\!\! & \!\!\!\!\!\!  0.01 \!\!\!\!\! &  \!\!\!\!  
			\!  -69.522 \!\!\! & \!\!\!  -57.247 \!\!\! & \!\!\!\!\!\!  0.00 \!\!\!\!\! & \!\!\!\!
			\!  -384.45 \!\!\! & \!\!\!  -53.367 \!\!\! & \!\!\!\!\!\!  6.78 \!\!\!\!\!\!\!\!\!\\ 
			\!\!\!\! 1157 \!\!\!\! & \!\!
			\!\!\! 0e+0 \!\!\!\!\! & \!\!\!\!\!\!\!\! 0 \!\!\!\!\!  & \!\!\!\!\!\!  0 \!\!\!\!\!  & \!\!\!\!\!\!  0.88 \!\!\!\!\! 
			&\!\!  -10.948 \!\!\!\!\! & \!\!\!\!\!\!  0.00 \!\!\!\!\! &  \!\!\!\!  
			\!  -11.414 \!\!\! & \!\!\!  -10.948 \!\!\! & \!\!\!\!\!\!  0.00 \!\!\!\!\! & \!\!\!\!
			\!  -80.51 \!\!\! & \!\!\!  -10.948 \!\!\! & \!\!\!\!\!\!  0.00 \!\!\!\!\!\!\!\!\!\\ 
			\!\!\!\! 1353 \!\!\!\! & \!\!
			\!\!\! 0e+0 \!\!\!\!\! & \!\!\!\!\!\!\!\! 0 \!\!\!\!\!  & \!\!\!\!\!\!  0 \!\!\!\!\!  & \!\!\!\!\!\!  0.45 \!\!\!\!\! 
			&\!\!  -7.714 \!\!\!\!\! & \!\!\!\!\!\!  0.00 \!\!\!\!\! &  \!\!\!\!  
			\!  -7.925 \!\!\! & \!\!\!  -7.714 \!\!\! & \!\!\!\!\!\!  0.00 \!\!\!\!\! & \!\!\!\!
			\!  -73.28 \!\!\! & \!\!\!  -7.714 \!\!\! & \!\!\!\!\!\!  0.00 \!\!\!\!\!\!\!\!\!\\ 
			\!\!\!\! 1423 \!\!\!\! & \!\!
			\!\!\! 2e-1 \!\!\!\!\! & \!\!\!\!\!\!\!\! 1 \!\!\!\!\!  & \!\!\!\!\!\!  2 \!\!\!\!\!  & \!\!\!\!\!\!  2.82 \!\!\!\!\! 
			&\!\!  -14.929 \!\!\!\!\! & \!\!\!\!\!\!  0.25 \!\!\!\!\! &  \!\!\!\!  
			\!  -16.313 \!\!\! & \!\!\!  -14.968 \!\!\! & \!\!\!\!\!\!  0.00 \!\!\!\!\! & \!\!\!\!
			\!  -76.13 \!\!\! & \!\!\!  -14.078 \!\!\! & \!\!\!\!\!\!  5.94 \!\!\!\!\!\!\!\!\!\\ 
			\!\!\!\! 1437 \!\!\!\! & \!\!
			\!\!\! 1e-2 \!\!\!\!\! & \!\!\!\!\!\!\!\! 1 \!\!\!\!\!  & \!\!\!\!\!\!  2 \!\!\!\!\!  & \!\!\!\!\!\!  7.02 \!\!\!\!\! 
			&\!\!  -7.789 \!\!\!\!\! & \!\!\!\!\!\!  0.00 \!\!\!\!\! & \!\!\!\!  
			\!  -9.601 \!\!\! & \!\!\!  -7.789 \!\!\! & \!\!\!\!\!\!  0.00 \!\!\!\!\! & \!\!\!\!
			\!  -87.58 \!\!\! & \!\!\!  -7.789 \!\!\! & \!\!\!\!\!\!  0.00 \!\!\!\!\!\!\!\!\!\\ 
			\!\!\!\! 1451 \!\!\!\! & \!\!
			\!\!\! 2e+0 \!\!\!\!\! & \!\!\!\!\!\!\!\! 2 \!\!\!\!\!  & \!\!\!\!\!\!  5 \!\!\!\!\!  & \!\!\!\!\!\!  24.45 \!\!\!\!\! 
			&\!\!  -87.573 \!\!\!\!\! & \!\!\!\!\!\!  0.01 \!\!\!\!\! &  \!\!\!\!  
			\!  -135.140 \!\!\! & \!\!\!  -87.577 \!\!\! & \!\!\!\!\!\!  0.00 \!\!\!\!\! & \!\!\!\!
			\!  -468.04 \!\!\! & \!\!\!  -86.860  \!\!\! & \!\!\!\!\!\!  0.82 \!\!\!\!\!\!\!\!\!\\ 
			\!\!\!\! 1493 \!\!\!\! & \!\!
			\!\!\! 5e-1 \!\!\!\!\! & \!\!\!\!\!\!\!\! 1 \!\!\!\!\!  & \!\!\!\!\!\!  2 \!\!\!\!\!  & \!\!\!\!\!\!  2.76  \!\!\!\!\! 
			&\!\!  -43.160 \!\!\!\!\! & \!\!\!\!\!\!  0.00 \!\!\!\!\! &  \!\!\!\!  
			\!  -47.239 \!\!\! & \!\!\!  -43.160 \!\!\! & \!\!\!\!\!\!  0.00 \!\!\!\!\! & \!\!\!\!
			\!  -395.69 \!\!\! & \!\!\!  -43.160 \!\!\! & \!\!\!\!\!\!  0.00 \!\!\!\!\!\!\!\!\!\\ 
			\!\!\!\! 1507 \!\!\!\! & \!\!
			\!\!\! 0e+0 \!\!\!\!\! & \!\!\!\!\!\!\!\! 0 \!\!\!\!\!  & \!\!\!\!\!\!  0 \!\!\!\!\!  & \!\!\!\!\!\!  0.61 \!\!\!\!\! 
			&\!\!  -8.301 \!\!\!\!\! & \!\!\!\!\!\!  0.00 \!\!\!\!\! &  \!\!\!\!  
			\!  -49.709 \!\!\! & \!\!\!  -8.301 \!\!\! & \!\!\!\!\!\!  0.00 \!\!\!\!\! & \!\!\!\!
			\!  -44.37 \!\!\! & \!\!\!  -8.301 \!\!\! & \!\!\!\!\!\!  0.00 \!\!\!\!\!\!\!\!\!\\ 
			\!\!\!\! 1535 \!\!\!\! & \!\!
			\!\!\! 5e-1 \!\!\!\!\! & \!\!\!\!\!\!\!\! 1 \!\!\!\!\!  & \!\!\!\!\!\!  10 \!\!\!\!\!  & \!\!\!\!\!\! 38.01 \!\!\!\!\! 
			&\!\!  -11.536 \!\!\!\!\! & \!\!\!\!\!\!  0.43 \!\!\!\!\! &  \!\!\!\!  
			\!  -13.407 \!\!\! & \!\!\!  -11.397 \!\!\! & \!\!\!\!\!\!  1.63 \!\!\!\!\! & \!\!\!\!
			\!  -107.86 \!\!\! & \!\!\!  -11.398 \!\!\! & \!\!\!\!\!\!  1.62 \!\!\!\!\!\!\!\!\!\\ 
			\!\!\!\! 1619 \!\!\!\! & \!\!
			\!\!\! 0e+0 \!\!\!\!\! & \!\!\!\!\!\!\!\! 0 \!\!\!\!\!  & \!\!\!\!\!\!  0 \!\!\!\!\!  & \!\!\!\!\!\!  2.38 \!\!\!\!\! 
			&\!\!  -9.217 \!\!\!\!\! & \!\!\!\!\!\!  0.00 \!\!\!\!\! &  \!\!\!\!  
			\!  -10.302 \!\!\! & \!\!\!  -9.217 \!\!\! & \!\!\!\!\!\!  0.00 \!\!\!\!\! & \!\!\!\!
			\!  -74.55 \!\!\! & \!\!\!  -9.217 \!\!\! & \!\!\!\!\!\!  0.00 \!\!\!\!\!\!\!\!\!\\ 
			\!\!\!\! 1661 \!\!\!\! & \!\!
			\!\!\! 1e-1 \!\!\!\!\! & \!\!\!\!\!\!\!\! 1 \!\!\!\!\!  & \!\!\!\!\!\!  2 \!\!\!\!\!  & \!\!\!\!\!\!  12.88 \!\!\!\!\! 
			&\!\!  -15.955 \!\!\!\!\! & \!\!\!\!\!\!  0.00 \!\!\!\!\! & \!\!\!\!  
			\!  -19.667 \!\!\! & \!\!\!  -15.955 \!\!\! & \!\!\!\!\!\!  0.00 \!\!\!\!\! & \!\!\!\!
			\!  -139.25 \!\!\! & \!\!\!  -15.955 \!\!\! & \!\!\!\!\!\!  0.00 \!\!\!\!\!\!\!\!\!\\ 
			\!\!\!\! 1675 \!\!\!\! & \!\!
			\!\!\! 5e-1 \!\!\!\!\! & \!\!\!\!\!\!\!\! 4 \!\!\!\!\!  & \!\!\!\!\!\!  0 \!\!\!\!\!  & \!\!\!\!\!\! 4.22 \!\!\!\!\! 
			&\!\!  -75.669 \!\!\!\!\! & \!\!\!\!\!\!  0.00 \!\!\!\!\! & \!\!\!\!  
			\!  -96.864 \!\!\! & \!\!\!  -75.669 \!\!\! & \!\!\!\!\!\!  0.00 \!\!\!\!\! & \!\!\!\!
			\!  -435.48 \!\!\! & \!\!\!  -75.669 \!\!\! & \!\!\!\!\!\!  0.00 \!\!\!\!\!\!\!\!\!\\ 
			\!\!\!\! 1703$^\dagger$ \!\!\!\! & \!\!
			\!\!\! 2e+0 \!\!\!\!\! & \!\!\!\!\!\!\!\! 1 \!\!\!\!\!  & \!\!\!\!\!\!  3 \!\!\!\!\!  & \!\!\!\!\!\! 13.50 \!\!\!\!\! 
			&\!\!  -72.376 \!\!\!\!\! & \!\!\!\!\!\!  0.00 \!\!\!\!\! &  \!\!\!\!  
			\!  -77.465 \!\!\! & \!\!\!  - \!\!\! & \!\!\!\!\!\!  - \!\!\!\!\! & \!\!\!\!
			\!  -317.99 \!\!\! & \!\!\!  -72.377 \!\!\! & \!\!\!\!\!\!  0.00 \!\!\!\!\!\!\!\!\!\\ 
			\!\!\!\! 1773 \!\!\!\! & \!\!
			\!\!\! 2e-1 \!\!\!\!\! & \!\!\!\!\!\!\!\! 3 \!\!\!\!\!  & \!\!\!\!\!\!  4 \!\!\!\!\!  & \!\!\!\!\!\!  18.01 \!\!\!\!\! 
			&\!\!  -14.626 \!\!\!\!\! & \!\!\!\!\!\!  0.11 \!\!\!\!\! &  \!\!\!\!  
			\!  -21.581 \!\!\! & \!\!\!  -14.642 \!\!\! & \!\!\!\!\!\!  0.00 \!\!\!\!\! & \!\!\!\!
			\!  -118.65 \!\!\! & \!\!\!  -14.636 \!\!\! & \!\!\!\!\!\!  0.04 \!\!\!\!\!\!\!\!\!\\ 
			\!\!\!\! 1886 \!\!\!\! & \!\!
			\!\!\! 2e+0 \!\!\!\!\! & \!\!\!\!\!\!\!\! 2 \!\!\!\!\!  & \!\!\!\!\!\!  4 \!\!\!\!\!  & \!\!\!\!\!\!  9.05 \!\!\!\!\! 
			&\!\!  -78.643 \!\!\!\!\! & \!\!\!\!\!\!  0.04 \!\!\!\!\! & \!\!\!\!  
			\!  -135.615 \!\!\! & \!\!\!  -78.672 \!\!\! & \!\!\!\!\!\!  0.00 \!\!\!\!\! & \!\!\!\!
			\!  -324.87 \!\!\! & \!\!\!  -78.672 \!\!\! & \!\!\!\!\!\!  0.00 \!\!\!\!\!\!\!\!\!\\ 
			\!\!\!\! 1913 \!\!\!\! & \!\!
			\!\!\! 1e+0 \!\!\!\!\! & \!\!\!\!\!\!\!\! 2 \!\!\!\!\!  & \!\!\!\!\!\!  6 \!\!\!\!\!  & \!\!\!\!\!\!  11.49 \!\!\!\!\! 
			&\!\! -52.108 \!\!\!\!\! & \!\!\!\!\!\!  0.00 \!\!\!\!\! &  \!\!\!\!  
			\!  -68.555 \!\!\! & \!\!\!  -52.109 \!\!\! & \!\!\!\!\!\!  0.00 \!\!\!\!\! & \!\!\!\!
			\!  -164.26 \!\!\! & \!\!\!  -51.348 \!\!\! & \!\!\!\!\!\!  1.46 \!\!\!\!\!\!\!\!\!\\ 
			\!\!\!\! 1922 \!\!\!\! & \!\!
			\!\!\! 2e+0 \!\!\!\!\! & \!\!\!\!\!\!\!\! 1 \!\!\!\!\!  & \!\!\!\!\!\!  5 \!\!\!\!\!  & \!\!\!\!\!\!  3.35 \!\!\!\!\! 
			&\!\!  -35.556 \!\!\!\!\! & \!\!\!\!\!\!  1.10 \!\!\!\!\! &  \!\!\!\!  
			\!  -121.872 \!\!\! & \!\!\!  -35.741 \!\!\! & \!\!\!\!\!\!  0.58 \!\!\!\!\! & \!\!\!\!
			\!  -123.2 \!\!\! & \!\!\!  -35.951 \!\!\! & \!\!\!\!\!\!  0.00 \!\!\!\!\!\!\!\!\!\\ 
			\!\!\!\! 1931 \!\!\!\! & \!\!
			\!\!\! 1e+0 \!\!\!\!\! & \!\!\!\!\!\!\!\! 1 \!\!\!\!\!  & \!\!\!\!\!\!  2 \!\!\!\!\!  & \!\!\!\!\!\! 2.99 \!\!\!\!\! 
			&\!\!  -55.674 \!\!\!\!\! & \!\!\!\!\!\!  0.06 \!\!\!\!\! &  \!\!\!\!  
			\!  -85.196 \!\!\! & \!\!\!  -53.760 \!\!\! & \!\!\!\!\!\!  3.50 \!\!\!\!\! & \!\!\!\!
			\!  -204.08 \!\!\! & \!\!\!  -54.290 \!\!\! & \!\!\!\!\!\!  2.55 \!\!\!\!\!\!\!\!\!\\ 
			\!\!\!\! 1967 \!\!\!\! & \!\!
			\!\!\! 5e+0 \!\!\!\!\! & \!\!\!\!\!\!\!\! 1 \!\!\!\!\!  & \!\!\!\!\!\!  8 \!\!\!\!\!  & \!\!\!\!\!\! 16.11 \!\!\!\!\! 
			&\!\!  -107.052 \!\!\!\!\! & \!\!\!\!\!\!  0.49 \!\!\!\!\! &  \!\!\!\!    
			\!  -136.098 \!\!\! & \!\!\!  0.000 \!\!\! & \!\!\!\!\!\!  100 \!\!\!\!\! & \!\!\!\!
			\!  -622.57 \!\!\! & \!\!\!  -107.581 \!\!\! & \!\!\!\!\!\!  0.00 \!\!\!\!\!\!\!\!\!\\ 
			\hline
			\!\!\!\! Max \!\!\!\! & \!\!
			\!\!\! 5 \!\!\!\!\! & \!\!\!\!\!\!\!\! 5 \!\!\!\!\!  & \!\!\!\!\!\!  10 \!\!\!\!\!  & \!\!\!\!\!\! 38 \!\!\!\!\! 
			& \!\!   \!\!\!\!\! & \!\!\!\!\!\!  1.1 \!\!\!\!\! &  \!\!\!\!    
			\!    \!\!\! & \!\!\!    \!\!\! & \!\!\!\!\!\!  100 \!\!\!\!\! & \!\!\!\!
			\!    \!\!\! & \!\!\!    \!\!\! & \!\!\!\!\!\!  6.78 \!\!\!\!\!\!\!\!\!\\ 
			\hline \hline
	\end{tabular}}
	\begin{flushleft}
		{\scriptsize $^\dagger$ Row 1703 is excluded from 
			maximum computations due to missing entries.}
	\end{flushleft}
	\label{tab:Seq_SDP_RLT}
\end{table}


As demonstrated in the tables, penalized $2\times 2$ SDP+RLT, SDP, and SDP+RLT have successfully obtained feasible points within $4\%$ gaps from QPLIB solutions. 
Sequential SDP requires a smaller number of rounds compared to sequential $2\times 2$ SDP to meet the stopping criterion \eqref{scri}. Using any of the relaxations, the infeasible initial points can be rounded to a feasible point with only two round of Algorithm \ref{al:alg_1} and all relaxations arrive at satisfactory gaps percentages. As demonstrated by the tables, the proposed sequential approach exhibits reasonable performance in comparison with the non-convex optimizers BARON and COUENNE.

Figures \ref{plt:rounds_1451}, shows the convergence of Algorithm \ref{al:alg_1} for cases 1507. The choice of $\eta$ for all curves are taken from the corresponding rows of the Tables 
\ref{tab:Seq_SOCP}, \ref{tab:Seq_SOCP_RLT}, \ref{tab:Seq_SDP}, and \ref{tab:Seq_SDP_RLT}.

\begin{figure}[h!]
	\captionsetup[subfigure]{position=b}
	\centering
	\includegraphics[width=0.6\textwidth]{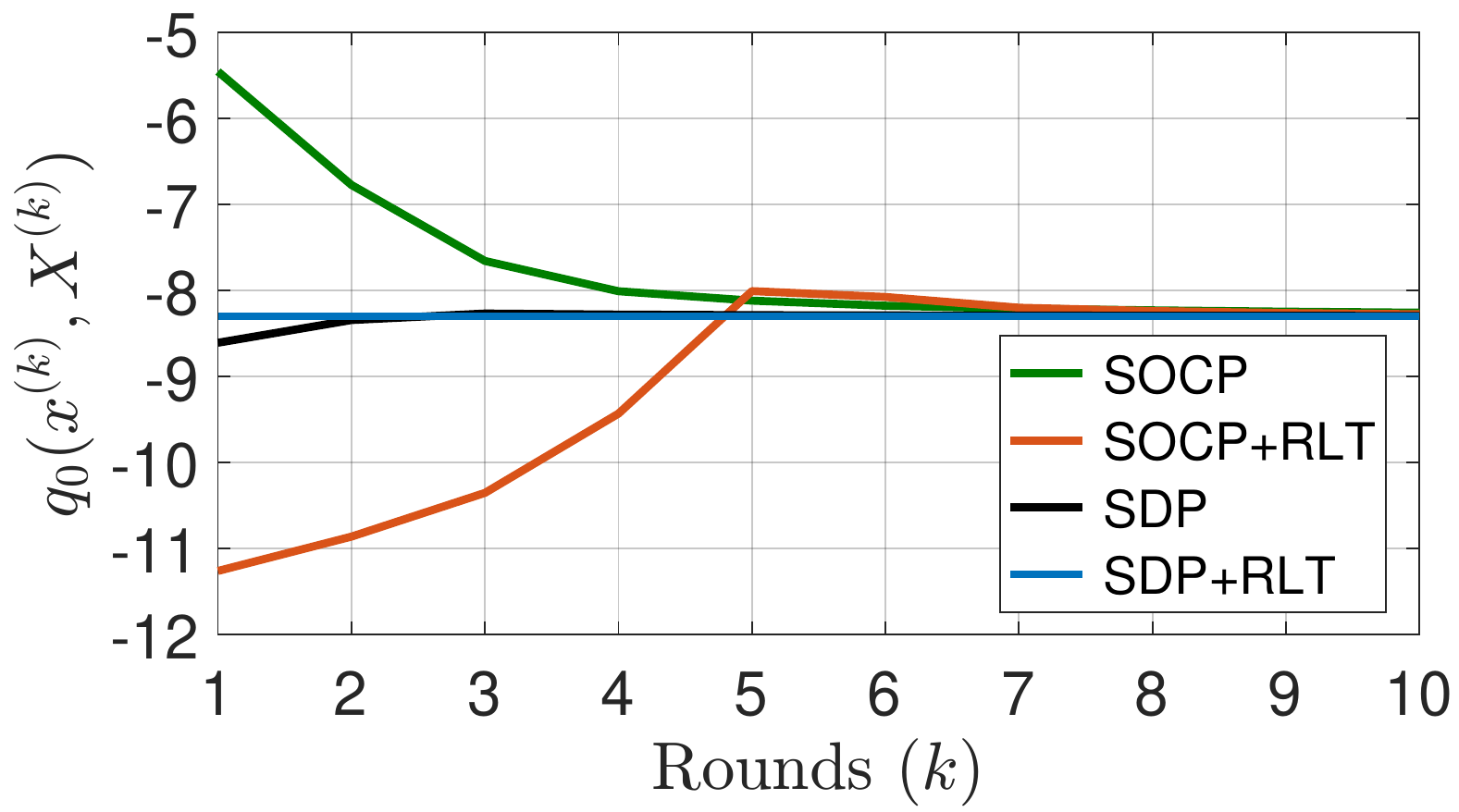}
	\caption{\small Convergence of sequential $2\times 2$ SDP, $2\times 2$ SDP+RLT, SDP, and SDP+RLT for inst. 1507.}
	\label{plt:rounds_1451}
\end{figure}

\subsubsection{Choice of the penalty parameter $\eta$}
In this experiment the sensitivity of different { penalization methods} to the choice of the penalty parameter $\eta$ is tested. To this end, one round of the penalized SDP \eqref{prob_lifted_pen_obj}-\eqref{prob_lifted_pen_conic} is solved for a wide range of $\eta$ values. The benchmark case 1143 is used for this experiment.  
If $\eta$ is small, none of the proposed penalized SDPs are tight for the case 1143. As the value of $\eta$ increases, the feasibility violation $\mathrm{tr}\{\accentset{\ast}{\boldsymbol{X}}-\accentset{\ast}{\boldsymbol{x}}\accentset{\ast}{\boldsymbol{x}}^{\top}\}$ abruptly vanishes once crossing { $\eta=1.9$, $\eta=7.7$, and $\eta=19.6$, for the penalized $2\times 2$ SDP, SDP and SDP+RLT, respectively}. 
Remarkably, if $\accentset{\ast}{\boldsymbol{x}}^{\mathrm{SDP+RLT}}$ is used as the initial point and $\eta\simeq 2$, then the penalized SDP+RLT \eqref{prob_lifted_pen_obj}-\eqref{prob_lifted_pen_conic} produces a feasible point for the benchmark case 1143 whose objective value is within 0.2\% of the reported optimal cost $q_0(\boldsymbol{x}^{\mathrm{QPLIB}})$.

{
Additionally, Figure \ref{plt:eta} shows the result of one round penalized SDP for a wide range of $\eta$ values, on cases QPLIB 1423, 1675, and 1967. As demonstrated by the figures, the resulting objective values of penalized SDP grow slowly beyond certain limits of $\eta$. This indicates that the proposed approach is not very sensitive to the choice of $\eta$ and a wide range of $\eta$ values can be used for penalization.}

\begin{figure}[h!]
	\captionsetup[subfigure]{position=b}
	\centering
	\includegraphics[width=0.5\textwidth]{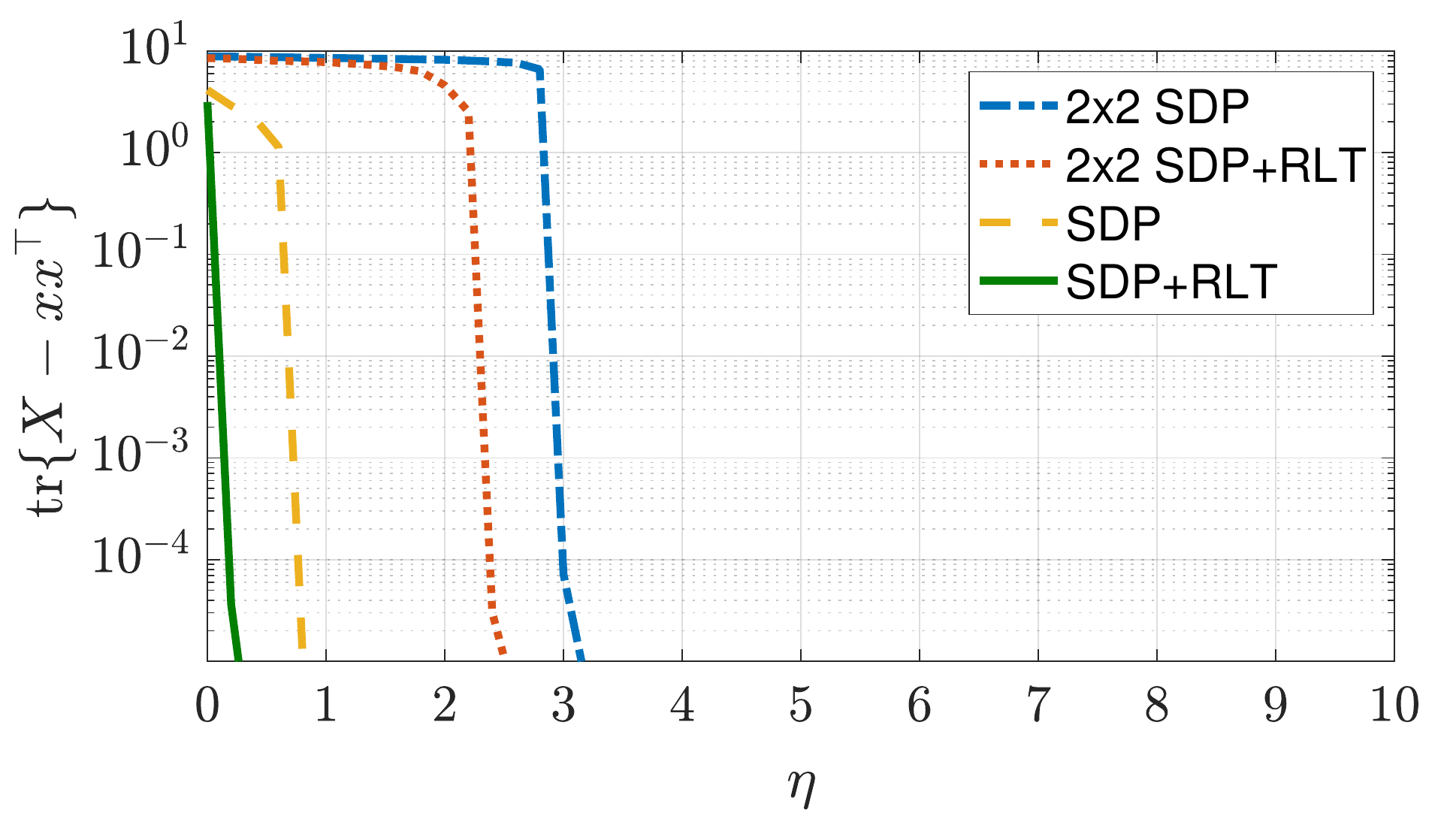}\!\!
	\includegraphics[width=0.5\textwidth]{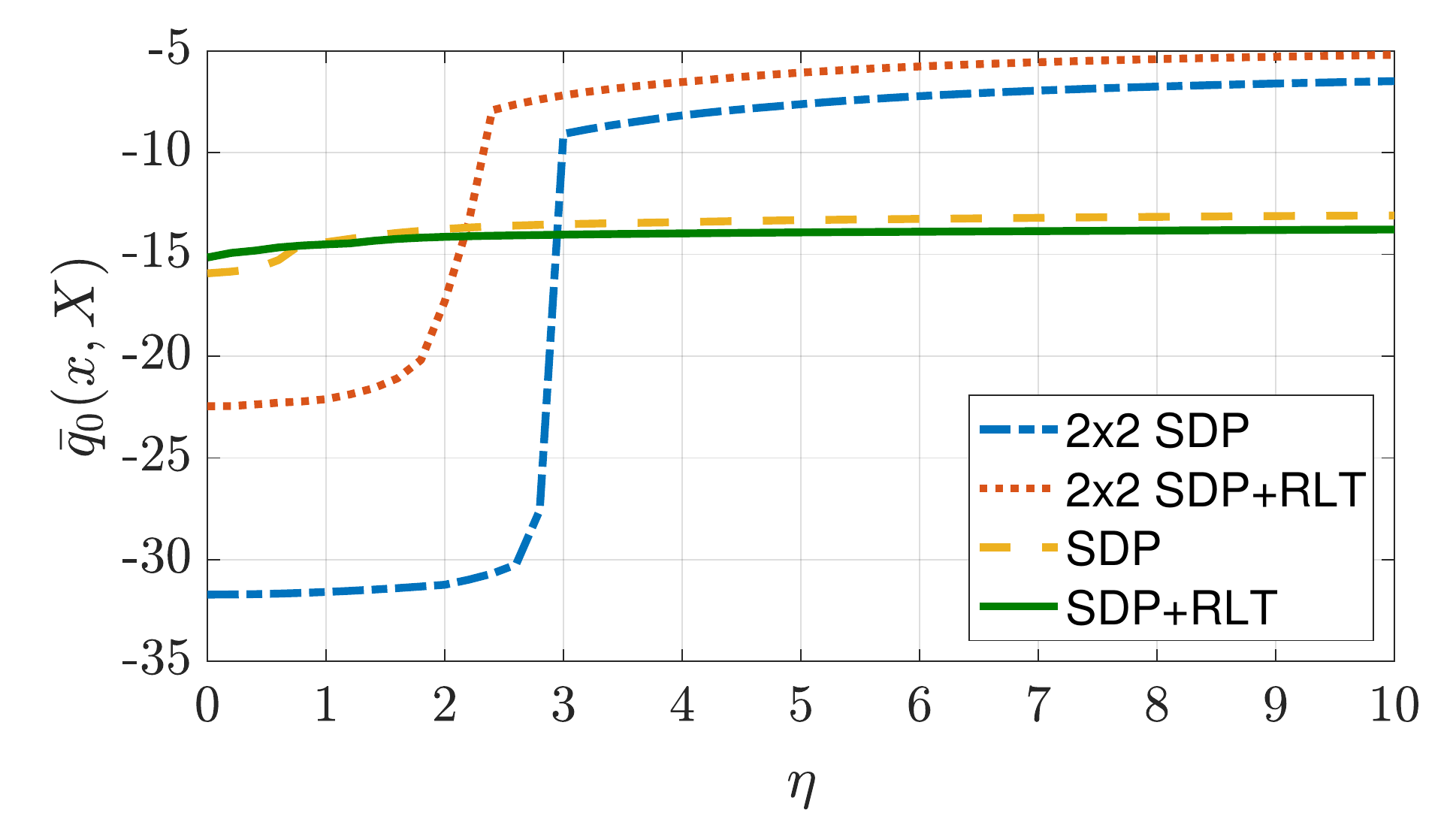}\\
	\includegraphics[width=0.5\textwidth]{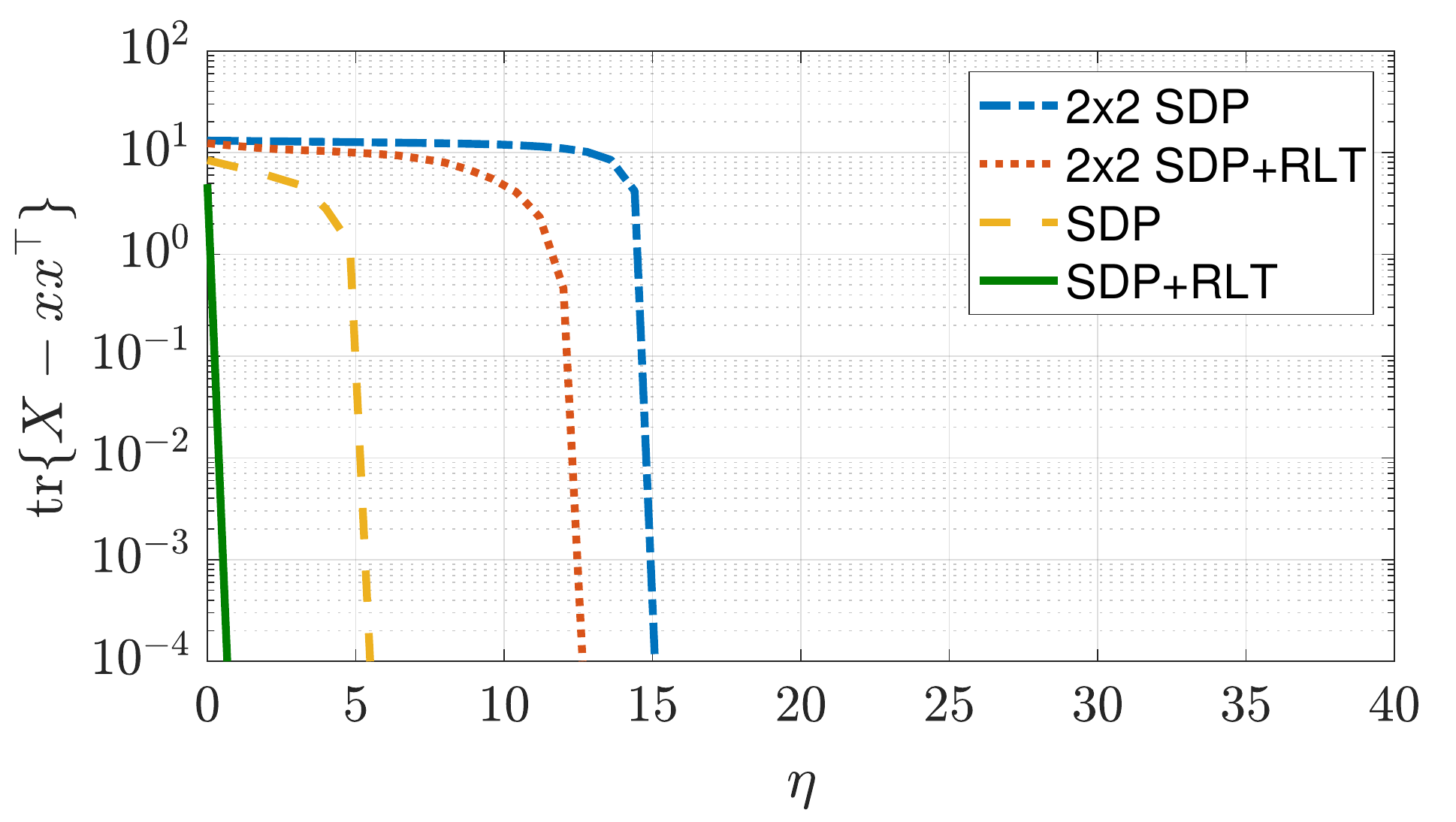}\!\!
	\includegraphics[width=0.5\textwidth]{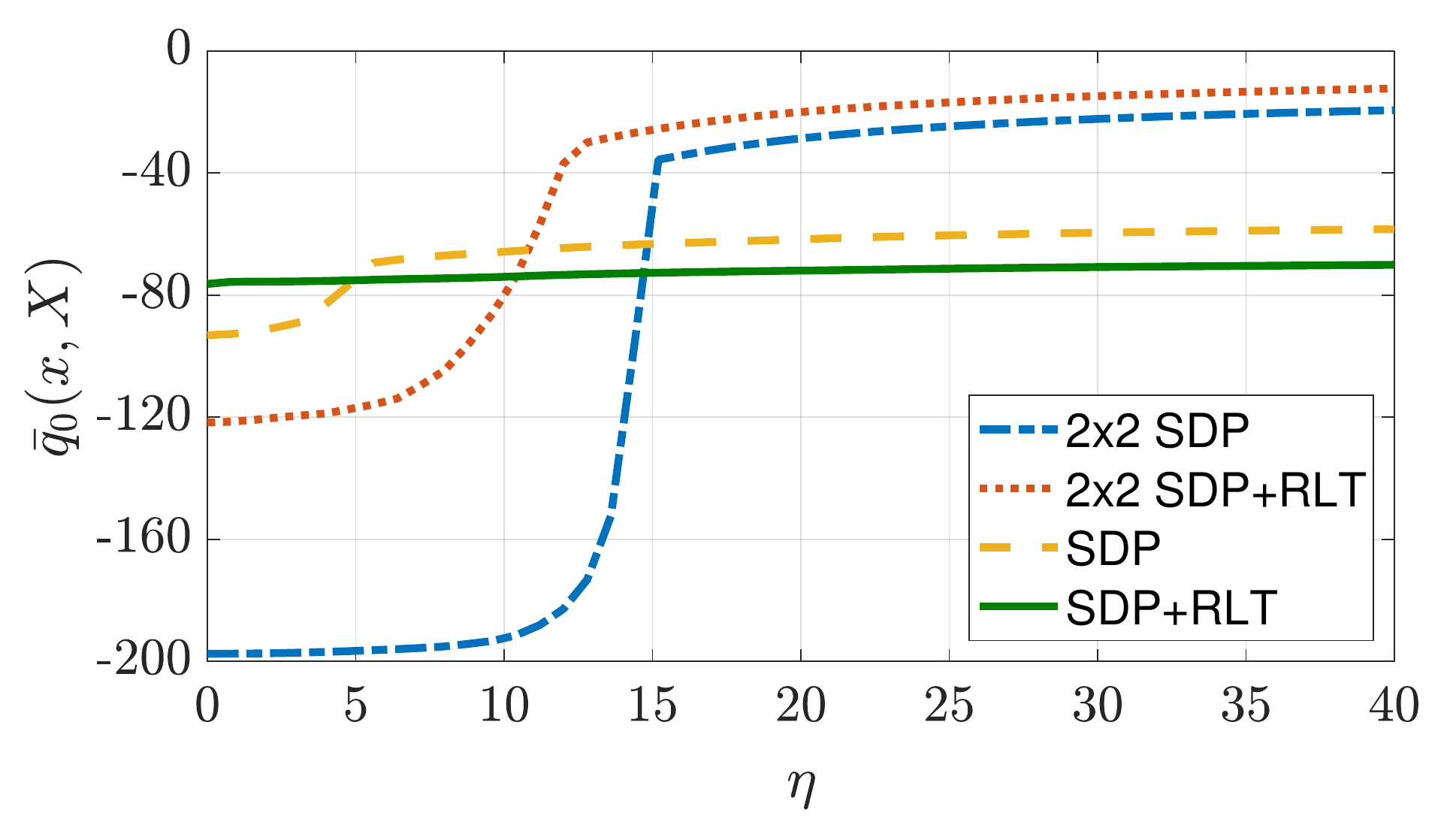}\\
	\includegraphics[width=0.5\textwidth]{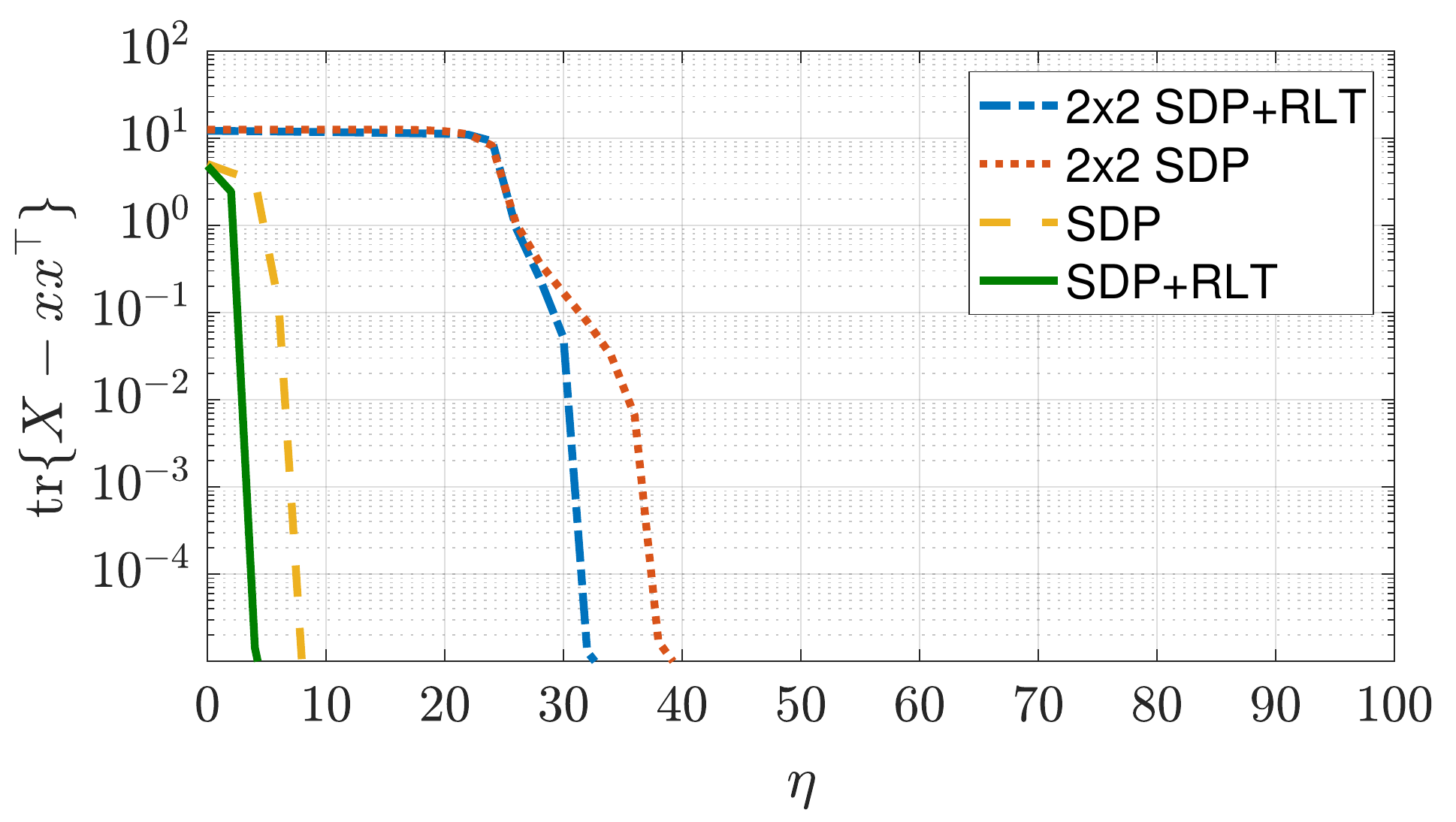}\!\!
	\includegraphics[width=0.5\textwidth]{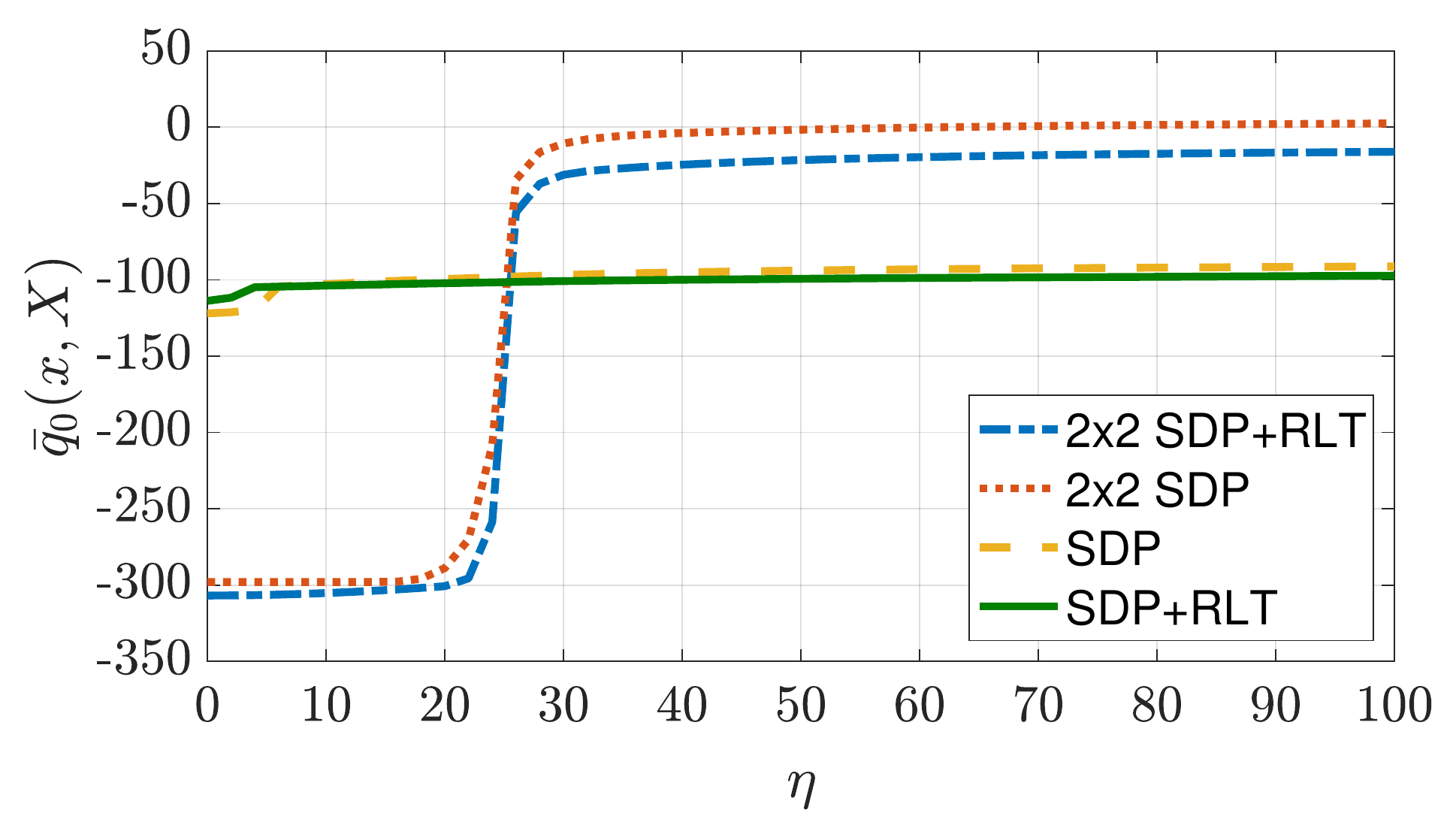}
	\caption{\small The effect of $\eta$ on the performance of penalized $2\times 2$ SDP, $2\times 2$ SDP+RLT, SDP, and SDP+RLT for cases QPLIB 1423, 1675, and 1967.}
	\label{plt:eta}
\end{figure}

\subsection{Large-scale system identification problems}


{
Following \cite{fattahi2018data}, this case study is concerned with the problem of identifying the parameters of linear dynamical systems given limited observations and non-uniform snapshots of state vectors. 
Optimization is an important tool for problems involving dynamical systems such as the identification of transfer functions and control synthesis \cite{1_rotkowitz2005characterization, 2_wang2019system,3_wang2018separable,4_kheirandishfard2018convex,5_fattahi2018transformation}.
One of these computationally-hard problems is system identification based solely on data (without intrusive means) which has been widely studied in the literature of control \cite{6_sarkar2018fast,7_pereira2010learning}. 
In this case study, we cast system identification as a non-convex QCQP and evaluate the ability of the proposed penalized SDP in solving very large scale instances of this problem.
}

Consider a discrete-time linear system described by the system of equations:
\begin{subequations}\label{sys}
	\begin{align}
	\boldsymbol{z}[\tau+1]&=\boldsymbol{A}\boldsymbol{z}[\tau]+\boldsymbol{B}\boldsymbol{u}[\tau]+\boldsymbol{w}[\tau] && \tau= 1,2,\ldots,T-1
	\end{align}
\end{subequations}
where 
\begin{itemize}
	\item $\{\boldsymbol{z}[\tau]\in\mathbb{R}^n\}^{T}_{\tau=1}$ are the state vectors that are known at times $\tau\in\{\tau_1,\ldots,\tau_o\}$,
	\item $\{\boldsymbol{u}[\tau]\in\mathbb{R}^m\}^{T}_{\tau=1}$ are the known control command vectors.
	\item $\boldsymbol{A}\in\mathbb{R}^{n\times n}$ and 
	$\boldsymbol{B}\in\mathbb{R}^{n\times m}$ are fixed unknown matrices, and
	\item $\{\boldsymbol{w}[\tau]\in\mathbb{R}^n\}^{T}_{\tau=1}$ account for the unknown disturbance vectors.
\end{itemize}
Our goal is to estimate the pair of ground truth matrices $(\bar{\boldsymbol{A}},\bar{\boldsymbol{B}})$, given a sample trajectory of the control commands $\{\bar{\boldsymbol{u}}[\tau]\in\mathbb{R}^n\}^{T}_{\tau=1}$ and the incomplete state vectors $\{\bar{\boldsymbol{z}}[\tau]\in\mathbb{R}^n\}_{\tau\in\{\tau_1,\ldots,\tau_o\}}$. To this end, we employ the minimum least absolute value estimator which amounts to the following QCQP:
\begin{subequations}\begin{align}
	& \hspace{-5mm}\underset{\begin{subarray}{l} 
		\hspace{-7mm}\phantom{\boldsymbol{B}\boldsymbol{A}\boldsymbol{z}}\{\boldsymbol{y}[\tau]\in\mathbb{R}^n\}^{T-1}_{\tau=1} \\
		\hspace{-7mm}\phantom{\boldsymbol{A}\boldsymbol{B}\boldsymbol{y}}\{\boldsymbol{z}[\tau]\in\mathbb{R}^n\}^{T}_{\tau=1} \\ 
		\hspace{-7mm}\phantom{\boldsymbol{B}\boldsymbol{y}\{\boldsymbol{z}[\tau]}\boldsymbol{A}\in\mathbb{R}^{n\times n}\\
		\hspace{-7mm}\phantom{\boldsymbol{A}\boldsymbol{y}\{\boldsymbol{z}[\tau]}\boldsymbol{B}\in\mathbb{R}^{n\times m}
		\end{subarray}}{\text{minimize}}
	& & \hspace{-5mm}\sum^{T-1}_{\tau=1}{\boldsymbol{1}^{\top}_n\boldsymbol{y}[\tau]} & & & & \hspace{-15mm}\label{lav_obj}\\
	& \hspace{-5mm}\text{~~~subject to}
	& & \hspace{-5mm}
	\boldsymbol{y}[\tau]\geq
	+\boldsymbol{z}[\tau+1]-\boldsymbol{A}\boldsymbol{z}[\tau]-\boldsymbol{B}\bar{\boldsymbol{u}}[\tau] 
	& & & & \hspace{-10mm}\tau\in\{1,2,\ldots,T-1\}\hspace{-5mm} \label{lav_1},\\
	& \hspace{-5mm}
	& & \hspace{-5mm}
	\boldsymbol{y}[\tau]\geq-\boldsymbol{z}[\tau+1]+\boldsymbol{A}\boldsymbol{z}[\tau]+\boldsymbol{B}\bar{\boldsymbol{u}}[\tau]
	& & & & \hspace{-10mm}\tau\in\{1,2,\ldots,T-1\} \hspace{-5mm}\label{lav_2},\\
	& \hspace{-5mm}
	& & \hspace{-5mm}\boldsymbol{z}[\tau] = \bar{\boldsymbol{z}}[\tau]
	& & & & \hspace{-10mm}\tau\in\{\tau_1,\ldots,\tau_o\}.\hspace{-5mm} \label{lav_3}
	\end{align}\end{subequations}
For every $\tau\in\{1,2,\ldots,T-1\}$, the auxiliary variable $\boldsymbol{y}[\tau]\in\mathbb{R}^n$ accounts for 
$|\boldsymbol{z}[\tau+1]-\boldsymbol{A}\boldsymbol{z}[\tau]-\boldsymbol{B}\bar{\boldsymbol{u}}[\tau]|$. This relation is imposed through the pair of constraints \eqref{lav_1} and \eqref{lav_2}.

The problem \eqref{lav_obj}--\eqref{lav_3}, can be cast in the form of \eqref{prob_obj}-\eqref{prob_eq}, with respect to the vector
\begin{align}
\boldsymbol{x}\triangleq[
\boldsymbol{z}[1]^{\top},\,\ldots,\,\boldsymbol{z}[T]^{\top},
\mathrm{vec}\{\boldsymbol{A}\}^{\top},
\alpha\,\boldsymbol{y}[1]^{\top},\,\ldots,\,\alpha\,\boldsymbol{y}[T-1]^{\top},
\alpha\,\mathrm{vec}\{\boldsymbol{B}\}^{\top}],\!\!\!
\end{align}
where $\alpha$ is a preconditioning constant. To solve the resulting problem, we use the sequential Algorithm \ref{al:alg_1} equipped with the $2\times 2$ SDP relaxation and the initial point $\hat{\boldsymbol{x}}=\boldsymbol{0}$.

We consider system identification problems with $n=25$, $m=20$, $T=500$ and $o=400$. In every experiment, $\{\tau_1,\ldots,\tau_o\}$ is a uniformly selected subset of $\{1,2,\ldots,T\}$. The resulting QCQP variable $\boldsymbol{x}$ is $23605$-dimensional and the problem is $16100$-dimensional if we exclude the known state vectors $\{\bar{\boldsymbol{z}}[\tau]\in\mathbb{R}^n\}_{\tau\in\{\tau_1,\ldots,\tau_o\}}$. Due to sparsity of the QCQP \eqref{lav_obj}-\eqref{lav_3} each round of the penalized $2\times 2$ SDP is solved within 30 minutes, by omitting the elements of the lifted variable $\boldsymbol{X}$ that do not appear in the objective and constraints. All of the convex programs are solved using MOSEK v8.1 \cite{mosek} through MATLAB 2017a and on a desktop computer with a 12-core 3.0GHz CPU and 256GB RAM. Due to the sheer size of this problem, we were only able to solve instances with $T\leq 70$ using BARON and COUENNE non of which resulted in successful recovery of the unknown matrices due to limited data points.

\begin{figure}[h!]
	\captionsetup[subfigure]{position=b}
	\centering
	\includegraphics[width=0.6\textwidth]{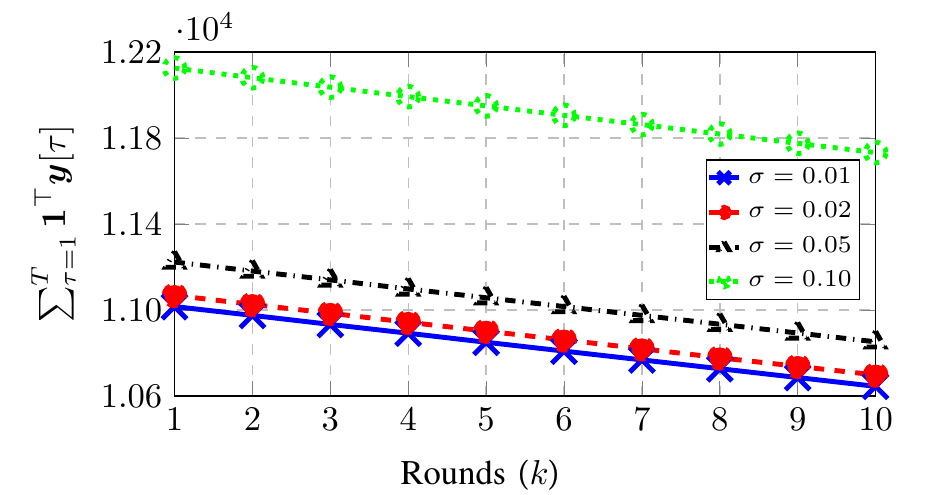}
	\caption{\small Convergence of the sequential penalized $2\times 2$ SDP for large-scale system identification with different disturbance levels.}
	\label{sys_id}
\end{figure}

The ground truth values are chosen as follows:
\begin{itemize}
	\item The elements of $\bar{\boldsymbol{A}}\in\mathbb{R}^{25\times 25}$ have zero-mean Gaussian distribution and the matrix is scaled in such a way that the largest singular value is equal to $0.5$.
	\item Every element of $\bar{\boldsymbol{B}}\in\mathbb{R}^{25\times 20}$, $\{\bar{\boldsymbol{u}}[\tau]\in\mathbb{R}^{20}\}^{T}_{\tau=1}$ and $\bar{\boldsymbol{z}}[1]\in\mathbb{R}^{25}$ have standard normal distribution.
	\item The elements of $\{\bar{\boldsymbol{w}}[\tau]\in\mathbb{R}^{25}\}^{T-1}_{\tau=1}$ have independent zero-mean Gaussian distribution with the standard deviation $\sigma\in\{0.01,0.02,0.05,0.10\}$.
\end{itemize}
For each experiment, we ran Algorithm \ref{al:alg_1} for 10 rounds. The preconditioning and penalty terms are set to $\alpha=10^{-3}$ and $\eta=40$, respectively. 
For each $\sigma\in\{0.01,0.02,$ $0.05,0.10\}$, we have run 10 random experiments resulting in the average recovery errors $0.0005$, $0.0010$, $0.0026$, and $0.0062$, respectively, for 
$\frac{1}{n}\|\bar{\boldsymbol{A}}-\boldsymbol{A}^{(10)}\|_F$, and the average errors $0.0014$, $0.0028$, $0.0070$, and $0.0141$, respectively, for $(mn)^{-\frac{1}{2}}\|\bar{\boldsymbol{B}}-\boldsymbol{B}^{(10)}\|_F$. 
In all of the trials, a feasible point is obtained in the first round of Algorithm \ref{al:alg_1}.
Figure \ref{sys_id} illustrates the convergence behavior of the objective functions for one of the trials for each disturbance level.



\section{Conclusions} \label{sec:conc}

This paper introduces a penalization approach for constructing feasible and near-optimal solutions to non-convex quadratically-constrained quadratic programming (QCQP) problems.
Given an arbitrary initial point (feasible or infeasible) for the original QCQP, penalized semidefinite programs are formulated by adding a linear term to the objective. A generalized linear independence constraint qualification (LICQ) condition is introduced as a regularity criterion for initial points, and it is shown that the solution of penalized {SDP} is feasible for QCQP if the initial point is regular and close to the feasible set. We show that the proposed penalized {SDPs} can be solved sequentially in order to improve the objective of the feasible solution. Numerical experiments on QPLIB benchmark cases demonstrate that the proposed sequential approach compares favorably with non-convex optimizers BARON and COUENNE. Moreover, the scalability of the proposed method is demonstrated on large-scale system identification problems.


\begin{acknowledgements}
The authors are grateful to GAMS Development Corporation for providing them with unrestricted access to a full set of solvers throughout the project.
\end{acknowledgements}

%
%

\bibliographystyle{spbasic}      
\bibliography{references}   


\appendix
\section{Application to polynomial optimization}\label{apo}
In this section, we show that the proposed penalized {SDP} approach can be used for polynomial optimization as well.
A polynomial optimization problem is formulated as
\begin{subequations}\begin{align}
	\underset{\boldsymbol{x}\in\mathbb{R}^n}{\text{minimize}} \ \ 
	& u_0(\boldsymbol{x}) & \label{poly_obj}\\
	\ \ \ \text{s.t.} \ \ \ 
	&  u_k(\boldsymbol{x})\leq0, \quad \quad  k\in\mathcal{I}\label{poly_ineq}\\
	& 
	u_k(\boldsymbol{x})=0, \quad \quad k\in\mathcal{E},\label{poly_eq}
	\end{align}\end{subequations}
for every $k\in\{0\}\cup\mathcal{I}\cup\mathcal{E}$, where each function $u_k:\mathbb{R}^n\to\mathbb{R}$ is a polynomial of arbitrary degree. 
Problem \eqref{poly_obj}--\eqref{poly_eq} can be reformulated as a QCQP of the form:
\begin{subequations}\begin{align}
	\underset{\begin{subarray}{l} \boldsymbol{x}\in\mathbb{R}^n, 
		{\boldsymbol{y}}\in\mathbb{R}^o\end{subarray}}{\!\!\text{minimize}} \ \ 
	&  \!\!\!\!\phantom{v_i} w_0(\boldsymbol{x},\boldsymbol{y}) & \label{qcqp_obj}\\
	\ \ \ \text{s.t.} \ \ \ 
	&  \!\!\!\!\phantom{v_i} w_k(\boldsymbol{x},\boldsymbol{y})\leq0,  \quad \quad k\in\mathcal{I}\label{qcqp_ineq}\\
	& 
	\!\!\!\!\phantom{v_i} w_k(\boldsymbol{x},\boldsymbol{y})=0,  \quad \quad k\in\mathcal{E}\label{qcqp_eq}\\
	& 
	\!\!\!\!\phantom{w_k} v_{i}(\boldsymbol{x},\boldsymbol{y})=0 , \quad \quad i\in\mathcal{O}, \label{qcqp_aux}
	\end{align}\end{subequations}
where $\boldsymbol{y}\in\mathbb{R}^{|\mathcal{O}|}$ is an auxiliary variable, and
$v_1,\ldots,v_{|\mathcal{O}|}$ and $w_0,w_1,\ldots,w_{|\{0\}\cup\mathcal{I}\cup\mathcal{E}|}$ are quadratic functions with the following properties:
\begin{itemize}
	\item For every $\boldsymbol{x}\in\mathbb{R}^n$, the function 
	$\boldsymbol{v}(\boldsymbol{x},\cdot):\mathbb{R}^{|\mathcal{O}|}\to\mathbb{R}^{|\mathcal{O}|}$ is invertible,
	\item If $\boldsymbol{v}(\boldsymbol{x},\boldsymbol{y})=\boldsymbol{0}_n$, then
	$w_k(\boldsymbol{x},\boldsymbol{y}) = u_k(\boldsymbol{x})$ for every 
	$k\in\{0\}\cup\mathcal{I}\cup\mathcal{E}\nonumber$.
\end{itemize}
Based on the above properties, there is a one-to-one correspondence between the feasible sets of \eqref{poly_obj}--\eqref{poly_eq} and \eqref{qcqp_obj}--\eqref{qcqp_aux}. Moreover, a feasible point $(\accentset{\ast}{\boldsymbol{x}},\accentset{\ast}{\boldsymbol{y}})$ is an optimal solution to the QCQP \eqref{qcqp_obj}--\eqref{qcqp_aux} if and only if $\accentset{\ast}{\boldsymbol{x}}$ is an optimal solution to the polynomial optimization problem \eqref{poly_obj}--\eqref{poly_eq}. 

\begin{theorem}[\!\!\cite{madani2014rank}]\label{QCQP_size_bound}
	Suppose that $\{u_k\}_{k\in\{0\}\cup\mathcal{I}\cup\mathcal{E}}$ are polynomials of degree at most $d$, consisting of $m$ monomials in total. There exists a QCQP reformulation of the polynomial optimization \eqref{poly_obj}--\eqref{poly_eq} in the form of \eqref{qcqp_obj}--\eqref{qcqp_aux}, where
	${|\mathcal{O}|}\leq m
	n
	\left(\lfloor\log_2(d)\rfloor+1\right)$.
\end{theorem}

The next proposition shows that the LICQ regularity of a point $\hat{\boldsymbol{x}}\in\mathbb{R}^n$ is inherited by the corresponding point $(\hat{\boldsymbol{x}},\hat{\boldsymbol{y}})\in\mathbb{R}^n\times\mathbb{R}^o$ of the QCQP reformulation \eqref{qcqp_obj}-\eqref{qcqp_aux}.

\begin{proposition}\label{thm4}
	Consider a pair of vectors $\hat{\boldsymbol{x}}\in\mathbb{R}^n$ and $\hat{\boldsymbol{y}}\in\mathbb{R}^{|\mathcal{O}|}$ satisfying
	$\boldsymbol{v}(\hat{\boldsymbol{x}},\hat{\boldsymbol{y}})=\boldsymbol{0}_n$. The following two statements are equivalent:
	\begin{enumerate}
		\item $\hat{\boldsymbol{x}}$ is feasible and satisfies the LICQ condition for the polynomial optimization problem \eqref{poly_obj}--\eqref{poly_ineq}. 
		\item $(\hat{\boldsymbol{x}},\hat{\boldsymbol{y}})$ is feasible and satisfies the LICQ condition for the QCQP \eqref{qcqp_obj}--\eqref{qcqp_aux}.
	\end{enumerate}
\end{proposition}

\begin{proof}
	From
	$\boldsymbol{u}(\hat{\boldsymbol{x}})=\boldsymbol{w}(\hat{\boldsymbol{x}},\hat{\boldsymbol{y}})$
	and the invertiblity assumption for
	$\boldsymbol{v}(\hat{\boldsymbol{x}},\cdot)$, we have
	\begin{align}
	\!\!\!\!\!\frac{\partial\boldsymbol{u}(\hat{\boldsymbol{x}})}{\partial\boldsymbol{x}} &\!=\!\!
	\begin{bmatrix}
	\frac{\partial\boldsymbol{w}(\hat{\boldsymbol{x}},\hat{\boldsymbol{y}})}{\partial\boldsymbol{x}} &
	\frac{\partial\boldsymbol{w}(\hat{\boldsymbol{x}},\hat{\boldsymbol{y}})}{\partial\boldsymbol{y}}
	\end{bmatrix}\!\!
	\begin{bmatrix}
	\boldsymbol{I} &
	-\!\left(\!\frac{\partial\boldsymbol{v}(\hat{\boldsymbol{x}},\hat{\boldsymbol{y}})}{\partial\boldsymbol{y}}\!\right)^{\!\!-1}
	\frac{\partial\boldsymbol{v}(\hat{\boldsymbol{x}},\hat{\boldsymbol{y}})}{\partial\boldsymbol{x}}\!
	\end{bmatrix}^{\!\top}\nonumber\\
	&=\!\frac{\partial\boldsymbol{w}(\hat{\boldsymbol{x}},\hat{\boldsymbol{y}})}{\partial\boldsymbol{x}}
	\!-\!\frac{\partial\boldsymbol{w}(\hat{\boldsymbol{x}},\hat{\boldsymbol{y}})}{\partial\boldsymbol{y}}\!
	\left(\!\frac{\partial\boldsymbol{v}(\hat{\boldsymbol{x}},\hat{\boldsymbol{y}})}{\partial\boldsymbol{y}}\!\right)^{\!\!-1}
	\!
	\frac{\partial\boldsymbol{v}(\hat{\boldsymbol{x}},\hat{\boldsymbol{y}})}{\partial\boldsymbol{x}}.\!\!\!
	\end{align}
	Therefore,
	$\mathcal{J}_{\mathrm{PO}}(\hat{\boldsymbol{x}})=\frac{\partial\boldsymbol{u}(\hat{\boldsymbol{x}})}{\partial\boldsymbol{x}}$
	is equal to the Schur complement of
	\begin{align}	
	\mathcal{J}_{\mathrm{QCQP}}(\hat{\boldsymbol{x}},\hat{\boldsymbol{y}})=
	\begin{bmatrix}
	\frac{\partial\boldsymbol{w}(\hat{\boldsymbol{x}},\hat{\boldsymbol{y}})}{\partial\boldsymbol{x}} &
	\frac{\partial\boldsymbol{w}(\hat{\boldsymbol{x}},\hat{\boldsymbol{y}})}{\partial\boldsymbol{y}}\\
	\frac{\partial\boldsymbol{v}(\hat{\boldsymbol{x}},\hat{\boldsymbol{y}})}{\partial\boldsymbol{x}} &
	\frac{\partial\boldsymbol{v}(\hat{\boldsymbol{x}},\hat{\boldsymbol{y}})}{\partial\boldsymbol{y}}\\
	\end{bmatrix},
	\end{align}
	which is the Jacobian matrix of the QCQP  \eqref{qcqp_obj}--\eqref{qcqp_aux} at the point $(\hat{\boldsymbol{x}},\hat{\boldsymbol{y}})$.
	As a result, the matrix $\mathcal{J}_{\mathrm{PO}}(\hat{\boldsymbol{x}})$ is singular if and only if $\mathcal{J}_{\mathrm{QCQP}}(\hat{\boldsymbol{x}},\hat{\boldsymbol{y}})$ is singular.
\end{proof}

\section{Reformulation-Linearization Technique}\label{apn:RLT}
This appendix covers the reformulation-linearization technique (RLT) of Sherali and Adams \cite{sherali2013reformulation} as an approach to strengthen convex relaxations of the form \eqref{prob_lifted_obj}--\eqref{prob_lifted_conic} in the presence of affine constraints. Define $\mathcal{L}$ as the set of affine constrains in the QCQP \eqref{prob_obj}--\eqref{prob_eq}, i.e., 
$\mathcal{L}\triangleq\{k\in\mathcal{I}\cup\mathcal{E}\;|\;\boldsymbol{A}_k=\boldsymbol{0}_{n\times n}\}$. Define also
\begin{subequations}
	\begin{align}
	\boldsymbol{H}&\triangleq 
	[\!\!\phantom{\boldsymbol{c}}\boldsymbol{B}\{\mathcal{L}\cap\mathcal{I}\}^{\top},
	\!\!\phantom{\boldsymbol{c}}\boldsymbol{B}\{\mathcal{L}\cap\mathcal{E}\}^{\top},
	-\!\!\phantom{\boldsymbol{c}}\boldsymbol{B}\{\mathcal{L}\cap\mathcal{E}\}^{\top}]^{\top},\\
	\boldsymbol{h}&\triangleq 
	[\!\!\!\phantom{\boldsymbol{B}}\boldsymbol{c}\{\mathcal{L}\cap\mathcal{I}\}^{\top}\,,
	\!\!\!\phantom{\boldsymbol{B}}\boldsymbol{c}\{\mathcal{L}\cap\mathcal{E}\}^{\top}\,,
	-\!\!\!\phantom{\boldsymbol{B}}\boldsymbol{c}\{\mathcal{L}\cap\mathcal{E}\}^{\top}\,]^{\top},
	\end{align}
\end{subequations}
where
$\boldsymbol{B}\triangleq 
[\boldsymbol{b}_1,\ldots,\boldsymbol{b}_{|\mathcal{I}\cap\mathcal{E}|}]^{\top}$
and 
$\boldsymbol{c}\triangleq 
[c_1,\ldots,c_{|\mathcal{I}\cap\mathcal{E}|}]^{\top}$. Every $\boldsymbol{x}\in\mathcal{F}$ satisfies
\begin{align}
\boldsymbol{H}\boldsymbol{x} + \boldsymbol{h} \leq 0,
\end{align}
and, as a result, all elements of the matrix  
\begin{align}
\boldsymbol{H}\boldsymbol{x}\boldsymbol{x}^{\!\top}\!\boldsymbol{H}^{\top}+
\boldsymbol{h}\boldsymbol{x}^{\!\top}\!\boldsymbol{H}^{\!\top}+
\boldsymbol{H}\boldsymbol{x}\boldsymbol{h}^{\!\top}+
\boldsymbol{h}\boldsymbol{h}^{\!\top}
\end{align}
are non-negative if $\boldsymbol{x}$ is feasible. Hence, the inequality
\begin{align}
\boldsymbol{e}^{\top}_i\boldsymbol{V}(\boldsymbol{x},\boldsymbol{x}\boldsymbol{x}^{\top})\boldsymbol{e}_j\geq 0
\end{align}
holds true for every $\boldsymbol{x}\in\mathcal{F}$ and $(i,j)\in\mathcal{H}\times\mathcal{H}$, where $\boldsymbol{V}:\mathbb{R}^n\times\mathbb{S}_n\to\mathbb{S}_{|\mathcal{H}|}$ is defined as
\begin{align}
\boldsymbol{V}(\boldsymbol{x},\boldsymbol{X})\triangleq\boldsymbol{H}\boldsymbol{X}\boldsymbol{H}^{\top}+
\boldsymbol{h}\boldsymbol{x}^{\!\top}\!\boldsymbol{H}^{\top}+
\boldsymbol{H}\boldsymbol{x}\boldsymbol{h}^{\!\top}+
\boldsymbol{h}\boldsymbol{h}^{\!\top},
\end{align}
$\mathcal{H}\triangleq\{1,\ldots,|\mathcal{L}\cap\mathcal{I}|+2|\mathcal{L}\cap\mathcal{E}|\}$, and $\boldsymbol{e}_1,\ldots,\boldsymbol{e}_{|\mathcal{H}|}$ denote the standard bases in $\mathbb{R}^{|\mathcal{H}|}$.

This leads to a strengthened relaxation of QCQP \eqref{prob_obj}--\eqref{prob_eq}:
\begin{subequations}\begin{align}
	\underset{\begin{subarray}{l} \boldsymbol{x}\in\mathbb{R}^n, 
		\boldsymbol{X}\in\mathbb{S}_n\end{subarray}}{\text{minimize}} \ \ &
	\bar{q}_0(\boldsymbol{x},\boldsymbol{X})  \label{prob_rtl_lifted_obj} \\
	\ \ \ \ \ \ \text{s.t.} \ \ \ \ 
	&  \bar{q}_k(\boldsymbol{x},\boldsymbol{X}) \leq 0, \qquad\qquad k\in\mathcal{I}\label{prob_rtl_lifted_ineq} \\ 
	&  \bar{q}_k(\boldsymbol{x},\boldsymbol{X}) = 0,  \qquad\qquad k\in\mathcal{E}\label{prob_rtl_lifted_eq} \\
	&  \boldsymbol{X}-\boldsymbol{x}\boldsymbol{x}^\top\succeq_{\mathcal{C}_r}0  \label{prob_rtl_lifted_conic}\\
	& \boldsymbol{e}^{\top}_i\boldsymbol{V}(\boldsymbol{x},\boldsymbol{X})\boldsymbol{e}_j\geq 0,
	\qquad\ (i,j)\in\mathcal{V} \label{prob_rtl_lifted_valid}
	\end{align}\end{subequations}
where $\mathcal{V}\subseteq\mathcal{H}\times\mathcal{H}$ is a selection of RLT inequalities.


\end{document}